\documentclass{amsart}

\usepackage{amssymb,amsmath,amsfonts,amsthm,epsfig,amscd}
\usepackage{stmaryrd}
\usepackage[all,cmtip,poly]{xy}
\usepackage{svn}
\usepackage{mathtools}
\usepackage{mathrsfs}
\usepackage{tikz-cd}
\usepackage{hyperref}
\usepackage{nameref}
\usepackage{dsfont}
\usepackage{stackrel}

\newtheorem{lem}{Lemma}[section]

\newtheorem{df}[lem]{Definition}

\newtheorem{cor}[lem]{Corollary}
\newtheorem{thm}[lem]{Theorem}

\newtheorem{prop}[lem]{Proposition}

\theoremstyle{remark}
\newtheorem{rem}[lem]{Remark}

\newcommand{\GL}{\mathrm{GL}}
\newcommand{\sO}{\ensuremath{\mathscr{O}}}

\newcommand{\sW}{\ensuremath{\mathscr{W}}}

\newcommand{\Q}{\QQ}

\newcommand{\Nbar}{\overline{N}}

\renewcommand{\AA}{{\mathbb A}}

\def\id{\mathrm{id}}
\def\Ind{\mathrm{Ind}}

\def\Hom{\mathrm{Hom}}

\def\for{\mathrm{for}}
\def\Lie{\mathrm{Lie}}

\def\Sp{\mathrm{Sp}}

\def\Res{\mathrm{Res}}

\newcommand{\CC}{{\mathbb C}}
\newcommand{\DD}{{\mathbb D}}

\newcommand{\FF}{{\mathbb F}}
\newcommand{\GG}{{\mathbb G}}

\newcommand{\QQ}{{\mathbb Q}}
\newcommand{\RR}{{\mathbb R}}

\newcommand{\ZZ}{{\mathbb Z}}

\newcommand{\cB}{{\mathcal B}}

\newcommand{\cG}{{\mathcal G}}
\newcommand{\cH}{{\mathcal H}}

\newcommand{\cL}{{\mathcal L}}
\newcommand{\cM}{{\mathcal M}}
\newcommand{\cN}{{\mathcal N}}
\newcommand{\cO}{{\mathcal O}}
\newcommand{\cP}{{\mathcal P}}

\newcommand{\cT}{{\mathcal T}}

\newcommand{\frakl}{\mathfrak{l}}

\newcommand{\frakn}{\mathfrak{n}}
\newcommand{\fraku}{\mathfrak{u}}

\newcommand{\Qbar}{\overline{\Q}}

\newcommand{\Qpbar}{\Qbar_p}

\newcommand{\wtK}{\widetilde{K}}

\makeatletter
\tikzset{
  column sep/.code=\def\pgfmatrixcolumnsep{\pgf@matrix@xscale*(#1)},
  row sep/.code   =\def\pgfmatrixrowsep{\pgf@matrix@yscale*(#1)},
  matrix xscale/.code=%
    \pgfmathsetmacro\pgf@matrix@xscale{\pgf@matrix@xscale*(#1)},
  matrix yscale/.code=%
    \pgfmathsetmacro\pgf@matrix@yscale{\pgf@matrix@yscale*(#1)},
  matrix scale/.style={/tikz/matrix xscale={#1},/tikz/matrix yscale={#1}}}
\def\pgf@matrix@xscale{1}
\def\pgf@matrix@yscale{1}
\makeatother 

\DeclareMathOperator{\supp}{supp}

\title{The cohomology of $p$-adic distribution representations}
\begin{document}
\author{Weibo Fu}
\address{Department of Mathematics, Princeton University, Princeton, NJ, USA.}
\email{wfu@math.princeton.edu}
\date{March,~2022}

\maketitle
\begin{abstract}
We give a generalization of Kostant's theorem on Lie algebra cohomology of finite dimensional highest weight representations to some infinite dimensional cases over a $p$-adic family of highest weight distribution representations.
For proving this, we develop a theory of eigen orthonormalizable Banach representations of $p$-adic torus over an affinoid algebra, and we construct eigen orthonormalizable weight completions of the distribution representations.
\end{abstract}
\tableofcontents

\section{Introduction}\label{intro}
Let $p$ be a prime number, let $G$ be either the general linear group $\GL_{2n}$ or the symplectic group $\Sp_{2n}$ or a unitary group over $\ZZ_p$ with a chosen Borel pair $(B,T)$. 
We use $I$ to denote the Iwahori subgroup of $G(\ZZ_p)$ associated to $(B, T)$, consisting of matrices congruent to $B(\FF_p)$ modulo $p$.

Let $N$ be a unipotent subgroup of $G$ with Lie algebra $\fraku$ of $N(\ZZ_p)$.
It corresponds to a Levi decomposition $P = LN$ for a parabolic subgroup $P$ and Levi subgroup $L$ of $G$.
Let $I_L := I \cap L(\ZZ_p)$ be the Iwahori subgroup of $L(\ZZ_p)$ for the Levi $L$.
For an algebraic representation $V$ of $\GL_n(\ZZ_p)$ and $\frakn := \Lie(N(\ZZ_p))$, the Lie algebra cohomology $H^\ast(\frakn, V)$ has a representation structure of $L(\ZZ_p)$.
Working over the complex numbers, Kostant's work on Lie algebra cohomology \cite{Kos61} directly extends to this setting, and completely describes $H^\ast(\frakn, V)$ in terms of irreducible representations of $I_L$ (as Lie algebra representations in \cite{Kos61}) for any finite dimensional algebraic representation $V$ of $I \subset \GL_n(V)$.

Let $\sW$ be the weight space of $T(\ZZ_p)$ over $\QQ_p$, parametrizing locally analytic characters of $T(\ZZ_p)$. 
And let $\Omega$ be an affinoid subdomain of $\sW$, there is a universal character $\chi_{\Omega} : T(\ZZ_p) \to \sO(\Omega)^\times$.
There is a countable filtration $\{I^s\}$ of open normal subgroups for $I$ such that $I^s$ consists of matrices in $I$ congruent to the identity matrix modulo $p^s$.
Let $N_I$ be the $\ZZ_p$-points of unipotent radical of $B$ and $$\Ind^s_\Omega:=\{f: I \rightarrow \mathscr{O}(\Omega), f \text { analytic on each } I^s ~ \text{coset}, $$ $$ f(g t n)=\chi_{\Omega}(t) f(g) ~ \forall n \in N_I, t \in T(\ZZ_p), g \in I\},$$
for some big enough $s$ depending on $\Omega$.

$\Ind^s_{\Omega}$ is a Banach representation of $I$ over $\sO(\Omega)$.
Let $\DD^s_\Omega$ be the $\sO(\Omega)$-linear Banach dual of $\Ind^s_{\Omega}$.
The so called distribution representation $\DD^s_\Omega$ interpolates $p$-adic analogue of Verma modules, and therefore finite dimensional representations.
This distribution representation plays a significant role in $p$-adic automorphic forms such as in \cite{AS}, \cite{Urb11}, \cite{AIP15}, \cite{Han17}.

Suppose $P$ is a Siegel parabolic subgroup of $G$ containing $B$ and whose unipotent radical $N$ is abelian.
For example, if $G = \GL_{2n}$, $B$ is chosen to be the upper triangular Borel and $P$ is chosen to be the block upper triangular parabolic with zero lower left $n \times n$ block:
\[
\left[ 
\begin{array}{c|c} 
  \ast & \ast \\ 
  \hline 
  0 & \ast 
\end{array} 
\right].
\]
Let $w$ be a relative Weyl group element for $G$ with respect to $P$ of length $l(w)$.
Let $N_w := N(\ZZ_p) \cap wIw^{-1}$.
Let $\DD^s_{w,\Omega}$ be the representation of $wIw^{-1}$ obtained from $\DD^s_\Omega$ by conjugation of $w$ (\S \ref{la rep}).
By choosing a representative in the class, we will consider certain $w$ such that the Iwahori group $I_L$ for $L$ is contained in $wIw^{-1}$.
A useful example is that \[ w = \begin{pmatrix} 1 & 0 & 0 & 0 \\ 0 & 0 & 1 & 0 \\ 
0 & 1 & 0 & 0 \\ 0 & 0 & 0 & 1 \end{pmatrix} \in \GL_4(\ZZ_p), \\ \]
Iwahori subgroups are associated to the standard upper triangular Borels of both $\GL_2 \times \GL_2 \subset \GL_4$.
In the present article, we aim to analyze cohomology of $\DD^s_{w,\Omega}$ with respect to an open subgroup of $N(\ZZ_p)$ through the Lie algebra cohomology $H^\ast(\frakn, \DD^s_{w,\Omega})$. 
The cohomology $H^\ast(\frakn, \DD^s_{w,\Omega})$ admits a natural action of $I_L$.
We extract a $I_L$-direct summand of this cohomology in Thm \ref{pppp}, which interpolates distribution representation of the Levi subgroup.
\begin{thm}\label{intro 1}
There is a natural $I_L$-equivariant direct summand $$i : \DD_{w \cdot \Omega}^s \hookrightarrow H^{l(w)}(\frak{n}, \DD^s_{w,\Omega})^{N_w}.$$ 
Here $\DD_{w \cdot \Omega}^s$ is the $\sO(\Omega)$-linear Banach dual of
$$\Ind^s_{w\cdot\Omega}:=\{f: I_L \rightarrow \mathscr{O}(\Omega), f \text { analytic on each } I_L^s-\operatorname{coset}, $$ $$ f(g t n)=\chi_{\Omega}(w^{-1}tw)\cdot (w\delta-\delta)(t) f(g) ~ \forall n \in N_p, t \in T_p, g \in I_p\},$$ where $I_L^s$ is an open subgroup of $I_L$ depending on $s$ and $\delta$ is half sum of positive roots with respect to $(B, T)$.
\end{thm}

The direct summand $\DD_{w \cdot \Omega}^s$ is a distribution module of $I_L$. In particular, it interpolates locally algebraic representations of $I_L$.
Our result can be regarded as a generalization of Kostant's theorem in a $p$-adic family.
Here we give a sketch of proof in \S \ref{N coho}. 

Using Chevalley–Eilenberg complex $\bigwedge^\bullet \frakn^\ast \otimes \DD^s_{w,\Omega}$ to compute the Lie algebra cohomology, we explicitly construct a $I_L$-equivariant inclusion $\DD_{w \cdot \Omega}^s \hookrightarrow \bigwedge^{l(w)} \frakn^\ast \otimes \DD^s_{w,\Omega}$ using Mackey's tensor product theorem and Kostant's theorem.
We prove this map induces an inclusion $\DD_{w \cdot \Omega}^s \hookrightarrow H^{l(w)}(\frak{u}, \DD^s_{w,\Omega})^{N_w}$ when passing to cohomology.
For proving the splitting of the inclusion, we use the infinitesimal character method of \cite{CO}.
And we develop a theory of eigen orthonormalizable Banach representations of $T(\ZZ_p)$ (which plays the role of weights in classical representation theory over a field) in section \S \ref{Banach rep}, together with some explicit computations and $p$-adic analysis for $G$ in \S \ref{analysis}.

More specifically, we construct certain completions of these distribution representations.
\begin{thm}
There are continuous equivariant embeddings for some open subgroups of the corresponding Iwahori groups
\[ \DD^s_{w,\Omega} \hookrightarrow \DD^{s,1}_{w,\Omega}, ~ \DD_{w \cdot \Omega}^s \hookrightarrow \DD_{w \cdot \Omega}^{s,1} \]
such that the images are dense.
Moreover, these auxiliary representations are bounded eigen orthonormalizable (\S \ref{Banach rep}, Def \ref{cons}) for an open subgroup $T^s \subset T(\ZZ_p)$.
\end{thm}
The construction will be given in \S \ref{auxiliary modules}, and \S \ref{analysis} will prove $\DD^{s,1}_{w,\Omega}$, $\DD_{w \cdot \Omega}^{s,1}$ are actually group representations.
The eigen orthonormalizable notions introduced in \S \ref{Banach rep} are generalizations of the weight space decomposition in classical Lie theory.
And the notion of orthonormalizable Banach module was already well considered and studied by, for example, \cite{Col97} and \cite{Buz07}.

For proving Thm \ref{intro 1}, we also need to use a Zariski topology argument for the weight space.
We prove distribution constituent $\DD_{w \cdot \Omega}^s$ ``almost" equals to the whole cohomology $H^\ast(\frakn, \DD^s_{w,\Omega})^{N_w}$ when specializes to a generic weight in $\Omega$.
The statement of Thm \ref{intro 1} is the best possible in the sense that for any other cohomological degree $\ast \neq l(w)$, $H^\ast(\frakn, \DD^s_{w,\Omega})^{N_w}$ is ``almost" vanishing generically by considering the infinitesimal character on $H^\ast(\frakn, \DD^{s,1}_{w,\Omega})^{N_w}$. 

And since we are dealing with Iwahori groups over a general $p$-adic field, we also need to establish some results for the weight spaces of the torus, which we think may be of independent interests.

If $E$ is a finite extension of $\QQ_p$, $\sigma : E \hookrightarrow \Qpbar$, and $\Gamma$ is a compact profinite $E$-analytic abelian group, we use $\sW_\Gamma$ to denote the weight space parametrizing continuous/$\QQ_p$-locally analytic characters of $\Gamma$, 
$\sW_\Gamma^\sigma$ to denote the weight space parametrizing $(\sigma, E)$-analytic 
locally analytic characters of $\Gamma$.
\begin{thm}
There is a decomposition of the weight space $\sW_\Gamma$ in terms of $\sW_\Gamma^\sigma$:
$$\prod \sigma: \prod_{\sigma: E \hookrightarrow K} \mathscr{W}_{\Gamma}^{\sigma} \xrightarrow{m} \mathscr{W}_{\Gamma}$$ up to isogeny, i.e., $\prod \sigma$ is surjective with the kernel being a zero-dimensional Zariski closed subspace. Moreover, points of the kernel correspond to smooth characters.
\end{thm}

Our motivation for studying this problem is however derived from a global Langlands correspondence perspective. 

We refer to $X^H$ as the symmetric space for a connected linear algebraic group $H$ over $F$. For example, if $H = \GL_{n,F}$, one can take $X^H = \GL_n(F_\infty)/K_\infty \RR^\times$ for $K_\infty \subset \GL_n(F_\infty)$ a maximal compact subgroup.
If $F$ is CM with a totally real field $F^+$, $L:=\Res_{F/F^+} \GL_{n,F}$ appears as a Levi of a Siegel parabolic subgroup $P$ of a quasi-split unitary group $G:=U(n,n)_{F/F^+}$.
$X^{G}$ has complex dimension $d = [F^+:\QQ]n^2$.
Given a certain good level subgroup $\wtK \subset G(\AA_{F^+}^\infty)$ and $K = \wtK \cap G(\AA_{F^+}^\infty)$, the locally symmetric space $X_{\wtK}=G(F^+)\backslash G(\AA_{F^+}^\infty) \times X^{G} / \wtK$ admits the Borel-Serre compatification $\overline{X}_{\wtK}$ (\cite{BS73}) and the Borel-Serre boundary $\partial X_{\wtK}:=\overline{X}_{\wtK}-X_{\wtK}$.
There is a stratum $X^P_{\wtK} \subset \partial X_{\wtK}$ together with a torus fibration \[ X^P_{\wtK} \stackrel{\pi}{\twoheadrightarrow} X_K=L(F^+) \backslash L(\AA_{F^+}^\infty) \times X^L / K. \]
Given a representation $V$ of $\wtK$, there is a local system $\underline{V}$ on $X_{\wtK}$ associated to $V$.
Consider the local system on $X^P_{\wtK}$ obtained by composition of pushforward and restriction via $X^P_{\wtK} \hookrightarrow \overline{X}_{\wtK} \hookleftarrow X_{\wtK}$, which is still denoted as $\underline{V}$.
There is a Hecke-equivariant long exact sequence
\begin{eqnarray}\label{S1}
\ldots\to H_c^i(X_{\wtK},\underline{V})\to H^i(X_{\wtK},\underline{V})\to H^i(\partial X_{\wtK}, \underline{V})\to \ldots\ ,
\end{eqnarray}
and a Hecke-equivariant Leray spectral sequence 
\begin{eqnarray}\label{S2}
H^i(X_K, R^j\pi_\ast \underline{V}) \Rightarrow H^{i+j}(X^P_{\wtK}, \underline{V}).\end{eqnarray}
If we take $V$ to be $\DD^s_\Omega$, the Leray spectral sequence naturally leads one to analyze the sheaf $R^j\pi_\ast \underline{\DD^s_\Omega}$ on $X_{I_pK^p}$.

In this article, we also prove a comparison \S \ref{Kos comp} between discrete group cohomology and the locally analytic cohomology introduced by \cite{Koh11} using the Koszul resolution.
Let $N \simeq \ZZ_p^n = \oplus_{i=1}^n \ZZ_p e_i$ be a compact $p$-adic analytic abelian group, with Lie algebra $\frak{n}$ over $\QQ_p$.  
Let $\Gamma = \oplus_{i=1}^n \ZZ e_i \subset N$ be a finitely generated free abelian group, dense in $N$.
\begin{thm}\label{group lie comparison}
There are natural isomorphisms
$$H^q(\Gamma, V) \simeq H^q(\frak{n}, V)^N$$ for all $q \geq 0$ and any complete Hausdorff locally convex $\QQ_p$-vector space $V$ with a separately continuous action of the distribution algebra $D(N, \QQ_p)$.
\end{thm}
As an application of the comparison theorem, we can relate the local systems $R^j\pi_\ast \underline{\widetilde{\DD}^s_\Omega}$ on the locally symmetric space $X_{I_p K^p}$ to $$H^j_{an}(N_w, \DD^s_{w,\Omega}) \simeq H^j(\frak{n}, \DD^s_{w,\Omega})^{N_w}.$$

The works of \cite{AS}, \cite{Urb11}, \cite{Han17} use overconvergent cohomology of locally symmetric space of a reductive group with the distribution representation coefficient to construct eigenvarieties.
Among reductive groups, eigenvarieties for groups with discrete series are easiest to study. We refer the reader to \cite{Urb11} for example.
The quasi-split unitary group $G$ is such an example.
But for a general $\GL_{n,F}$, the eigenvarieties are rather mysterious, we do not even know their dimensions. 
To better understand these eigenvarieties, we may want to relate the overconvergent cohomology for $\GL_{n,F}$ to the overconvergent cohomology for $G$.

Theorem \ref{intro 1} and Theorem \ref{group lie comparison} together provide the distribution representation coefficient of $\GL_{n,F}$ as a direct summand of $R^j\pi_\ast \underline{\widetilde{\DD}^s_\Omega}$ for suitable degree shift $j$.
Therefore sequences (\ref{S1}), (\ref{S2}) realize this relation between $G$ and $L$ when we take $V$ to be $\DD^s_\Omega$.
We hope our paper offers some foundations and stimulations in the future studies of this direction for functoriality of eigenvarieties between unitary groups and general linear groups.

\subsection*{Acknowledgments}
I heartily thank my advisor Richard Taylor for suggesting me to think about this problem with his numerous discussions and constant encouragement.
I would also like to thank Brian Conrad, Yiwen Ding, Jan Kohlhaase, Lue Pan, Jun Su for answering related questions and helpful discussions.

\section{Notation}\label{Notation}
\hfill\break

If $H$ is a quasi-split reductive group over a field $k$ which splits over a Galois extension $k'/k$ and $S$ is a maximal $k$-split torus in $H$. Let $T=Z_H(S) \subset H$ be the centralizer of $S$ as a 
maximal torus contained in a fixed Borel subgroup $B$, then we write $$W(H, S):=N_H(S)/Z_H(S)=N_H(S)/T$$ for the relative Weyl group. 
It is a constant \'etale group scheme over $k$ by \cite[Lem P.4.1]{Con17}, hence we will identify it with its $k$-points. 
Let $$W(H, T):=N_H(T)/Z_H(T)=N_H(T)/T$$ be the Weyl group for $H$ as a \'etale group scheme over $k$. 
Note that these two notions coincide for split groups. In the quasi-split case, $W(H, S)$ is embedded into $W(H, T)$ as a \'etale subgroup scheme over $k$ and is furthermore identified with $W(H, T)(k)$.  
Length functions for either relative or absolute Weyl groups are defined using the Coxeter group structures for each of them and they are usually different functions when restricted to $W(H, S)$. .

If $P \subset H$ is a parabolic subgroup which contains $B$ with its unipotent radical $N$, then there is a unique Levi subgroup $L \subset P$ which contains $T$. 
We write $W_P(H, S)$ and $W_P(H, T)$ for the corresponding Weyl groups of this Levi subgroup, which can be identified with subgroups of $W(H, S)$ and $W(H, T)$. 

We can form $$X^\ast(T):=\Hom(T \times_k k', \GG_m)$$ as the algebraic characters group of $T$. 
The absolute Weyl group $W(H, T)(k')$ acts on $X^\ast(T)$. Similarly, $W(H,S)$ acts on $$X^\ast(S):=\Hom(S,\GG_m).$$ 
The absolute root system $\Delta(H, T) \subset X^\ast(T)$ and relative root system $\Delta(H, S) \subset X^\ast(S)$ are stable respectively under actions of $W(H, T)(k')$ and $W(H, S)$. With respect to $B$, $$\Delta(H, T)=\Delta^+(H, T) \sqcup \Delta^-(H, T), ~ \Delta(H, S)=\Delta^+(H, S) \sqcup \Delta^-(H, S).$$
If $P$ is a parabolic subgroup of $H$ which contains $B$, then $B \cap L$ is a Borel subgroup of $L$ and we have similar notions $\Delta^{(\pm)}(L, T)$, $\Delta^{(\pm)}(L, S)$ for $L$. The sets
\[ W^P(H, T) = \{ w \in W(H, T)(k) \subset W(H, T)(k') \mid w^{-1}(\Delta^+(L, T)) \subset \Delta^+(H, T) \} \]
\[ W^P(H, S) = \{ w \in W(H, S) \mid w^{-1}(\Delta^+(L, S)) \subset \Delta^+(H, S)\} \]
are respectively parabolic quotients of Coxeter groups $W(H, T)(k)$ and $W(H, S)$ as sets of representatives for the quotients $$W_P(H, T)(k) \backslash W(H, T)(k) ~ \mathrm{and} ~ W_P(H, S) \backslash W(H, S).$$ 
For $w_1, w_2 \in W^P(H, S)$ (resp. $W^P(H,T)$), $w_1 \geq w_2$ if there exists $w_L^1, ~ w_L^2 \in W_P(H, S)$ such that $w_L^1 w_1 \geq w_L^2 w_2$ by proposition 2.5.1 of \cite{BB05}, which defines a partial Bruhat order on $W^P(H, S)$ (resp. $W^P(H,T)$). 
See also section \S 9.1.3 in \cite{Har} for the use of such sets of representatives. 

Let $E$ be a $p$-adic local field with its ring of integers $\cO_E$, uniformizer $\varpi$ and residue field $k_E$. 
Let $k=E, ~ k'=K$ for the discussion above. 
With respect to a previously fixed Borel subgroup containing a maximal torus $T \subset B \subset H$, we use $\Delta^{+}_{H, p} \subset \Delta_{H, p}$, $\Delta_{H, p} =\Delta^{+}_{H, p} \sqcup \Delta^{-}_{H, p}$ to denote corresponding roots for the split root system $\Phi(\overline{\QQ}_p \otimes R_{E / \QQ_p} H, \overline{\QQ}_p \otimes R_{E/ \QQ_p} T)$. 
We note that $\Delta_{H, p}$ (resp. $\Delta^\pm_{H, p}$) is decomposed as disjoint unions of subsets of $\Delta(H, T)$ (resp. $\Delta^\pm(H, T)$) indexed by different embeddings of $E \hookrightarrow \overline{\QQ}_p$. 
Let $\delta$ denote the half sum of $\Delta^{+}(H, T)$ (resp. $\Delta^+_{H, p}$), $w\delta-\delta$ is an $E$-algebraic integral weight for $T$ (resp. a $\QQ_p$-algebraic integral weight for $R_{E / \QQ_p} T$).

Suppose $H_{/ E}$ is unramified and has an integral model $\mathscr{H}_{/ \cO_E}$.
 Then $R_{E/\QQ_p} H$ spreads out to a group scheme $\cH:=R_{\cO_E/\ZZ_p} \mathscr{H}$ over $\ZZ_p$. 
 Moreover, group schemes of $\mathscr{T}, \mathscr{B}, \mathscr{P}, \mathscr{N}$ over $\cO_E$ are defined compatibly with their generic fibres $T, B, P, N$. 
 And we use $\cT, \cB, \cP, \cN$ over $\ZZ_p$ to denote $R_{\cO_E/\ZZ_p} \mathscr{T}$, $R_{\cO_E/\ZZ_p} \mathscr{B}$, $R_{\cO_E/\ZZ_p} \mathscr{P}$, $R_{\cO_E/\ZZ_p} \mathscr{N}$. 
 We call $(\cB,\cT)$ a Borel pair of $\cH$. We remark that $\cT, \cB, \cP, \cN_{/ \ZZ_p}$ may no longer be smooth due to ramification of $E / \QQ_p$. 
 We write $ \overline{\mathscr{B}}, \overline{\mathcal{B}}$ for the opposite Borel corresponding to $\mathscr{B}, \mathcal{B}$.
We write $ \overline{\mathscr{P}}, \overline{\mathcal{P}}$ for the opposite parabolic corresponding to $\mathscr{P},\mathcal{P}$. 
We write $ \mathscr{N}_B, \mathcal{N}_B$, $ \overline{\mathscr{N}}_B, \overline{\mathcal{N}}_B$ and $\overline{\mathscr{N}}, \overline{\mathcal{N}}$
for the unipotent radicals of $\mathscr{B}, \mathcal{B}$, $ \overline{\mathscr{B}}, \overline{\mathcal{B}}$ and $\overline{\mathscr{P}}, \overline{\mathcal{P}}$. We write $I$ for the Iwahori subgroup\[
I=\left\{ h\in \mathscr{H}(\cO_E)\,\mathrm{with} ~ h ~ \mathrm{mod} ~ p \in \mathcal{B}(k)\right\} .\]
Set $N_I:=\mathscr{N}_B(\cO_E)=\mathcal{N}_B(\ZZ_p)$. For any integer $s\geq1$, set 
$$B^{s}=\left\{ b\in\mathscr{B}(\cO_E),~ b\equiv1 ~ \mathrm{in} ~ \mathscr{H}(\cO_E/ \varpi^{s}\cO_E)\right\},$$
$$\overline{B}^{s}=\left\{ b\in\overline{\mathscr{B}}(\cO_E),~ b\equiv1 ~ \mathrm{in} ~ \mathscr{H}(\cO_E/ \varpi^{s}\cO_E)\right\},$$
$N_I^{s}=\mathscr{N}_B(\cO_E) \cap B^{s}=\mathcal{N}_B(\ZZ_p)\cap B^{s}$, $\Nbar_I^{s}=\overline{\mathscr{N}}_B(\cO_E) \cap \overline{B}^{s}=\overline{\mathcal{N}}_B(\ZZ_p)\cap\overline{B}^{s}$
and $T^{s}=\mathscr{T}(\cO_E) \cap \overline{B}^{s}=\mathcal{T}(\ZZ_p)\cap\overline{B}^{s}$, $\Nbar_I:=\Nbar_I^1$ so the Iwahori
decomposition is $I=\Nbar_I \times \mathcal{T}(\ZZ_p)\times N_I$. $B_I:=\mathcal{T}(\ZZ_p)\times N_I$.
We also set \[ I^s=I \cap \mathrm{ker}\left(\mathscr{H}(\cO_E) \to \mathscr{H}(\cO_E/\varpi^s\cO_E)\right), \]
\[ I_0^s=\left\{ g\in I\,\mathrm{with} ~ g ~ \mathrm{mod} ~ \varpi^s \in \overline{\mathscr{B}}(\cO_E/ \varpi^s\cO_E)\right\} .\] with the Iwahori decomposition $I_0^s=\Nbar_I\times \mathcal{T}(\ZZ_p) \times N_I^s$.

Moreover, if there are a finite collection of unramified $p$-adic reductive groups and their reductive group schemes $$\{(H_1)_{/ E_1},\cdots,(H_l)_{/ E_l}\}, ~ \{(\mathscr{H}_1)_{/ \cO_{E_1}},\cdots,(\mathscr{H}_l)_{/ \cO_{E_l}}\}$$ with 
\[ \{S_1,\cdots,S_l\}, ~ \{B_1,\cdots,B_l\}, ~ \{T_1,\cdots,T_l\}, \] 
\[ \{I_1^s,\cdots,I_l^s\}, ~ \{\overline{B}_1^s,\cdots,\overline{B}_l^s\}, ~ \{\Nbar_{I_1}^s,\cdots,\Nbar_{I_l}^s\}, ~ \{N_{I_1},\cdots,N_{I_l}\} \] 
associated to them, we use $I=I^1=\prod\limits_{1 \leq i \leq l} I_i$ to denote the Iwahori subgroup of $\prod\limits_{1 \leq i \leq l} H_i(E_i)$. We use $\varpi_i$ and $k_i$ to denote a choice of uniformizer and residue field for $E_i$.
 In general, we use $$I^s=\prod\limits_{1 \leq i \leq l} I^s_i, ~ I^s_0=\prod\limits_{1 \leq i \leq l} I^s_{0,i},$$ $$T^s=\prod\limits_{1 \leq i \leq l} T_i^s, ~ \overline{B}^s=\prod\limits_{1 \leq i \leq l} \overline{B}_i^s, ~ B_I=\prod\limits_{1 \leq i \leq l} B_{I_i},$$ $$\Nbar_I^s=\prod\limits_{1 \leq i \leq l} \Nbar_{I_i}^s, ~ N_I^s=\prod\limits_{1 \leq i \leq l} N_{I_i}^s$$ to denote various subgroups of the $p$-adic Lie group $\prod\limits_{1 \leq i \leq l} H_i(E_i)$ and $I=\Nbar_I \times T^0 \times N_I$ to denote the Iwahori decomposition similar to notions we introduced above. 
 Similarly one defines $$W(H, S), ~ W(H, T), ~ W_P(H, S), ~ W_P(H, T), ~ W^P(H, S), ~ W^P(H, T)$$ to be products of respective Weyl groups for each component.
 
For an element $a$ in the finite rank abelian group $\ZZ^d$, total degree $|a|$ of  $a$ is defined to be the sum of the absolute value of coordinates of $a$. There is a partial order relation $\underline{a} \geq \underline{a'}$ if and only if $a_i \geq a'_{i}$ for all $i$.

We refer to \cite{BGR84} and \cite{Sch01} for basics in rigid analytic geometry and nonarchimedean functional analysis. 
For two local nonarchimedean field of characteristic zero $E \subset K$ and a locally $E$-analytic manifold $M$, we write $C^{an}(M, K)$ for the space of $K$-valued $E$-analytic functions on $M$. 
For any $K$-affinoid algebra $R$, $C^{an}(M,R):=C^{an}(M, K) \hat{\otimes} R$.

If $E \subset K$ are finite extensions of $\QQ_p$ and $G$ is a locally $E$-analytic group,
let $D(G, K)$ be the distribution algebra of $G$, i.e., the strong dual of the space of locally analytic $K$ valued functions (\cite{ST02}, \cite{ST03}).

For defining the unitary group of our interests, we use the Hermitian form 
\begin{eqnarray*}
J_n = \begin{pmatrix}
0 & \Psi_n \\
-\Psi_n & 0
\end{pmatrix}, 
\end{eqnarray*} where $\Psi_n$ is the matrix with $1$’s on the anti-diagonal and $0$’s elsewhere.

\section{Weight spaces}\label{Weight spaces}

Let $E$ be a finite extension of $\QQ_p$ with uniformizer $\varpi$. Suppose $T_{/E}$ is an unramified torus as generic fibre of a reductive group scheme $\mathscr{T}_{/\cO_E}$.
$$T^s:=\left\{ t \in \mathscr{T}(\cO_E), ~ t \equiv1 ~ \mathrm{in} ~ \mathscr{T}(\cO_E/ \varpi^{s}\cO_E)\right\}.$$ In particular, $T^0=\mathscr{T}({\cO_E})$ is a compact (profinite) $E$-analytic abelian group. In this section we give a product structure up to isogeny for the usual weight spaces of $T^0$ in terms of their locally $\sigma$-analytic subspaces for $\sigma : E \hookrightarrow \overline{\QQ}_p$ running through all embeddings.

Let $\Gamma$ be an $E$-analytic profinite abelian group containing an open subgroup ($E$-analytic) isomorphic to $\cO_E^d$ for some $d$. Let $\Gamma_0$ denote the locally $\QQ_p$ analytic group obtained from $\Gamma$ by restriction of scalars from $E$ to $\QQ_p$. 
When $\Gamma=T^0$, we use $T^0_0$ to denote $\Gamma_0$. If $X$ is a rigid space over a $p$-adic local field, let $\sO(X)$ denote the ring
of rigid functions on $X$. According to \cite{Buz07}, we say that a group
homomorphism $\Gamma \to \sO(X)$ (resp. $\Gamma \to \sO(X)^\times$) is \emph{continuous} if,
for all affinoid subdomains $U$ of $X$, the induced map $\Gamma \to \sO(U)$ (resp. $\Gamma \to \sO(U)^\times$) is continuous.

By \cite[Lem 8.2]{Buz07}, we denote $\mathscr{W}_{\Gamma}$ as
the quasi-separated rigid space over 
$\mathbb{Q}_p$ (also as a group object in the category of $\QQ_p$-rigid spaces) represented the functor on rigid spaces over $\QQ_p$ sending a $\QQ_p$-rigid space $X$ to the set of continuous group homomorphisms $\Gamma \to \sO(X)^\times$, i.e., $X$-points of $\mathscr{W}$
are given by\[
\mathscr{W}_{\Gamma}(X)=\mathrm{Hom}_{\mathrm{cts}}(\Gamma,\sO(X)^\times).\]
We denote by $\lambda$ for both a $\overline{\mathbb{Q}}_{p}$-point
of $\mathscr{W}_{\Gamma}$ as well as the corresponding character of $\mathscr{T}(\cO_E)$.
Let $\mathscr{W}_{\cT}$ denote $\mathscr{W}_{\mathscr{T}(\cO_E)}$.
Given any such character $\lambda:\mathscr{T}(\cO_E)\to\overline{\mathbb{Q}}_{p}^{\times}$,
the image of $\lambda$ generates a subfield $K_{\lambda}\subset\overline{\mathbb{Q}}_{p}$
finite over $\mathbb{Q}_{p}$.

For an $E$-affinoid algebra $R$, we say a continuous group homomorphism $\Gamma \to R$ (resp. $\Gamma \to R^\times$) \emph{locally $E$-analytic} if the underlying function in $C^{an}(\Gamma_0, R)$ belongs to the image of $C^{an}(\Gamma, R)$. 
For a rigid space $X$ over $K$, we say a continuous group homomorphism $\Gamma \to \sO(X)$ (resp. $\Gamma \to \sO(X)^\times$) locally $E$-analytic if for all affinoid subdomains $U$ of $X$, the induced map is locally $E$-analytic. 
We use $\mathrm{Hom}_{\mathrm{an}}^{E}(\Gamma,\sO(X)^\times)$ to denote the subgroup of locally $E$-analytic characters.
Thus we get a subfunctor of $\mathscr{W}_{\Gamma}$. In the case that $\Gamma$ is $\cO_E$, Schneider and Teitelbaum constructed a solution variety using a generalization of Amice's $p$-adic Fourier transform. 
We prove this subfunctor is (relatively) represented by a closed Zariski subspace $\mathscr{W}_{\Gamma}^\sigma$ and compare our results with results in \cite{ST01}.

\cite{ST01} works in the setting that $R$ is a complete field over $E$ and contained in $\CC_p$. $\sigma$ was not needed. 

Suppose $\cO_E=\bigoplus\limits_{i=1}^{[E:\QQ_p]} \ZZ_p u_i$ for $u_i \in \cO_E$, $1 \leq i \leq [E:\QQ_p]$. We use $R^{00}$ to denote topologically nilpotent elements in an $E$-affinoid algebra $R$.
Then we have a well defined log map \begin{eqnarray*} \log: R^{00} & \to & R \\ r & \mapsto & \log(1+r).
\end{eqnarray*}
\begin{lem}\label{CRE}
Let $\chi: \cO_E \to R^\times$ be a continuous character. Then the followings are equivalent:
\begin{enumerate}
\item $\chi$ is locally $E$-analytic.
\item $\chi|_{\varpi^n\cO_E}$ is $E$-analytic for certain $n \in \ZZ_{\geq 0}$, i.e, this function is given by a function in the Tate algebra for $\varpi^n\cO_E$.
\item $\chi$ satisfies a set of ``Cauchy-Riemann" equations, i.e., 
$$\frac{\log(\chi(u_1))}{u_1}=\cdots=\frac{\log(\chi(u_{[E:\QQ_p]}))}{u_{[E:\QQ_p]}}.$$
\end{enumerate} 
\end{lem}

\begin{proof}
The equivalence of the first two claims follows from the definition of $C^{an}(\cO_E, R)$. By \cite[Lem 1]{Buz04}, $\chi(a)-1$ is topologically nilpotent for any $a \in \cO_E$, the log map makes sense for $\chi(a)$. 

If $\chi$ is locally $E$-analytic, there exists $n_1 \in \ZZ_{\geq 0}$ such that $\chi|_{p^{n_1}\cO_E}$ is $E$-analytic and is represented by an analytic function $f(x)$, $x$ as a variable for $\cO_E \simeq p^{n_1}\cO_E$. $f(0)=1 \in R$ and $|f(x)-1|_{R\langle x \rangle} < \infty$ with respect to the norm of $R\langle x \rangle$. Thus there exists $n_2 \in \ZZ_{\geq 0}$ such that $|f(x)-1|_{R\langle \frac{x}{p^{n_2}} \rangle}$ is sufficiently small (say, $< \frac{1}{p}$) with respect to the norm of $R\langle \frac{x}{p^{n_2}} \rangle$. Hence $\log \circ \chi|_{p^{n_1+n_2}\cO_E} : \cO_E \to R$ is an $E$-analytic group homomorphism. It is direct to show that if $g(x) \in R\langle x \rangle$ such that $g(x+y)=g(x)+g(y)$ for any $x,y \in \cO_E$, then $g$ must be linear. $\log \circ \chi|_{p^{n_1+n_2}\cO_E}$ is given by $g(x)=cx$ for certain $c \in R$ ($c \in 1+R^{00}$ if $n_2$ is sufficiently large). This means that $\frac{\log(\chi(p^{n_1+n_2}u_i))}{u_i}=p^{n_1+n_2}\frac{\log(\chi(u_i))}{u_i}$ is independent of $i$ ($1 \leq i \leq [E:\QQ_p]$) for sufficiently large $n_1+n_2$.

For any continuous character $\chi$ satisfying the equations above, it suffices to restrict to $p^n\cO_E$ such that the image of $\chi$ is contained in a sufficiently small neighbourhood of $1$ for which $\exp \circ \log$ makes sense and equals to the identity map. There exists $r_i \in R^{00}$ for each $1 \leq i \leq [E:\QQ_p]$ such that $\chi(p^n u_i)=(1+r_i)^{u_i}$, then the equations force $r_i$ to all be equal to $r$ for certain small enough $r$. $\chi(p^n x)=(1+r)^x$ for any $x \in \cO_E$. And since $r$ is sufficiently small, the coefficients $\frac{r^n}{n!}$ go to $0$, for which $(1+r)^x$ defines a rigid function on $p^n\cO_E$.
\end{proof}

\begin{thm}\label{E-rep}
The subfunctor associating to any $E$-rigid space $X$ $$X \mapsto \mathrm{Hom}_{\mathrm{an}}^{E}(\Gamma,\sO(X)^\times) \subseteq
 \mathrm{Hom}_{\mathrm{cts}}(\Gamma,\sO(X)^\times)$$ is relatively represented by a Zariski closed subgroup object $\mathscr{W}_{\Gamma}^E \hookrightarrow \mathscr{W}_{\Gamma}$ in the category of $E$-rigid spaces. 
\end{thm}

\begin{proof}
By assumption $\cO_E^d \subset \Gamma$ as a open subgroup. By \cite[Lem 2, (iv)]{Buz04}, the natural transform corresponding to $$\mathrm{Hom}_{\mathrm{cts}}(\Gamma,\sO(X)^\times) \to \mathrm{Hom}_{\mathrm{cts}}(\cO_E^d,\sO(X)^\times)$$ by restriction is represented by a finite \'etale map $\mathscr{W}_{\Gamma} \to \mathscr{W}_{\cO_E^d}$. Once $\mathscr{W}_{\cO_E^d}^E \hookrightarrow \mathscr{W}_{\cO_E^d}$ is constructed, we make $\mathscr{W}_{\Gamma}^E$ as the fibre product $\mathscr{W}_{\Gamma} \times_{\mathscr{W}_{\cO_E^d}} \mathscr{W}_{\cO_E^d}^E \hookrightarrow \mathscr{W}_{\Gamma}$ in the sense of rigid spaces, see for example, \cite[\S 9.3.5]{BGR84}. We obtain the universal character of $\Gamma$ by composing $$\tilde{\chi}: \Gamma \to \sO(\mathscr{W}_{\Gamma})^\times \to \sO(\mathscr{W}_{\Gamma}^E)^\times.$$ We verify this immersion is the desired one. For any morphisms of $E$-rigid spaces $X \to \mathscr{W}_{\Gamma}^E \to \mathscr{W}_{\cO_E^d}^E$, we have the commutative diagram $$\xymatrix{
 \Gamma \ar[r] & \sO(\mathscr{W}_{\Gamma})^\times \ar[r] & \sO(\mathscr{W}_{\Gamma}^E)^\times \ar[r] & \sO(X)^\times\\
\cO_E^d \ar[r] \ar[u] & \sO(\mathscr{W}_{\cO_E^d})^\times \ar[r] \ar[u] & \sO(\mathscr{W}_{\cO_E^d}^E)^\times \ar[u] &\\
}$$
$\Gamma \to \sO(X)^\times$ is locally $E$-analytic restricted to $\cO_E^d$ by representability of $\mathscr{W}_{\cO_E^d}^E$, hence locally $E$-analytic. 
On the other hand, for and $E$-rigid space $X$ any locally $E$-analytic character $\chi_X: \Gamma \to \sO(X)^\times$, $\chi_X$ is continuous and induced by a certain map $X \to \mathscr{W}_{\Gamma}$. 
It is also locally $E$-analytic restricted to $\cO_E^d$, which gives another map $X \to \mathscr{W}_{\cO_E^d}^E$, hence induced by a map $X \to \mathscr{W}_{\Gamma}^E$ from the universal property of fibre products. 
The fact that going from morphisms to characters are inverse to each other follows from the same statement for continuous characters of $\Gamma$ and $\mathscr{W}_{\Gamma}$.

For construction of $\mathscr{W}_{\cO_E^d}^E \hookrightarrow \mathscr{W}_{\cO_E^d}$, it suffices to consider $d=1$ as the immersion is really $d$-fold product of $\mathscr{W}_{\cO_E}^E \hookrightarrow \mathscr{W}_{\cO_E}$. $\mathscr{W}_{\cO_E}$ is a $[E:\QQ_p]$-fold product of open unit disks. 
For any affinoid subdomain $U$ in $\mathscr{W}_{\cO_E}$, the ``Cauchy-Riemann" equations in Lem \ref{CRE} define rigid analytic functions on $U$ since coefficients of log function going to infinity with a smaller rate of any positive linear growth and hence cut out a Zariski closed space $U^E$. 
$\mathscr{W}_{\cO_E}^E \hookrightarrow \mathscr{W}_{\cO_E}$ is constructed by gluing $U^E$ for all affinoid subdomain $U$ in $\mathscr{W}_{\cO_E}$. 
Then everything is reduced to Lem \ref{CRE}. The fact that $\mathscr{W}_{\cO_E}^E$ is a subgroup object follows from that subfunctor is valued in the category of groups as well. 
\end{proof}

\begin{rem}
Our variety for $\cO_E$ produces the same set of $R$ points when $R$ is a finite extension field of $E$ as the reduced variety constructed in \cite{ST01} does. 
Although our ``Cauchy-Riemann" equations seem slightly different from theirs as they use infinitesimal ``Cauchy-Riemann" differential equations.
\end{rem}

Let $K/E$ be a sufficiently large extension such that $K$ contains all Galois conjugates of $E$, with a fixed embedding $K \hookrightarrow \overline{\QQ}_p$. 
We make a base change for $\mathscr{W}_{\Gamma}$ from $\QQ_p$ to $K$. 
By \cite[Lem 8.2. (c)]{Buz07}, $(\mathscr{W}_{\Gamma})_{/ K}$ represents the functor on rigid spaces over $K$ sending a $K$-rigid space $X$ to the set of continuous group homomorphisms $\Gamma \to \sO(X)^\times$, for which we still denote as $\mathscr{W}_{\Gamma}$. 
For any embedding $\sigma : E \hookrightarrow K \hookrightarrow \overline{\QQ}_p$, it induces a closed embedding of $$C^{an}(\Gamma, K) \hookrightarrow C^{an}(\Gamma_0, K)$$ by \cite[Lem 1.1]{ST01}. 
In general for a $K$-affinoid algebra $R$, $\sigma$ induces an embedding $$\sigma: C^{an}(\Gamma, R) \hookrightarrow C^{an}(\Gamma_0, R).$$ 
One can similarly define locally $(\sigma,E)$-analytic group homomorphisms (resp. characters) of $\Gamma$ valued in rigid functions of any rigid spaces over $E$. 
We use $\mathrm{Hom}_{\mathrm{an}}^{\sigma}(\Gamma,\sO(X)^\times)$ to denote the subgroup of locally $(\sigma,E)$-analytic characters. 
By Thm \ref{E-rep}, the subfunctor associating to any $K$-rigid space $X$ $$X \mapsto \mathrm{Hom}_{\mathrm{an}}^{\sigma}(\Gamma,\sO(X)^\times) \subseteq
 \mathrm{Hom}_{\mathrm{cts}}(\Gamma,\sO(X)^\times)$$ is represented by $\mathscr{W}_{\Gamma}^{\sigma} \hookrightarrow \mathscr{W}_{\Gamma}$ over $K$. 
The logarithmic functions again define rigid analytic functions on $\mathscr{W}_{\Gamma}$. And they cut out zero-dimensional Zariski subspaces $\mathscr{W}_{\Gamma}^{sm} \hookrightarrow \mathscr{W}_{\Gamma}^{\sigma} \hookrightarrow \mathscr{W}_{\Gamma}$ for any $\sigma: E \hookrightarrow K$ whose points correspond to smooth characters of $\Gamma$.

\begin{lem}\label{p-Jacobian}
$R$ is a $K$-affinoid algebra. Let $\chi: \cO_E \to R^\times$ be a continuous character. 
There exists a constant $c_E$ depending only on $E$ such that if $|\chi(a)-1| < c_E$ for all $a \in \cO_E$, then $\chi$ has a factorization $\chi=\prod\limits_{\sigma: E \hookrightarrow K} \chi^\sigma$ such that $\chi^\sigma$ is locally $(\sigma,E)$-analytic for any $\sigma: E \hookrightarrow K$. Each $\chi^\sigma$ is unique up to a smooth character for the factorization.
\end{lem}

\begin{proof}
Suppose $\chi(u_i)=1+r_i$ for all $1 \leq i \leq [E:\QQ_p]$ and the set of embedding from $E \hookrightarrow K$ is $\{\sigma_1,\cdots,\sigma_{[E:\QQ_p]}\}$.  
For $z_1,\cdots,z_{[E:\QQ_p]} \in \ZZ_p$, $$z=u_1\cdot z_1+\cdots+u_{[E:\QQ_p]}\cdot z_{[E:\QQ_p]} \in \cO_E,$$ 
$$\sigma_i z = \sigma_i(u_1)\cdot z_1+\cdots+\sigma_i(u_{[E:\QQ_p]})\cdot z_{[E:\QQ_p]}.$$ $$\mathrm{Let} ~ M:=\begin{pmatrix} \sigma_1 u_1 & \cdots & \sigma_1 u_{[E:\QQ_p]} \\
\vdots & \ddots & \vdots \\ \sigma_{[E:\QQ_p]} u_1 & \cdots & \sigma_{[E:\QQ_p]} u_{[E:\QQ_p]}
\end{pmatrix}, ~ c_E:=|\mathrm{det}(M)|p^{-1}.$$ 
$$\mathrm{Let} ~ (p_1,\cdots,p_{[E:\QQ_p]}) := (\log(1+r_1),\cdots,\log(1+r_{[E:\QQ_p]})) \cdot M^{-1},$$ then $|p_i|<p^{-1}$, $|\exp(p_i)-1|<p^{-1}$ for any $1 \leq i \leq [E:\QQ_p]$. 
Define $\chi^{\sigma_i}(z):=\exp(p_i)^{\sigma_i z}$ for any $z \in \cO_E$, whose underlying functions are rigid functions (convergent power series) on $\cO_E$ via $\sigma_i$ because of $|\exp(p_i)-1|<p^{-1}$, $1 \leq i \leq [E:\QQ_p]$. 
Since $|r_i|<p^{-\frac{1}{p-1}}$, $\exp \circ \log(1+r_i)=1+r_i$. 
Applying exponential function to both sides of $$(p_1,\cdots,p_{[E:\QQ_p]})\cdot M=(\log(1+r_1),\cdots,\log(1+r_{[E:\QQ_p]}))$$ shows values of $\chi$ and $\prod\limits_{1 \leq i \leq [E:\QQ_p]} \chi^{\sigma_i}$ coincide on $\{u_1,\cdots,u_{[E:\QQ_p]}\}$, hence they are equal.

Uniqueness suffices to show that if $\prod\limits_{1 \leq i \leq [E:\QQ_p]} \chi^{\sigma_i} = 1$ for locally $(\sigma_i,E)$-analytic characters, then each $\chi^{\sigma_i}$ is smooth for $1 \leq i \leq [E:\QQ_p]$. By passing to a small open subgroup, we may assume each $\chi^{\sigma_i}$ is $(\sigma_i,E)$-analytic and whose image is contained in a small enough neighbourhood of $1$. By reasonings in the second paragraph of proof of Lem \ref{CRE}, $\log \circ \chi^{\sigma_i}$ is linear and represented by $c_i \sigma_i(z)$. Then $\sum\limits_{i=1}^{[E:\QQ_p]}c_i \sigma_i(z)=0$ for any $z \in \cO_E$ $\Rightarrow c_i=0$, implying smoothness of $\chi^{\sigma_i}$ for all $1 \leq i \leq [E:\QQ_p]$.
\end{proof}

\begin{cor}\label{unique ana decom}
$R$ is a $K$-affinoid algebra. Let $\chi: \Gamma \to R^\times$ be a continuous character. Then there exists an open subgroup $\Gamma_0 \simeq \cO_E^d$ such that $\chi|_{\Gamma_0}$ admits a unique decomposition $$\chi|_{\Gamma_0} = \prod\limits_{\sigma: E \hookrightarrow K} \chi^\sigma_0$$ of $(\sigma,E)$-analytic characters $\chi^\sigma : \cO_E^d \simeq \Gamma_0 \to R^\times$, i.e., each $\chi^\sigma_0$ is represented by a rigid analytic function on $\cO_E$ via $\sigma$.
\end{cor}

\begin{proof}
Everything is reduced to Lem \ref{p-Jacobian} except for the uniqueness. The uniqueness follows from the fact that as a $K$-affinoid algebra, $R$ only contains a finite number of roots of unity. Hence if $\Gamma_0$ is small enough, any smooth character of $\Gamma_0$ lifting to a character of $\Gamma$ must be trivial.
\end{proof}

\begin{thm}
The composition of the maps using that $\mathscr{W}_{\Gamma}$ is a group object 
$$\prod \sigma: \prod_{\sigma: E \hookrightarrow K} \mathscr{W}_{\Gamma}^{\sigma} \hookrightarrow \prod_{\sigma: E \hookrightarrow K} \mathscr{W}_{\Gamma} \xrightarrow{m} \mathscr{W}_{\Gamma}$$ is an isogeny, i.e., $\prod \sigma$ is surjective with the kernel being a zero-dimensional Zariski closed subspace. Moreover, if we denote the kernel as $K_{\Gamma}$, then $K_{\Gamma} \hookrightarrow \prod\limits_{\sigma: E \hookrightarrow K} \mathscr{W}_{\Gamma}^{sm}$.
\end{thm}

\begin{proof}
As in the proof of Thm \ref{E-rep}, we have the pullback diagram 
$$\begin{tikzcd}
K_{\Gamma} \ar[r, hook] \ar[d] & \prod\limits_{\sigma: E \hookrightarrow K} \mathscr{W}_{\Gamma}^{\sigma} \ar[r, hook] \ar[d] & \prod\limits_{\sigma: E \hookrightarrow K} \mathscr{W}_{\Gamma} \ar[r] \ar[d] & \mathscr{W}_{\Gamma} \ar[d]\\
K_{\cO_E^d} \ar[r, hook] & \prod\limits_{\sigma: E \hookrightarrow K} \mathscr{W}_{\cO_E^d}^{\sigma} \ar[r, hook] & \prod\limits_{\sigma: E \hookrightarrow K} \mathscr{W}_{\cO_E^d} \ar[r] & \mathscr{W}_{\cO_E^d}\\
\end{tikzcd}$$ and the bottom line is a $d$-fold product. It suffices to prove the case $\Gamma=\cO_E$. 
For surjectivity, it boils down to show that for any $K$-affinoid algebra $R$ any a continuous character $\chi: \cO_E \to R^\times$, $\chi=\prod\limits_{\sigma: E \hookrightarrow K} \chi^\sigma$ for locally $(\sigma,E)$-analytic characters $\chi^\sigma$. 
$\mathscr{W}_{\cO_E} \to \mathscr{W}_{p\cO_E}$ (resp. $\mathscr{W}_{\cO_E}^{\sigma} \to \mathscr{W}_{p\cO_E}^{\sigma}$) is surjective since for any $K$-affinoid algebra $R$ and a topologically nilpotent element $r \in R$, $(1+r)^{\frac{1}{p}}-1$ are again topologically nilpotent (one may need to enlarge the ring $R$ to include this element). 
If we prove this factorization exists when $\chi$ restricts to some $p^n\cO_E$, $n \in \ZZ_{\geq 0}$, namely, $\chi|_{p^n\cO_E}=\prod\limits_{\sigma: E \hookrightarrow K} \chi_n^\sigma$, we pick $\chi_\circ^\sigma: \cO_E \to R^\times$ lifting each $\chi_n^\sigma$, then $\chi^{-1} \prod\limits_{\sigma: E \hookrightarrow K} \chi_\circ^\sigma$ is smooth, hence locally $(\sigma,E)$-analytic for any $\sigma$. 
By passing to a sufficiently small open subgroup as we do in Lem \ref{CRE}, we can assume the image of $\chi$ is contained in a small enough neighbourhood of $1$ in $R$. Then everything is reduced to Lem \ref{p-Jacobian}.
\end{proof}

Given an affinoid subdomain $\Omega \subset \mathscr{W}_{\cT}$, we write $\chi_{\Omega}:\mathscr{T}(\cO_E)\to \mathscr{O}(\Omega)^{\times}$ for the universal (continuous) character it determines. 
In the case $T$ comes from a maximal torus of a reductive group $H$ with the relative Weyl group $W(H, S)$, we make the following notation. 
For a general weight valued in an affinoid algebra $R$, $$\lambda : \mathscr{T}(\cO_E) \to R^\times, \lambda^w(t):=\lambda(w^{-1}tw),$$ $$w\cdot\lambda(t):= \lambda^w(t)\cdot (w\delta-\delta)(t)$$ for $w \in W(H, S)$ and all $t \in \mathscr{T}(\cO_E)$.

\section{Locally analytic distributions}\label{la rep}
\hfill\break
Consider a 
 reductive group $H$ over a $p$-adic local field $E$ along with various subgroups and an Iwahori subgroup $I$ defined in \S \ref{Notation}. 
Let $K/E$ be an unramified quadratic extension. 

Suppose for each $s \geq 1$, there exists and we fix an analytic isomorphism $\psi_{I,K}^s: \cO_E^{d_e} \times \cO_K^{d_k} \simeq \Nbar_I^s$, \mbox{$d=d_e+2d_k=\mathrm{dim} ~ \mathscr{N}_B$} 
 such that each coordinate corresponds to a negative root in $\Delta^-(H, S)$.  
 \footnote{One can certainly generalise this assumption. However, to save notation and words we decide not to do so. Also, the most complicated case of considerations will be the quasi-split unitary group for an unramified extension $K/E$ in \S \ref{analysis}, which satisfies the assumption.} 
An explicit analytic (or strict analytic in the sense of \cite[\S 3.3]{AS}) structure for relevant groups will be given in \S \ref{analysis}.
Namely, let \[x=\left((x_E^1,\cdots,x_E^{d_e}),(x_K^1,\cdots,x_K^{d_k})\right) \in \cO_E^{d_e} \times \cO_K^{d_k},\] for any $t \in \mathscr{T}(\cO_E)$ and characters $\{\delta_1^{E-},\cdots,\delta_{d_e}^{E-},\delta^{K-}_1,\cdots,\delta^{K-}_{d_k}\}$, 
$$t \cdot x=txt^{-1}=\left((\delta_1^{E-}(t)x_E^1,\cdots,\delta_{d_e}^{E-}(t)x_E^{d_e}),(\delta^{K-}_1(t)x_K^1,\cdots,\delta^{K-}_{d_k}(t)x_K^{d_k})\right).$$ 
Choose an isomorphism $\imath_{K/E}: \cO_K \simeq \cO_E^2$, which amounts to an isomorphism $\psi^s_I := \psi_{I,K}^s \circ (\mathrm{id},(\imath_{K/E}^{-1})^{d_k}) : \cO_E^d \simeq \Nbar_I^s$. 
 
We furthermore assume that the transition function between $\psi^1_I$ and $\psi^s_I$ is given by 
 \begin{eqnarray*}
 (\psi^1_I)^{-1} \circ \psi^s_I : \cO_E^d \xrightarrow{\sim} \Nbar_I^s & \hookrightarrow &\Nbar_I^1 \xrightarrow{\sim} \cO_E^d \\
 (x_1,\cdots,x_d) & \mapsto & (\varpi^{s-1}x_1,\cdots,\varpi^{s-1}x_d).
 \end{eqnarray*} We will verify this key assumption for our interested groups in \S \ref{analysis}.
 
We choose an isomorphism $\imath_E: \cO_E \simeq \ZZ_p^{[E:\QQ_p]}$, which leads to an isomorphism $\psi^s_{I, p}:=\psi^s_I \circ (\imath_E^{-1})^{d}: \ZZ_p^{d[E:\QQ_p]} \simeq \Nbar_I^s$.

For example, if $H$ is a symplectic group over $E$, \[ d_k=0, d=d_e=\mathrm{dim} \: \cN_B=n^2. \] 
If $H$ is a quasi-split unitary group with the skew anti-diagonal Hermitian form (\S \ref{analysis}), \[ d_e=n, d_k=n(n-1), d=d_e+2d_k=\mathrm{dim} \: \mathscr{N}_B=n(2n-1). \] 

For Iwahori subgroup $I$ of a finite product of $p$-adic groups $\prod\limits_{1 \leq i \leq l} H_i(E_i)$ considered in \S \ref{Notation}, we fix isomorphisms $\psi^s_{I_i}:\cO^{d_i}_{E_i} \simeq \Nbar_{I_i}^s$ with respect to root decomposition coordinates as well as $\imath_{E_i}: \cO_{E_i} \simeq \ZZ_p^{[E_i:\QQ_p]}$, $\psi^s_{I_i, p}: \psi^s_{I_i} \circ \imath_{E_i}^{-1} : \ZZ_p^{d_i[E_i:\QQ_p]} \simeq \cO_{E_i}$ for $1 \leq i \leq l$ and $\psi^s_{I, p}: \prod\limits_{1 \leq i \leq l} \psi^s_{I_i} \circ \imath_{E_i}^{-1} : \ZZ_p^d \simeq \Nbar_I^s$ for $d=\sum\limits_{1 \leq i \leq l} d_i[E_i:\QQ_p]$ as above.

\begin{df}
If $R$ is any $\mathbb{Q}_p$-Banach
algebra and $s$ is a positive integer, the module $C^{s,an}(\Nbar_I,R)$
of $s$-locally analytic $R$-valued functions on $\Nbar_I$
is the $R$-module of continuous functions $f:\Nbar_I \to R$
such that \[
f\left(x\psi_{I,p}^{s}(z_{1},\dots,z_{d})\right):\ZZ_p^{d}\to R\]
is given by an element of the $m$-variables Tate algebra $T^{I}_{R}=R\left\langle z_{1},\dots,z_{d}\right\rangle $
for any fixed $x\in \Nbar_I$.
\end{df}

Let $\left\Vert \cdot \right\Vert _{T^{I}_{R}}$ denote the Gauss
norm on the Tate algebra, the norm $\left\Vert f(x\psi_{I,p}^{s})\right\Vert _{T^{I}_{R}}$
depends only on the image of $x$ in $\Nbar_I^{1} / \Nbar_I^{s}$,
and the formula\[
\left\Vert f\right\Vert _{s}=\mathrm{sup}_{x\in \Nbar_I^{1}}\left\Vert f(x\psi^{s}_{I,p})\right\Vert _{T^{I}_{R}}\]
defines a Banach $R$-module structure on $C^{s,an}(\Nbar_I,R)$,
with respect to which the canonical inclusion $C^{s,an}(\Nbar_I,R) \subset C^{s+1,an}(\Nbar_I,R)$ is compact. For a proof of compactness, see for example, \cite[Lem 3.2.2]{Urb11}.

For each $s\geq1$ and $w \in \prod\limits_{1 \leq i \leq l} W^{P_i}(H_i, S_i)=W^P(H, S)$, we define $$\psi_{w,\underline{K}}^s:=\prod\limits_{1 \leq i \leq l} \psi_{w_iI_iw_i^{-1}, K_i}^s : \prod\limits_{1 \leq i \leq l} \cO_{E_i}^{d_{e,i}} \times \cO_{K_i}^{d_{k,i}} \simeq w\Nbar_I^{s}w^{-1},$$ $$\psi_w^s:=\prod\limits_{1 \leq i \leq l} \psi_{w_iI_iw_i^{-1}}^s : \prod\limits_{1 \leq i \leq l} \cO_{E_i} \simeq w\Nbar_I^{s}w^{-1},$$ $$\psi_{w,p}^{s}:=\psi_{wIw^{-1},p}^s: \ZZ_p^{d}\simeq w\Nbar_I^{s}w^{-1}.$$ This is just the case where the usual Borel $B$ is replaced by $wBw^{-1}$. We have a parallel definition as above. 
\begin{df}
If $R$ is any $\mathbb{Q}_p$-Banach
algebra and $s$ is a positive integer, the module $C^{s,an}(w\Nbar_Iw^{-1},R)$
of $s$-locally analytic $R$-valued functions on $w\Nbar_Iw^{-1}$
is the $R$-module of continuous functions $f:w\Nbar_Iw^{-1} \to R$
such that \[
f\left(x\psi_{w,p}^{s}(z_{1},\dots,z_{d})\right):\ZZ_p^{d}\to R\]
is given by an element of the $d$-variables Tate algebra $T^w_{I,R}=R\left\langle z_{1},\dots,z_{d}\right\rangle $
for any fixed $x\in w\Nbar_Iw^{-1}$.
\end{df}
Similarly we use $\left\Vert \cdot \right\Vert _{T^w_{I,R}}$ to denote the Gauss
norm on $T^w_{I,R}$, 
and the norm\[
\left\Vert f\right\Vert _{s}=\mathrm{sup}_{x\in w\Nbar_Iw^{-1}}\left\Vert f(x\psi_{w,p}^{s})\right\Vert _{T^w_{I,R}}\]
defines a Banach $R$-module structure on $C^{s,an}(w\Nbar_Iw^{-1},R)$.

For the Iwahori subgroup $I$ of $\prod\limits_{1 \leq i \leq l} H_i(E_i)$ and for a spherically complete field $K \supset E_1,\cdots, E_l$ with respect to a nonarchimedean valuation extending the ones on $E_1,\cdots,E_l$, $D(I, K)$ is the distribution algebra of the Iwahori group $I$.
Let $\Omega$ be an irreducible Zariski closed subspace of an affinoid subdomain of $\mathscr{W}_{T^0}=\prod\limits_{1 \leq i \leq l} \mathscr{W}_{T^0_i}$.

For an analytic Banach $\mathscr{O} (\Omega)[B_I]$-module $V$, where $B_I$ is the Borel subgroup of $I$, $T$ is the diagonal torus of $I$ and $N_I, \Nbar_I$ are the corresponding unipotent and opposite unipotent of $I$, we can define the locally analytic induction of $V$ from $B_I$ to $I$ as follows:
\begin{multline*}
\mathrm{Ind}_{I}^{s} (V)=\{ f : I \to V,\,  f\,\mathrm{analytic\, on\, each\,} I^{s} -\mathrm{coset}, 
f(gb)=b\cdot f(g)\,\forall b \in B_I, \, g\in I\}.
\end{multline*}  
Recall $A=\sO(\Omega)$, for definition of $V$-valued analyticity, we mean $f|_{i \cdot I^s} \in C^{an}(i\cdot I^s, A) \hat{\otimes}_A V$ for every $i \in I/I^s$ and $C^{an}(i\cdot I^s, A)$ stands for the space of $A$-valued analytic functions on $i\cdot I^s$, or in other words, copies of Tate algebras corresponding to $i \cdot I^s$ as each $i\cdot I^s$ is isomorphic to a finite copies of $(\ZZ_p)^m$.
\begin{eqnarray*}
\mathrm{Ind}_{I}^{s} (V) & \simeq & C^{s,an}(\Nbar_I, V) \simeq C^{s,an}(\Nbar_I, \sO(\Omega)) \hat{\otimes}_{\sO(\Omega)} V\\
f & \mapsto & f|_{\Nbar_I},\end{eqnarray*}
and we regard $\mathrm{Ind}_{I}^{s} (V)$ as a Banach $\mathscr{O} (\Omega)$-module
via pulling back the Banach module structure on $C^{s,an}(\Nbar_I, V)$
under this isomorphism. The rule $(f|\gamma)(g)=f(\gamma g)$ gives
$\mathrm{Ind}_{I}^{s} (V)$ the structure of a continuous right $\mathscr{O} (\Omega)[I]$-module. 

We define the Banach dual, i.e., bounded $\sO(\Omega)$ linear functionals, as a Banach $\mathscr{O} (\Omega)[\Delta^+_H]$-module \begin{eqnarray*}
\mathbb{D}_I^{s}(V) & := & \mathcal{L}_{\mathscr{O}(\Omega)}(\mathrm{Ind}_{I}^{s}(V),\mathscr{O}(\Omega)).
\end{eqnarray*} The action of $\sO(\Omega)[\Delta^+_H]$ for $\mathbb{D}_{I}^{s}(V)$ is induced from that of $\mathrm{Ind}_{I}^{s} (V)$. In particular, for a character $\chi$ of $T^0$ valued in $\sO(\Omega)^\times$,
\begin{multline*}
\mathrm{Ind}_{I,\chi}^{s}:=\{ f : I \to \sO(\Omega),\, f\,\mathrm{analytic\, on\, each\,}I^s-\mathrm{coset}, \\
f(gtn)=\chi(t)f(g)\,\forall n\in N_I,\, t\in \cT(\ZZ_p),\, g\in I\} , 
\end{multline*} $$\mathbb{D}_{I,\chi}^{s}:=\mathcal{L}_{\mathscr{O}(\Omega)}(\mathrm{Ind}_{I, \chi}^{s},\mathscr{O}(\Omega)).$$ 

We define $\chi_{\Omega}: T^0 \to \sO(\Omega)^\times$ to be the universal (continuous) character of $T^0$ for $\Omega \subset \mathscr{W}_{T^0}$.
We define $s[\Omega]$ as the minimal integer such that $\chi_{\Omega}|_{T_{i}^{s[\Omega]}}$ satisfies Cor \ref{unique ana decom} (uniquely decomposes as a product of analytic characters) for all $1 \leq i \leq l$. For any positive integer $s$, the Weyl group element $w \in W^{P}(H, T)$ and $wIw^{-1}$, we make the following definitions.
\begin{multline*}
\mathrm{Ind}_{w, \chi}^{s}:=\mathrm{Ind}_{wIw^{-1}}^s(\chi)=\{ f : wIw^{-1} \to \sO(\Omega),\, f\,\mathrm{analytic\, on\, each\,}w I^s w^{-1}-\mathrm{coset}, \\
f(gtn)=\chi(t)f(g)\,\forall n\in w\mathscr{N}_B(\mathcal{O}_E)w^{-1},\, t\in \cT(\ZZ_p),\, g\in I\} .
\end{multline*}
By the Iwahori decomposition for $wIw^{-1}$, restricting an element $f \in \mathrm{Ind}_{w, \chi}^{s}$
to $w\Nbar_Iw^{-1}$ induces an isomorphism \begin{eqnarray*}
\mathrm{Ind}_{w, \chi}^{s} & \simeq & C^{s,an}(w\Nbar_Iw^{-1}, \mathscr{O} (\Omega))\\
f & \mapsto & f|_{w\Nbar_Iw^{-1}},\end{eqnarray*} 
and we regard $\mathrm{Ind}_{w, \chi}^{s}$ as a Banach $\mathscr{O} (\Omega)$-module
via pulling back the Banach module structure on $C^{s,an}(w\Nbar_Iw^{-1}, \mathscr{O} (\Omega))$
under this isomorphism. The rule $(f|\gamma)(g)=f(\gamma g)$ gives
$\mathrm{Ind}_{w, \chi}^{s}$ the structure of a continuous right $\mathscr{O} (\Omega)[wIw^{-1}]$-module. We define the Banach dual as a Banach $\mathscr{O} (\Omega)[wIw^{-1}]$-module \begin{eqnarray*}
\mathbb{D}_{w,\chi}^{s} & := & \mathcal{L}_{\mathscr{O}(\Omega)}(\mathrm{Ind}_{w, \chi}^{s},\mathscr{O}(\Omega)).
\end{eqnarray*} The action of $\sO(\Omega)[\Delta^+_H]$ for $\mathbb{D}_{w,\chi}^{s}$ is induced from that of $\mathrm{Ind}_{w, \chi}^{s}$.
When $s \geq s[\Omega]$, we set $\mathrm{Ind}_{w, \Omega}^{s}:=\mathrm{Ind}_{w, \chi^{w}_{\Omega}}^{s}$ and $\mathbb{D}_{w,\Omega}^{s} := \mathbb{D}_{w,\chi_{\Omega}^{w}}^{s}$, where $\chi_{\Omega}^{w}(t)=\chi_{\Omega}(w^{-1}tw)$ for $\forall t \subset \mathcal{T}(\mathcal{O}_E)$. In particular, we define $\DD^s_\Omega:=\DD^s_{id,\Omega}$.

\section{$p$-adic Banach representations over $T_0$}\label{Banach rep}
\hfill\break

In this section, $T_0$ is meant to be a $p$-adic compact torus. More specifically, $T_0 \subset (\cO_E^\times)^{n}$ is a closed subgroup for a $p$-adic local field $E$. In application we will set $T_0 \subset \prod\limits_{1 \leq i \leq l} (\cO_{E_i})^{n_i}$ to be a compact open subgroup of finite index. For understanding the structure of the distribution modules better, we need to establish some general results on $T_0$-Banach representations.

Let $A$ be an affinoid algebra over finite extensions $K/\mathbb{Q}_p$, where $K$ contains all conjugates of $E$.
\begin{df}
\label{cons}
Let $\imath : S \subset \mathbb{Z}^l_{\leq 0}$ be a subset of $l$ copies of non-positive integers. 
For $\forall s \in S$, $V_s$ is a finite free $A$-module with unitary Banach representation structure of $T_0$, with norm $| \cdot |_s$. 
Moreover, the $T_0$ action on $V_s$ is via a locally analytic character $\chi_s : T_0 \to A^\times$. 
For different $s$, $\chi_s$ are different. $$(\prod_{s \in S} V_s)^c := \{ v \in \prod_{s \in S} V_s \big| ~ \mathrm{for} ~ \forall \varepsilon > 0, \exists N>0, ~ \mathrm{s. t.} ~ |(v)_s|_s \leq \varepsilon ~ \mathrm{if} ~ \underset{ 1 \leq j \leq l}{\mathrm{max}} ~ \imath(s)_j > N  \}$$
$$(\displaystyle\prod_{s \in S} V_s)^b := \{ v \in \prod_{s \in S} V_s \big| ~ \underset{s \in S}{\mathrm{max}} ~ |(v)_s|_s < \infty  \}$$ We equip both with the norm $|\cdot|:=\underset{s \in S}{\mathrm{max}} ~ |\cdot|_s$. Thus $(\displaystyle\prod_{s \in S} V_s)^c$ and $(\displaystyle\prod_{s \in S} V_s)^b$ are unitary Banach representations of $T_0$ over $A$ and $V_s \hookrightarrow (\displaystyle\prod_{s \in S} V_s)^c, V_s \hookrightarrow (\displaystyle\prod_{s \in S} V_s)^b$ as the $\chi_s$ isotypic part of $(\displaystyle\prod_{s \in S} V_s)^c, (\displaystyle\prod_{s \in S} V_s)^b$. 
We call a $T_0$ Banach $A$-module \emph{convergent eigen orthonormalizable} if it is isomorphic to some $(\displaystyle\prod_{s \in S} V_s)^c$, \emph{bounded eigen orthonormalizable} if it is isomorphic to some $(\displaystyle\prod_{s \in S} V_s)^b$.
\end{df}
\begin{rem}
$V_s$ can be zero for $s \in S$. $(\displaystyle\prod_{s \in S} V_s)^c \hookrightarrow (\displaystyle\prod_{s \in S} V_s)^b$ as a closed subunitary $T_0$-Banach representation over $A$ with the induced subnorm. 
The latter notion is dual to the first one, the exact meaning will be clear shortly afterwards. In the remaining section, $(\displaystyle\prod_{s \in S} V_s)^c$ is convergent eigen orthonormalizable and $(\displaystyle\prod_{s \in S} V_s)^b$ is bounded eigen orthonormalizable.
\end{rem}

For two Banach $A$-modules $M,N$, the $A$-module $\cL_A(M,N)$ of continuous $A$-linear homomorphisms from $M$ to $N$ is then also a Banach $A$-module. For a Banach $A$-module $M$, we use $M^\ast$ to denote $\cL(M, A)$, the Banach dual of $M$.

For a given $(\displaystyle\prod_{s \in S} V_s)^c$ and assignment $s \mapsto V_s$ for $s \in S$, we define the bounded (resp. convergent) eigen orthonormalizable Banach representation $(\displaystyle\prod_{s \in S} V^\ast_s)^b$ (resp. $(\displaystyle\prod_{s \in S} V^\ast_s)^c$) constructed above in Def \ref{cons} to be construction associated to the dual assignment $s \mapsto V^\ast_s$ for $s \in S$.

\begin{lem}
\label{dual}
$(\displaystyle\prod_{s \in S} V^\ast_s)^b \simeq \cL_A((\displaystyle\prod_{s \in S} V_s)^c, A)$ as $T_0$-unitary Banach representation.
\end{lem}
\begin{proof}
It is a standard fact in nonarchimedean functional analysis that the Banach dual of $c_0$ is $l^\infty$, where $c_0$ is the space of sequences converging to $0$ and $l^\infty$ is the space of bounded sequences. Apply this fact to our $(\displaystyle\prod_{s \in S} V_s)^c$ and $(\displaystyle\prod_{s \in S} V_s)^b$ as $c_0$ and $l^\infty$.
\end{proof}
\begin{rem}
In $(\displaystyle\prod_{s \in S} V^\ast_s)^b$, each $s \mapsto V_s^\ast$ for the assignment with the natural dual norm. 
The dual of $(\displaystyle\prod_{s \in S} V_s)^b$ is not of the form $(\displaystyle\prod_{s \in S} V^\ast_s)^c$ as we know for Banach spaces over spherically complete fields, infinite dimensional Banach spaces are never reflexive.
\end{rem}
In applications, $(\displaystyle\prod_{s \in S} V_s)^c$ is the $s$-analytic induction up to tensoring a finite dimensional vector space. 
And $(\displaystyle\prod_{s \in S} V_s)^b$ is the $s$-distribution dual to the $s$-analytic induction up to tensoring a finite dimensional vector space. 
We have the following simple statement.
\begin{lem}
\label{eigenweight}
Assume $A$ is an integral domain. If $v \in (\displaystyle\prod_{s \in S} V_s)^b$ (resp. $(\displaystyle\prod_{s \in S} V_s)^c$) is an eigenvector for the $T_0$-action, then $v \in V_s$ for some $s \in S$, the action is given by $\chi_s$.
\end{lem}
\begin{proof}
$v=(v_s)$. If $v_s \in V_s$ and $v_{s'} \in V_{s'}$ are both non-zero for $s \neq s'$, $t\cdot v=\chi_v(t)\cdot v$ for any $t \in T_0$ and a character $\chi_v : T_0 \to A^\times$. Thus $$(\chi_s(t)-\chi_v(t))\cdot v_s=(\chi_{s'}(t)-\chi_v(t))\cdot v_{s'}=0$$ for any $t \in T_0$. Since $A$ is an integral domain and $v_s, v_{s'}$ are non-zero, $\chi_s=\chi_v=\chi_{s'}$, which is a contradiction. 
\end{proof}

\begin{df}
\label{nice}
We call a continuous linear map $f: (\displaystyle\prod_{s \in S} V_s)^b \to (\displaystyle\prod_{s' \in S'} V_{s'})^b$ \emph{nice} between two bounded eigen orthonormalizable $A$-modules if for any $\hat{v}=(\hat{v}_s) \in (\displaystyle\prod_{s \in S} V_s)^b$, $\sum_{s \in S}(f|_{V_{s}}(\hat{v}_{s}))_{s'}$ converges for all $s' \in S'$ and moreover we have $$(f(\hat{v}))_{s'}=\sum_{s \in S}(f|_{V_s}(\hat{v}_s))_{s'}, \mathrm{for} ~ f(\hat{v}) \in (\displaystyle\prod_{s' \in S'} V_{s'})^b, \forall s' \in S'.$$
\end{df}
We will show that the differentials in the Chevalley–Eilenberg complex and Iwahori group action for $\mathbb{D}^{s,1}_{w,\chi}$ are nice.

\begin{lem}
\label{dual nice}
Let $$f: (\displaystyle\prod_{s_1 \in S_1} V_{1,s_1})^c \to (\displaystyle\prod_{s_2 \in S_2} V_{2,s_2})^c$$ be a $A$-linear continuous map between two convergent eigen orthonormalizable $A$-modules as in Def \ref{cons}. Then the dual $A$-linear map $$f^\ast: (\displaystyle\prod_{s_2 \in S_2} V_{2,s_2}^\ast)^b \to (\displaystyle\prod_{s_1 \in S_1} V_{1,s_1}^\ast)^b$$ of $f$ between bounded eigen orthonormalizable $A$-modules is nice.
\end{lem}
\begin{proof}
Let $\imath_1: S_1 \subset \ZZ^{l_1}_{\leq 0}$, $\imath_2: S_2 \subset \ZZ^{l_2}_{\leq 0}$. For any finite free Banach $A$-module $M$ with basis $e_1,\cdots,e_m$, norm on $M$ is equivalent to the norm with orthonormalizable basis $e_1,\cdots,e_m$. For any $s_1 \in S_1$, we can assume that $V_{1,s_1}$ has orthonormalizable basis $e^{s_1}_1,\cdots,e^{s_1}_{m_{s_1}}$. Pick any $\hat{v}=(\hat{v}_{s_2}) \in (\displaystyle\prod_{s_2 \in S_2} V_{2,s_2}^\ast)^b$, $(f^\ast|_{V_{2,s_2}^\ast}(\hat{v}_{s_2}))_{s_1} \to 0$ as $\imath_2(s_2) \to \infty$ since for any $\varepsilon>0$, there exists $N_2>0$ such that $$|(f^\ast|_{V_{2,s_2}^\ast}(\hat{v}_{s_2}))_{s_1}(e^{s_1}_i)|=|\hat{v}_{s_2}(f(e^{s_1}_i)_{s_2})| < \varepsilon$$ whenever $\imath_2(s_2) > N_2$, $1 \leq i \leq m_{s_1}$, that means $|(f^\ast|_{V_{2,s_2}^\ast}(\hat{v}_{s_2}))_{s_1}(v_{s_1})| < \varepsilon$ for all $v_{s_1} \in V_{1,s_1}$ such that $|v_{s_1}| \leq 1$ and $|\imath_2(s_2)| > N_2$. Hence $\sum_{s_2 \in S_2} (f^\ast|_{V_{2,s_2}^\ast}(\hat{v}_{s_2}))_{s_1}$ exists and the reasoning above also shows that $f^\ast(\hat{v}_2)_{s_1}=\sum_{s_2 \in S_2} (f^\ast|_{V_{2,s_2}^\ast}(\hat{v}_{s_2}))_{s_1}$ for all $s_1 \in S_1$.
\end{proof}

\begin{lem}
\label{ext of t0}
Suppose $T$ is an abelian group with a finite index subgroup $T_0$.
Let $(\displaystyle\prod_{s \in S} V_s)^c$ be convergent eigen orthonormalizable. 
Moreover, $T$ acts on $(\displaystyle\prod_{s \in S} V_s)^c$ compatibly with $T_0$.  
Then the induced action of $T$ on $(\displaystyle\prod_{s \in S} V^\ast_s)^b \simeq \cL_A((\displaystyle\prod_{s \in S} V_s)^c, A)$ is given by $$(t\cdot\hat{v}^\ast)_s=t\cdot\hat{v}^\ast_s,$$ for any $t \in T, \hat{v}^\ast \in (\displaystyle\prod_{s \in S} V^\ast_s)^b$. In particular, $t : (\displaystyle\prod_{s \in S} V^\ast_s)^b \to (\displaystyle\prod_{s \in S} V^\ast_s)^b$ is a nice automorphism for any $t \in T$.  
\end{lem}
\begin{proof}
Notice that the equality makes sense only if we know $T$ maps $V_s$ to $V_s$ for any $s \in S$, which is clear since $T$ is abelian and the action is compatible with restriction to $T_0$. $$(t\cdot\hat{v}^\ast)(v_s)=\hat{v}^\ast(t^{-1}\cdot v_s)=\hat{v}^\ast_s(t^{-1}\cdot v_s)=t\cdot\hat{v}^\ast_s(v_s)$$ for any $t\in T, s\in S, v_s \in V_s$.
\end{proof}

\begin{lem}
\label{rep of t}
Assume $A=K$. If $\tilde{V}$ is a closed sub $T_0$-Banach representation of $(\displaystyle\prod_{s \in S} V_s)^c$, then there exists an assignment of subspaces $\tilde{V}_s \subset V_s$ to each $s \in S$, with the induced subspace norms, such that $\tilde{V} \simeq (\displaystyle\prod_{s \in S} \tilde{V}_s)^c$. Moreover, $(\displaystyle\prod_{s \in S} V_s)^c/\tilde{V} \simeq (\displaystyle\prod_{s \in S} V_s/\tilde{V}_s)^c$ as $T_0$-Banach representations, where $(\displaystyle\prod_{s \in S} V_s)^c/\tilde{V}$ is equipped with the complete quotient norm of $(\displaystyle\prod_{s \in S} V_s)^c$ and $(\displaystyle\prod_{s \in S} V_s/\tilde{V}_s)^c$ is the Banach representation associated to the assignment $s \mapsto V_s/\tilde{V}_s$ by our construction in Definition \ref{cons}. Each $V_s/\tilde{V}_s$ is equipped with the complete quotient norm of $V_s$. 
\end{lem}
\begin{proof}
$\tilde{V} \cap \bigoplus\limits_{s \in S}V_s$ is a sub representation of $\bigoplus\limits_{s \in S}V_s$, which must be $\bigoplus\limits_{s \in S}\tilde{V}_s$ for an assignment of subspaces $\tilde{V}_s \subset V_s$. We need to prove $\tilde{V} \simeq (\displaystyle\prod_{s \in S} \tilde{V}_s)^c$ for such a sub assignment $s \mapsto \tilde{V}_s$. It is enough to prove $\tilde{V} \subset \overline{\bigoplus\limits_{s \in S}\tilde{V}_s}$. It suffices to prove that for $\forall v \in \tilde{V}, v=(v_s) \in (\displaystyle\prod_{s \in S} V_s)^c$, if $0 \neq v_s \in V_s$, then $v_s \in \tilde{V}$. Given any $s \in S$, there are only finitely many $s' \in S$ such that $|v_{s'}| > |v_s|$. Without loss of generality, assume $|v_s|_s=|v|$, i.e., $v_s$ is the largest component vector of $v$. Suppose $T_0$ is topologically finitely generated by \{$t_1,\cdots,t_n$\}. By Lem \ref{eigenweight}, it suffices to inductively construct $\{v_1,\cdots,v_n\} \subset \tilde{V}$ such that $(v_i)_s=v_s$ and $t_j\cdot v_i=\chi_s(t_j)v_i$ for $j \leq i$. 

Assume we have already constructed \{$v_1,\cdots,v_l$\} ($l$ can be $0$). Again we inductively construct a sequence $v_l^N$ ($N \geq 0$) converging to our desired $v_{l+1}$ with $v_l^0:=v_l$. 
Assume we have constructed \{$v_l^0,\cdots,v_l^N$\}. $\imath: S \hookrightarrow \mathbb{Z}_{\leq 0}^l$ (recall that total degree $|s|$ of $s \in S$ is defined to be the sum of the absolute value of coordinates of $\imath(s)$). 
Set $S_l^N:=\{s' \in S \big| (v_l^N)_{s'}\neq 0\}$. 
Now either $\chi_{s'}(t_{l+1})=\chi_s(t_{l+1})$ for $\forall s' \in S_l^N$, or we can pick an element $s_{N,l} \in S_l^N$ such that $|s_{N,l}|$ is minimal among those of all such elements satisfying the property that $$|\chi_{s_{N,l}}(t_{l+1})-\chi_s(t_{l+1})| \geq |\chi_{s'}(t_{l+1})-\chi_s(t_{l+1})|$$ for $\forall s' \in S_l^N$. 
If the former case occurs, we set $v_{l+1}:=v_l^N$ and stop constructing the sequence. 
Otherwise set $$v_l^{N+1}:=(\chi_{s_{N,l}}(t_{l+1})-\chi_s(t_{l+1}))^{-1}\cdot(\chi_{s_{N,l}}(t_{l+1}) v_l^N-t_{l+1}\cdot v_l^N).$$ $v_l^{N+1}$ satisfies that 
\begin{eqnarray*}
(v_l^{N+1})_s &=& v_s, \\ 
|v_l^{N+1}-v_s| & \leq & |v_l^N-v_s|, \\ 
 S_l^{N+1} & \subset & S_l^N \backslash \{s_{N,l}\} ~ ~ \left((v_l^{N+1})_{s_{N,l}}=0\right). 
\end{eqnarray*} 
If $S_l^N$ is finite, then $v_{l+1}$ is obtained within finitely many steps. 
If $S_l^N$ is infinite, the latter cases always happen, we claim the existence of $\lim\limits_{N \to \infty} v_l^N$. 
For any $s' \in S_l^0$ such that $\chi_{s'}(t_{l+1}) \neq \chi_s(t_{l+1})$, there exists $N'>0$ such that $|s_{N,l}|$ is greater then the maximal coordinate of $|s|$ for all $N>N'$. $$|(v_l^{N+1})_{s'}| \leq |\varpi|\cdot |(v_l^N)_{s'}|$$ for all $N>N'$, which means that $s'$ component of $v_l^N \mapsto v_l^{N+1}$ is a contraction for all $N>N'$. 
Such components clearly go to $0$. For $s' \in S_l^0$ such that $\chi_{s'}(t_{l+1}) = \chi_s(t_{l+1})$, $(v_l^N)_{s'}=(v_l)_{s'}$ for any $N \geq 0$. 
Set $v_{l+1}:=\lim\limits_{N \to \infty} v_l^N$ for this case. 
We see from the construction $(v_{l+1})=v_s, t_j\cdot v_{l+1}=\chi_s(t_j)v_{l+1}$ for $j \leq l+1$, which completes the induction, hence the proof of the first part.

For the second part of the lemma, we construct the following natural map 
\begin{eqnarray*}
(\displaystyle\prod_{s \in S} V_s)^c/\tilde{V} & \to & (\displaystyle\prod_{s \in S} V_s/\tilde{V}_s)^c \\
(x_s) & \mapsto & (\overline{x}_s)  , ~ ~ ~ ~ x_s \in V_s, ~ ~ \overline{x}_s \in V_s/\tilde{V}_s.
\end{eqnarray*} This map makes sense and is well defined thanks to the first part of the lemma. It is continuous since $|\overline{x}_s| \leq |x_s|$ for each $s \in S$. Given $\overline{x}_s \in V_s/\tilde{V}_s$, there exists $x_s^0 \in V_s$ such that $|x_s^0|=|\overline{x}_s|$. We construct the inverse map
\begin{eqnarray*}
(\displaystyle\prod_{s \in S} V_s/\tilde{V}_s)^c & \to & (\displaystyle\prod_{s \in S} V_s)^c/\tilde{V} \\
(\overline{x}_s) & \mapsto & (x_s^0).
\end{eqnarray*} Note if $|x^0_s+v^0_s|=|x^0_s|$ for some $v_s^0 \in \tilde{V}_s$, $|v^0_s| \leq \mathrm{max}(|x^0_s+v^0_s|,|x^0_s|)=|\overline{x}_s|$. Thus the inverse map is well defined.
We know the two norms on both sides induced from the isomorphism are equivalent since one is stronger than another.
To prove $(\displaystyle\prod_{s \in S} V_s)^c/\tilde{V}$ and $(\displaystyle\prod_{s \in S} V_s/\tilde{V}_s)^c$ have the same norm, it is translated to prove for $\forall x \in (\displaystyle\prod_{s \in S} V_s)^c$, $$\inf_{v \in \tilde{V}}\sup_{s \in S}|x_s+v_s|=\sup_{s \in S}\inf_{v_s \in \tilde{V}_s}|x_s+v_s|.$$
Fix a $v \in \tilde{V}$, $$\sup_{s \in S}|x_s+v_s| \geq \sup_{s \in S}\inf_{v'_s \in \tilde{V}_s}|x_s+v'_s|,$$ which gives "$\geq$". 
For each $s \in S$, choose $v_s^0 \in \tilde{V}_s$ minimizing $|x_s+v_s^0|$, then $|v_s^0| \leq |x_s|$ by the same non-Archimedean triangle inequality. 
Let $v^0=(v^0_s)_s \in \tilde{V}=(\displaystyle\prod_{s \in S} \tilde{V}_s)^c$, $$\sup_{s \in S}|x_s+(v^0)_s|=\sup_{s \in S}|x_s+v^0_s|=\sup_{s \in S}\inf_{v_s \in \tilde{V}_s}|x_s+v_s|,$$ yielding "$\leq$".
\end{proof}

There is a \^{} completion operation on closed $T_0$-subrepresentations of $(\displaystyle\prod_{s \in S} V_s)^b$ (and $(\displaystyle\prod_{s \in S} V_s)^c$).
We regard $(\displaystyle\prod_{s \in S} V_s)^c$ as a closed subrepresentation of $(\displaystyle\prod_{s \in S} V_s)^b$.
By the previous lemma, $\tilde{V} \cap (\displaystyle\prod_{s \in S} V_s)^c \simeq (\displaystyle\prod_{s \in S} \tilde{V}_s)^c$ for the assignment $s \mapsto \tilde{V}_s$. We set the \^{} completion of $\tilde{V}$ to be $(\displaystyle\prod_{s \in S} \tilde{V}_s)^b$ for this assignment in the sense of Def \ref{cons}. 

\begin{rem}
It seems that in general $\tilde{V}$ is not necessarily contained in $\hat{\tilde{V}}$.
\end{rem}

\begin{lem}
\label{l11}
Assume $A=K$. Let $\tilde{V}$ be a closed $T_0$-representation of $(\displaystyle\prod_{s \in S} V_s)^b$. The followings are equivalent: 
\begin{enumerate}
\item $\tilde{V} \subset \hat{\tilde{V}}$.
\item For any $\tilde{v}=(\tilde{v}_s)_s \in \tilde{V}$, if $\tilde{v}_s \neq 0$, then $\tilde{v}_s \in \tilde{V}$.
\end{enumerate}
\end{lem}
\begin{proof}
$(2) \Rightarrow (1)$ is obvious. $(1) \Rightarrow (2)$: If there exists $\tilde{v}=(\tilde{v}_s)_s \in \tilde{V}$ such that $\tilde{v}_s \neq 0$ and $\tilde{v}_s \notin \tilde{V}$. 
Then for $\hat{\tilde{V}}=(\displaystyle\prod_{s \in S} \tilde{V}_s)^b$, we have $\tilde{v}_s \notin \tilde{V}_s$, which implies that $\tilde{v}=(\tilde{v}_s)_s \notin \hat{\tilde{V}}$.
\end{proof}

\begin{lem}
\label{direct sum}
Assume $A=K$. $V_1, V_2$ are both closed $T_0$-representations of $(\displaystyle\prod_{s \in S} V_s)^b$. Suppose $p: (\displaystyle\prod_{s \in S} V_s)^b \to V_1$ is a continuous projection and $(V_1 \cap (\displaystyle\prod_{s \in S} V_s)^c) \oplus (V_2 \cap (\displaystyle\prod_{s \in S} V_s)^c)=(\displaystyle\prod_{s \in S} V_s)^c$, then $$(\displaystyle\prod_{s \in S} V_s)^b=\hat{V_1} \oplus \hat{V_2}.$$
\end{lem}
\begin{proof}
$V_1, V_2$ are both closed in $(\displaystyle\prod_{s \in S} V_s)^b$, so are $V_1 \cap (\displaystyle\prod_{s \in S} V_s)^c, V_2 \cap (\displaystyle\prod_{s \in S} V_s)^c$ in $(\displaystyle\prod_{s \in S} V_s)^c$. By the Lem \ref{rep of t}, $V_i \cap (\displaystyle\prod_{s \in S} V_s)^c=(\displaystyle\prod_{s \in S} V_{i, s})^c$ for the assignment $s \mapsto V_{i,s}$ ($i=1,2$) such that $V_{1,s} \oplus V_{2,s}=V_s$ for every $s \in S$. If $v \in \hat{V_1} \cap \hat{V_2}$, its $s$-component $v_s \in V_{1,s} \cap V_{2,s}=0$, so $v=0, \hat{V_1} \cap \hat{V_2}=0$. And for any $v \in (\displaystyle\prod_{s \in S} V_s)^b$, $v_s=v_{1,s}+v_{2,s}$ for each $s \in S$ and unique $v_{1,s} \in V_{1,s}, v_{2,s} \in V_{2,s}$. Since $p$ is continuous, there exists a constant $c$ such that $$|v_{1,s}|_s=|v_{1,s}|^S=|p(v_s)| \leq c\cdot |v_s|_s$$ $$|v_{2,s}|_s=|v_{2,s}|^S=|v_s-p(v_s)| \leq c\cdot |v_s|_s.$$ $|(v_{i,s})|_s$ are bounded since $|v_s|_s$ are bounded. $(v_{i,s})_s \in V_{i,s}$ for $i=1,2$.
\end{proof}

For two bounded eigen orthonormalizable $(\displaystyle\prod_{s_1 \in S_1} V_{1, s_1})^b, (\displaystyle\prod_{s_2 \in S_2} V_{2, s_2})^b$, we define $\mathrm{Hom}_{\mathrm{nice}}((\displaystyle\prod_{s_1 \in S_1} V_{1, s_1})^b, (\displaystyle\prod_{s_2 \in S_2} V_{2, s_2})^b)$ to be the space of all nice maps from $(\displaystyle\prod_{s_1 \in S_1} V_{1, s_1})^b$ to $(\displaystyle\prod_{s_2 \in S_2} V_{2, s_2})^b$.
\begin{lem}
\label{limit nice}
$$\mathrm{Hom}_{\mathrm{nice}}((\displaystyle\prod_{s_1 \in S_1} V_{1, s_1})^b, (\displaystyle\prod_{s_2 \in S_2} V_{2, s_2})^b) \hookrightarrow \mathrm{Hom}_{\mathrm{cont}}((\displaystyle\prod_{s_1 \in S_1} V_{1, s_1})^b, (\displaystyle\prod_{s_2 \in S_2} V_{2, s_2})^b)$$ is a closed $A$-linear subspace.
\end{lem}
\begin{proof}
It is clear that from the definition of nice map that \[\mathrm{Hom}_{\mathrm{nice}}((\displaystyle\prod_{s_1 \in S_1} V_{1, s_1})^b, (\displaystyle\prod_{s_2 \in S_2} V_{2, s_2})^b)\] is a linear subspace. 
For closeness, let \{$f_i$\} be a sequence of nice maps between $(\displaystyle\prod_{s_1 \in S_1} V_{1, s_1})^b$ and $(\displaystyle\prod_{s_2 \in S_2} V_{2, s_2})^b$ converging to \[f \in \mathrm{Hom}_{\mathrm{cont}}((\displaystyle\prod_{s_1 \in S_1} V_{1, s_1})^b, (\displaystyle\prod_{s_2 \in S_2} V_{2, s_2})^b). \] $\imath_1: S_1 \hookrightarrow \mathbb{Z}_{\leq 0}^l$, total degree $|s_1|$ of $s_1 \in S_1$ is sum of the absolute value of coordinates of $\imath_1(s_1)$. 
For any nonzero $v \in (\displaystyle\prod_{s_1 \in S_1} V_{1, s_1})^b, s_2 \in S_2, \varepsilon>0$, there exists $N_0 >0$ such that $|f-f_n| < \varepsilon |v|^{-1}$ for all $n \geq N_0$, and there exists $N_1 > 0$ such that for any finite set \{$s_{n_1},\cdots,s_{n_m}$\} satisfying $|s_{n_k}|>N_1$, $\sum\limits_{1 \leq k \leq m} (f_{N_0}(v_{s_{n_k}}))_{s_2}<\varepsilon$ since \{$f_i$\} are nice. By non-Archimedean triangle inequality, $\sum\limits_{1 \leq k \leq m} (f(v_{s_{n_k}}))_{s_2}<\varepsilon$ for any finite set \{$s_{n_1},\cdots,s_{n_m}$\} satisfying $|s_{n_k}|>N_1$, $\sum\limits_{s_1 \in S_1} f(v_{s_1})_{s_2}$ converges for all $s_2 \in S_2$ and \begin{eqnarray*}
\sum\limits_{s_1 \in S_1} f(v_{s_1})_{s_2} &=& \lim_{n \to \infty} \sum_{s_1 \in S_1, |s_1| \leq n} f(v_{s_1})_{s_2} \\ &=& \lim_{n \to \infty} \lim_{i \to \infty} \sum_{s_1 \in S_1, |s_1| \leq n} f_i(v_{s_1})_{s_2} \\
&=& \lim_{n \to \infty} (\lim_{i \to \infty} \sum_{s_1 \in S_1} f_i(v_{s_1})_{s_2} - \lim_{i \to \infty} \sum_{s_1 \in S_1, |s_1|>n} f_i(v_{s_1})_{s_2}) \\
&=& \lim_{n \to \infty} (\lim_{i \to \infty} f_i(v)_{s_2} - \lim_{i \to \infty} \sum_{s_1 \in S_1, |s_1|>n} f_i(v_{s_1})_{s_2}) \\
&=& f(v)_{s_2} - \lim_{n \to \infty} \lim_{i \to \infty} \sum_{s_1 \in S_1, |s_1|>n} f_i(v_{s_1})_{s_2} \\
&=& f(v)_{s_2}.
\end{eqnarray*}  The last equation holds since for any $\varepsilon>0$, there exists $N_1>0$ as we chose before such that $\sum\limits_{s_1 \in S_1, |s_1|>N_1} f_i(v_{s_1})_{s_2} < \varepsilon$ independent of $f_i$.
\end{proof}

\begin{lem}
\label{nice1}
Assume $A=K$. The kernel of a $T_0$-equivariant nice $f: (\displaystyle\prod_{s_1 \in S} V_{1, s_1})^b \to (\displaystyle\prod_{s_2 \in S} V_{2, s_2})^b$ is also of the form $(\displaystyle\prod_{s \in S} V_s)^b$. Moreover, $\overline{\mathrm{Im}(f)} \subset \widehat{\overline{\mathrm{Im}(f)}}$.
\end{lem}
\begin{proof}
The kernel of $f|_{(\displaystyle\prod_{s_1 \in S} V_{1, s_1})^c}$ is of the form $(\displaystyle\prod_{s_1 \in S} \widetilde{V_{1, s_1}})^c$ by the Lem \ref{rep of t}.
Since $f$ is nice, $f(\widehat{(\displaystyle\prod_{s_1 \in S} \widetilde{V_{1, s_1}})^c})=f((\displaystyle\prod_{s_1 \in S} \widetilde{V_{1, s_1}})^b)=0$. 
Again by the niceness of $f$, $f(x) \neq 0$ for $x \not\in (\displaystyle\prod_{s_1 \in S} \widetilde{V_{1, s_1}})^b$. 
For the inclusion, it suffices to prove $$\widehat{\overline{\mathrm{Im}(f)}}=(\displaystyle\prod_{s_2 \in S} \mathrm{Im}(f)(V_{1,s_2}))^b.$$ 
We have $\mathrm{Im}(f)(V_{1, s_1}) \subset \mathrm{Im}(f)$ for any $s_1 \in S$. 
The equality follows from $\overline{\mathrm{Im}(f)} \cap V_{2, s_2}=\mathrm{Im}(f)(V_{1, s_2})$ for any $s_2 \in S$.
\end{proof}

For a finite free $T_0$-eigen orthonormalizable $A$-module $M \simeq \displaystyle\prod_{s_m \in S_M} M_{s_m}$ ($S_M$ can be chosen as a finite set), $M \otimes_A (\displaystyle\prod_{s \in S} V_s)^c$ (resp. $M \otimes_A (\displaystyle\prod_{s \in S} V_s)^b$) is again equipped with the same structure we will explain in the following.
Moreover, the Banach norm on $M \otimes_A V$ is defined to be the norm for $\mathcal{L}_A(M^\ast, V)$ for a Banach space $V$, where $V$ can be $(\displaystyle\prod_{s \in S} V_s)^c, (\displaystyle\prod_{s \in S} \tilde{V}_s)^c, V_s$ in our setting. 

More specifically, if $M \simeq \bigoplus\limits_{m \in S_M} M_m$, $V_l:=\bigoplus\limits_{\substack{m \in S_M, s \in S \\ s+m=l }} M_m \otimes V_s$,
 $$M \otimes_A (\displaystyle\prod_{s \in S} V_s)^b \simeq \{ v \in \prod_{(s+m) \in L} V_{s+m} \big| ~ \underset{s \in S, m\in S_M}{\mathrm{max}} |(v)_{s+m}|_{s+m} < \infty  \},$$ $$\bigoplus_{m \in S_M} x_m \otimes v^m \mapsto \sum_{m \in S_M} (x_m \otimes (v^m)_{l-m})_l,$$ where $L := S \times S_M / \{ (s,m) \sim (s',m') ~ \mathrm{whenever} ~ \chi_s\cdot\chi_m=\chi_{s'}\cdot\chi_{m'}\}$ and $s+m := (s,m) \in L$. 
For any $l \in L$ and $m \in S_M$, there is at most one $l-m$ in $S$ such that $(l-m,m) = l$. 
Similarly for $l \in L$ and $s \in S$, there is at most one $l-s$ in $S_M$ such that $(s,l-s) = l$. 
We can similarly define $M \otimes_A (\displaystyle\prod_{s \in S} V_s)^c$.

\begin{lem}
\label{tensor nice}
For a $A$-linear map $g: M_1 \to M_2$ between two finite free eigen orthonormalizable $A$-modules, 
if $f: (\displaystyle\prod_{s_1 \in S_1} V_{1, s_1})^b \to (\displaystyle\prod_{s_2 \in S_2} V_{2, s_2})^b$ is nice, then $$M_1 \otimes (\displaystyle\prod_{s_1 \in S_1} V_{1, s_1})^b \xrightarrow{g \otimes f} M_2 \otimes (\displaystyle\prod_{s_2 \in S_2} V_{2, s_2})^b$$ is nice.
\end{lem}
\begin{proof}
Let 
\[ M_1 \simeq \displaystyle\prod_{s_{m_1} \in S_{M_1}} M_{s_{m_1}}, ~ M_2 \simeq \displaystyle\prod_{s_{m_2} \in S_{M_2}} M_{s_{m_2}}, \] \[ M_1 \otimes_A (\displaystyle\prod_{s_1 \in S_1} V_{1, s_1})^b \simeq (\displaystyle\prod_{l_1 \in L_1} V_{l_1})^b, ~ M_2 \otimes_A (\displaystyle\prod_{s_2 \in S_2} V_{2, s_2})^b \simeq (\displaystyle\prod_{l_2 \in L_2} V_{l_2})^b. \] 
Here we use the same notation for $L_1, L_2$ as in the previous discussion before the current lemma. 

For any eigen element $m \in V_k$ for $k \in S_{M_1}$ and $\hat{v}_1 \in (\displaystyle\prod_{s_1 \in S_1} V_{1, s_1})^b$, $g(m)$ can be expressed as $g(m)=\sum\limits_{k_2 \in S_{M_2}} m_{k_2},$ where $m_{k_2} \in V_{k_2}$. We have
\begin{eqnarray*}
(g \otimes f)(m \otimes \hat{v}_1)_{l_2} & = & ((\sum\limits_{k_2 \in S_{M_2}}  m_{k_2}) \otimes f(\hat{v}_1))_{l_2} \\ & = & (\sum\limits_{k_2 \in S_{M_2}} m_{k_2} \otimes f(\hat{v}_1)_{l_2-k_2})_{l_2} \\
& = & \sum\limits_{k_2 \in S_{M_2}} m_{k_2} \otimes \sum_{s_1 \in S_1}(f|_{V_{s_1}}((\hat{v}_1)_{s_1}))_{l_2-k_2} \\
& = & \sum\limits_{k_2 \in S_{M_2}} m_{k_2} \otimes \sum_{l_1 \in L_1}(f|_{V_{l_1-k}}((\hat{v}_1)_{l_1-k}))_{l_2-k_2} \\
& = & (\sum_{l_1 \in L_1}(g \otimes f)(m \otimes (\hat{v}_1)_{l_1-k}))_{l_2}.
\end{eqnarray*} The fourth equality above is a substitution of variables from $s_1$ to $l_1-k$.
\end{proof}

\begin{lem}
\label{extend nice}
Assume $A=K$. Let $f: (\displaystyle\prod_{s_1 \in S_1} V_{1, s_1})^b \to (\displaystyle\prod_{s_2 \in S_2} V_{2, s_2})^b$ be a nice map. 
If $$W_1 \hookrightarrow (\displaystyle\prod_{s_1 \in S_1} V_{1, s_1})^b, W_2 \hookrightarrow (\displaystyle\prod_{s_2 \in S_2} V_{2, s_2})^b$$ are closed $A[T_0]$ subrepresentations inside $(\displaystyle\prod_{s_1 \in S_1} V_{1, s_1})^b, (\displaystyle\prod_{s_2 \in S_2} V_{2, s_2})^b$ \linebreak such that $f(W_1) \subset W_2$. 
Assume $W_2 \subset \widehat{W}_2$, then $f(\widehat{W_1}) \subset \widehat{W_2}$.
\end{lem}
\begin{proof}
We have $\widehat{W_1} \simeq (\displaystyle\prod_{s_1 \in S_1} W_{1, s_1})^b, \widehat{W_2} \simeq (\displaystyle\prod_{s_2 \in S_2} W_{2, s_2})^b$. 
For $\forall \hat{w}_1 \in \widehat{W_1}$, $s_2 \in S_2$ $$(f(\hat{w}_1))_{s_2}=\sum_{s_1 \in S_1}(f|_{V_{s_1}}((\hat{w}_1)_{s_1}))_{s_2}.$$ 
As $W_1 \cap (\displaystyle\prod_{s_1 \in S_1} V_{1, s_1})^c$ being convergent eigen orthonormalizable by Lem \ref{rep of t}, 
$$(\hat{w}_1)_{s_1} \in W_1 \cap (\displaystyle\prod_{s_1 \in S_1} V_{1, s_1})^c \subset W_1.$$ 
Each term $(f|_{V_{s_1}}((\hat{w}_1)_{s_1}))_{s_2} \in W_2$ by Lem \ref{l11} as $W_2 \subset \widehat{W}_2$. 
The sum converges in $(\displaystyle\prod_{s_2 \in S_2} V_{2, s_2})^b$ and bounded uniformly for all $s_2$, thus in $W_2$. $f(\hat{w}_1) \in \widehat{W_2}$.
\end{proof}

\section{Some auxiliary modules and their duals}\label{auxiliary modules}
\hfill\break
We use the same notation as in \S \ref{Notation}, \S \ref{la rep}. And let $T^0=\mathscr{T}(\cO_E)$ with $T^s \subset T^0$.
In the remaining paper we view the right $I$-modules $\Ind^s_I(V)$ and $\DD^s_I(V)$ introduced in \S \ref{la rep} as left $I$-modules in the usual way. 
Let $\mathfrak{n}, \overline{\frakn}$ be the Lie algebra of $N, \overline{N}$ over $\QQ_p$. 
All the tensor products are defined over $\QQ_p$. 
Let $A=\sO(\Omega)$ for some irreducible Zariski closed subspace of an affinoid subdomain $\Omega \subset \mathscr{W}$ ($A$ is an integral domain). 

We assume in this section $\mathscr{H}_{/\cO_E}=\GL_n$, $\mathrm{Sp}_{2n}$ or a unitary group associated to $J_n$ and an unramified extension $K/E$. For an Iwahori subgroup of $\mathscr{H}(\cO_E)$, 
$$I=\Nbar_I \times T^0 \times N_I,$$ let $\Delta^+_{H}, \Delta^-_{H}$ denote positive roots and negative roots respectively viewed over $E$. Every factor of a product of different $I$ satisfies the assumption for $\mathscr{H}$ as above. In particular, for each Iwahori factor, we use the coordinates of $\Nbar_I^s$ as in the beginning of \S \ref{la rep} which will be constructed for each case of $\mathscr{H}$ in \S \ref{analysis}.

Assume that $A=\sO(\Omega)$ is an $E$-Banach algebra containing a large enough $p$-adic field such that $A$ contains all Galois conjugates of $E$ and the splitting field of $H$.
For later applications, we introduce a Banach space of functions on $\Nbar_I \simeq \cO_E^d$. 
For $x \in \Nbar_I$, we use $x_1,\cdots,x_d$ to denote the coordinates of $x$ via the isomorphism above.
A basis of \[ C^{s,1}_{I,\chi}:=C^{s,1}(I, A) \subset C^{s,an}(\Nbar_I, A)=\Ind^s_{I,\chi}\] as representation of $I$ for this case is given by  \[
f^\sigma_{g, \underline{a}^\sigma}(x):= \begin{cases} 
      (\sigma x_1)^{a_1^\sigma}\cdots (\sigma x_d)^{a_d^\sigma} & x \in g\cdot \Nbar_I^s \\
      0 & x \not\in g\cdot \Nbar_I^s,
   \end{cases}
\] for $\sigma$ running through embedding from $E$ to $\overline{\QQ}_p$ and $\underline{a}^\sigma$ running through $\ZZ^{d}_{\geq 0}$. 
Recall that there is a partial order relation $\underline{a} \geq \underline{a'}$ if and only if $a_i \geq a'_{i}$ for all $i$. 
$C^{s,1}(\Nbar_I, A)$ is defined to be the subspace of $C^{s,an}(\Nbar_I, A)$ whose elements can be expressed as the sum of $f^\sigma_{g,\underline{a}^\sigma}$ with coefficients tending to zero.
Or equivalently, $C^{s,1}(\Nbar_I, A)$ consists of (globally) analytic functions on $\Nbar_I$ multiplied by characteristic functions on $\Nbar^s_I$-cosets.

In the case of unitary group for an unramified quadratic extension $K/E$, there is another basis which is eigenbasis for $T^s$ action with respect to the identification $\Nbar_I \simeq \cO_E^{d_e} \times \cO_K^{d_k}$ which is equivalent to the basis above by
\[
f^{\sigma}_{g, \underline{a}^\sigma_E, \underline{a}^\sigma_K}(x):= \begin{cases} 
      (\sigma x_E^1)^{a_{E,1}^{\sigma}}\cdots (\sigma x_E^{d_e})^{a_{E,d_e}^{\sigma}}\cdot (\sigma x_K^1)^{a_{K,1}^{\sigma}}\cdots (\sigma x_K^{d_k})^{a_{K,d_k}^{\sigma}} & x \in g\cdot \Nbar_I^s \\
      0 & x \not\in g\cdot \Nbar_I^s,
   \end{cases}
\] 
for $\sigma$ running through embedding from $K$ to $\overline{\QQ}_p$ and $\underline{a}^\sigma_E=(a_{E,1}^\sigma,\cdots,a_{E,d_e}^\sigma)$ running through $\ZZ^{d_e}_{\geq 0}$, $\underline{a}^\sigma_K=(a_{K,1}^\sigma,\cdots,a_{K,d_k}^\sigma)$ running through $\ZZ^{d_k}_{\geq 0}$. 

To save spaces, we introduce the following notation:
$$\delta^{E+}_i=(\delta^{E-}_i)^{-1},\delta^{K+}_i=(\delta^{K-}_i)^{-1},$$
$$\delta^{E\pm}(t)x_E:=(\delta^{E\pm}_1(t)x_E^1, \cdots, \delta^{E\pm}_{d_e}(t)x_E^{d_e}),$$ $$\delta^{K\pm}(t)x_K:=(\delta^{K\pm}_1(t)x_K^1, \cdots, \delta^{K\pm}_{d_e}(t)x_K^{d_k}),$$ $$\sigma\delta^{E\pm}(t):=\sigma\delta^{E\pm}_1(t)^{a_{E,1}^\sigma}\cdots \sigma\delta^{E\pm}_{d_e}(t)^{a_{E,d_e}^\sigma},$$ $$\sigma\delta^{K\pm}(t):=\sigma\delta^{K\pm}_1(t)^{a_{K,1}^\sigma}\cdots \sigma\delta^{K\pm}_{d_k}(t)^{a_{K,d_k}^\sigma}.$$

  $$t\cdot f^{\sigma}_{g, \underline{a}^\sigma_E, \underline{a}^\sigma_K}=\chi(t)^{-1} \cdot \sigma\delta^{E+}(t) \cdot \sigma\delta^{K+}(t) \cdot f^{\sigma}_{g, \underline{a}^\sigma_E, \underline{a}^\sigma_K}$$ since for any $x \in g\cdot \Nbar_I^s$ and any $t \in T^s$, 
\begin{eqnarray*}
t \cdot f^{\sigma}_{g, \underline{a}^\sigma_E, \underline{a}^\sigma_K}(x)  & = & f^\sigma_{g, \underline{a}^\sigma}(t^{-1}x) \\
 & = & f^{\sigma}_{g, \underline{a}^\sigma_E, \underline{a}^\sigma_K}((t^{-1}xt)t^{-1}) \\
 & = & \chi(t)^{-1}\cdot f^{\sigma}_{g, \underline{a}^\sigma_E, \underline{a}^\sigma_K}(t^{-1}xt) \\
 & = & \chi(t)^{-1}\cdot f^{\sigma}_{g, \underline{a}^\sigma_E, \underline{a}^\sigma_K} (\delta^{E-}(t^{-1})x_E,\delta^{K-}(t^{-1})x_K) \\
 & = & \chi(t)^{-1}\cdot f^{\sigma}_{g, \underline{a}^\sigma_E, \underline{a}^\sigma_K}(\delta^{E+}(t)x_E,\delta^{K+}(t)x_K)) \\
 & = & \chi(t)^{-1} \cdot \sigma\delta^{E+}(t) \cdot \sigma\delta^{K+}(t) \cdot f^{\sigma}_{g, \underline{a}^\sigma_E, \underline{a}^\sigma_K}(x).
 \end{eqnarray*}
A general $t \in T^0$ permutes different $f^{\sigma}_{g, \underline{a}^\sigma_E, \underline{a}^\sigma_K}$ and $f^{\sigma}_{g', \underline{a}^\sigma_E, \underline{a}^\sigma_K}$ for $g'=tgt^{-1}$ up to a scalar twist.

\begin{lem}\label{equ basis}
The bases $\{f^\sigma_{g, \underline{a}^\sigma}\}$ and $\{f^{\sigma}_{g, \underline{a}^\sigma_E, \underline{a}^\sigma_K}\}$ generate the same subspace in $C^{s,an}(\Nbar_I,A)$. 
\end{lem}

\begin{proof}
Suppose $\imath_{K/E}$ gives rise to $\cO_K = \cO_E u_1 \oplus \cO_E u_2$ for $u_1, u_2 \in \cO_K$. By transitions of coordinates introduced in \ref{la rep}, $x_K^i=u_1\cdot x_{d_e+2i-1}+u_2 \cdot x_{d_e+2i}$, $\sigma x_K^i=\sigma u_1 \cdot \sigma x_{d_e+2i-1} + \sigma u_2 \cdot \sigma x_{d_e+2i}$. Moreover, there are two lifts of embedding $\tilde{\sigma}_1, \tilde{\sigma}_2: K \hookrightarrow \overline{\QQ}_p$ for each $\sigma: E \hookrightarrow \overline{\QQ}_p$. Then
$$\begin{pmatrix} \tilde{\sigma}_1 x_K^i \\ \tilde{\sigma}_2 x_K^i
\end{pmatrix}=\begin{pmatrix}
\tilde{\sigma}_1 u_1 & \tilde{\sigma}_1 u_2 \\ \tilde{\sigma}_2 u_1 & \tilde{\sigma}_2 u_2
\end{pmatrix}\cdot\begin{pmatrix} \sigma x_{d_e+2i-1} \\ \sigma x_{d_e+2i}
\end{pmatrix},$$ $\begin{pmatrix}
\tilde{\sigma}_1 u_1 & \tilde{\sigma}_1 u_2 \\ \tilde{\sigma}_2 u_1 & \tilde{\sigma}_2 u_2
\end{pmatrix}^{-1} \in \GL_2(\cO_K)$ since $K/E$ is unramified.  
Change of variables says that $f^{\sigma}_{g, \underline{a}^\sigma_E, \underline{a}^\sigma_K}$ (resp. $f^\sigma_{g, \underline{a}^\sigma}$) is the sum of the first basis $\{f^\sigma_{g, \underline{a}^\sigma}\}$ (resp. $\{f^{\sigma}_{g, \underline{a}^\sigma_E, \underline{a}^\sigma_K}\}$) with integral coefficients. 
\end{proof}

For Iwahori subgroup $I$ of a finite product of $p$-adic groups $\prod\limits_{1 \leq i \leq l} H_i(E_i)$ considered in \S \ref{Notation}, let $\chi: T^0=\prod\limits_{1 \leq i \leq l} T^0_i \to A^\times$ be a locally analytic character, where $A=\sO(\Omega)$, $\Omega$ is an irreducible Zariski closed subspace of an affinoid subdomain of the weight space of $T^0$. 
And recall that $\chi_{\Omega}:T^0 \to A^\times$ is the universal character of $T^0$ for $\Omega$. 
We keep assuming that $A$ contains all Galois conjugates of $E_i$ and the splitting field of $H_i$ for $1 \leq i \leq l$.

We use variables $x^i=(x^i_1,\cdots,x^i_{d_i})$ to denote coordinates for $\psi^1_{I_i}$ introduced in \S \ref{la rep}.
$C^{s,1}(\Nbar_I, A)$ is defined to be the subspace of $C^{s,an}(\Nbar_I, A)$ whose elements can be expressed as sum of $ 
f^{\underline{\sigma}}_{g,\underline{a}^{\underline{\sigma}}}$ with coefficients tending to zero, indexed by $g \in \Nbar_I^1 / \Nbar_I^s \simeq \prod\limits_{1 \leq i \leq l} \Nbar_{I_i}^1 / \Nbar_{I_i}^s$, $\underline{\sigma}=(\sigma_1,\cdots,\sigma_l)$ running through all possible embeddings $\prod\limits_{1 \leq i \leq l} E_i \to (\overline{\QQ}_p)^l$, $\underline{a}^{\underline{\sigma}}=\prod\limits_{1 \leq i \leq l} \underline{a}^{\sigma_i}$ running through $\prod\limits_{1 \leq i \leq l} \ZZ_{\geq 0}^{d_i}$ and 
\[
f^{\underline{\sigma}}_{g, \underline{a}^{\underline{\sigma}}}(x):= \begin{cases} 
    \prod\limits_{1 \leq i \leq l}  (\sigma_i x^i_1)^{a_1^{\sigma_i}}\cdots (\sigma_i x^i_{d_i})^{a_{d_i}^{\sigma_i}} & x \in g\cdot \Nbar_I^s \\
      0 & x \not\in g\cdot \Nbar_I^s.
   \end{cases}
\]
Like $l=1$ case, there is another equivalent basis which is eigenbasis for $T^s=\prod\limits_{1 \leq i \leq l} T^s_i$ action by
\[
f^{\sigma}_{g, \underline{a}^\sigma_E, \underline{a}^\sigma_K}(x):= \begin{cases} 
    \prod\limits_{1 \leq i \leq l}  (\sigma x_{E_i}^1)^{a_{E_i,1}^{\sigma}}\cdots (\sigma x_{E_i}^{d_{e,i}})^{a_{E_i,d_{e,i}}^{\sigma}} & \\
    \cdot \prod\limits_{1 \leq i \leq l} (\sigma x_{K_i}^1)^{a_{K_i,1}^{\sigma}}\cdots (\sigma_{K_i} x_{K_i}^{d_{k,i}})^{a_{K_i,d_{k_i}}^{\sigma}} & x \in g\cdot \Nbar_I^s \\
      0 & x \not\in g\cdot \Nbar_I^s,
   \end{cases}
\] 
for $\sigma$ running through embedding from $K$ to $\overline{\QQ}_p$ and $\underline{a}^\sigma_E= \prod\limits_{1 \leq i \leq l} (a_{E_i,1}^\sigma,\cdots,a_{E_i,d_{e,i}}^\sigma)$ running through $\prod\limits_{1 \leq i \leq l} \ZZ^{d_{e,i}}_{\geq 0}$, $\underline{a}^\sigma_K= \prod\limits_{1 \leq i \leq l} (a_{K_i,1}^\sigma,\cdots,a_{K_i,d_{k,i}}^\sigma)$ running through $\ZZ^{d_{k,i}}_{\geq 0}$.
We make the following notation:
$$\sigma\delta^{E\pm}(t):=\prod\limits_{1 \leq i \leq l} \sigma\delta^{E_i+}_1(t)^{a_{E_i,1}^\sigma}\cdots \sigma\delta^{E_i+}_{d_{e,i}}(t)^{a_{E,d_{e,i}}^\sigma},$$ $$\sigma\delta^{K\pm}(t):=\prod\limits_{1 \leq i \leq l} \sigma\delta^{K_i+}_1(t)^{a_{K_i,1}^\sigma}\cdots \sigma\delta^{K_i+}_{d_{k,i}}(t)^{a_{K_i,d_{k,i}}^\sigma}.$$ Then similar calculations show
$$t\cdot f^{\sigma}_{g, \underline{a}^\sigma_E, \underline{a}^\sigma_K}=\chi(t)^{-1} \cdot \sigma\delta^{E+}(t) \cdot \sigma\delta^{K+}(t) \cdot f^{\sigma}_{g, \underline{a}^\sigma_E, \underline{a}^\sigma_K}$$ for any $x \in g \cdot \Nbar_I^s$ and any $t \in T^s=\prod\limits_{1 \leq i \leq l} T^s_i$.
The proof of Lem \ref{equ basis} shows that these two bases define the same subspace $C^{s,1}(\Nbar_I,A)$.
For an analytic Banach $\mathscr{O} (\Omega)[B_I]$-module $V$,
\begin{eqnarray*}
\mathrm{Ind}_{I}^{s} (V) & \simeq & C^{s,an}(\Nbar_I, V)\\
f & \mapsto & f|_{\Nbar_I}.\end{eqnarray*} 

Similarly we define $$C^{s,1}(\Nbar_I, V):=C^{s,1}(\Nbar_I, A) \hat{\otimes}_{A} V \subset C^{s,an}(\Nbar_I, A) \hat{\otimes}_{A} V \simeq C^{s,an}(\Nbar_I, V). $$ 
If $V$ is finite free eigen orthonormalizable with eigenbasis $e_1,\cdots, e_m$ for $T^0$ (notions of being convergent or bounded eigen orthonormalizable coincide for finite rank case), we will see in Prop \ref{CS1} this auxiliary module $C^{s,1}(\Nbar_I, V)$ has the advantageous structure being convergent eigen orthonormalizable in Def \ref{cons} with the basis $f^\sigma_{g,\underline{a}^\sigma} \otimes e_i$. 
In particular, $C^{s,1}(\Nbar_I, V)$ is orthonormalizable in the sense of \cite{Buz07} with the basis $f^\sigma_{g,\underline{a}^\sigma} \otimes e_i$.

Pick a Weyl group element $w=(w_1,\cdots,w_l) \in \prod\limits_{1 \leq i \leq l} W(H_i, S_i)$ and set \[ N_w:=w I w^{-1} \cap \prod\limits_{1 \leq i \leq l} N(E_i), ~~ \overline{N}_w:= w I w^{-1} \cap \prod\limits_{1 \leq i \leq l} \overline{N}(E_i). \] 
For the $wIw^{-1}$-representation $\Ind^s_{w,\chi} = \Ind^s_{wIw^{-1}}(\chi) \simeq C^{s,an}(w\Nbar_I w^{-1})$ defined in \S \ref{la rep}, we can similarly define \[ C^{s,1}_{w,\chi} := C^{s,1}(w\Nbar_I w^{-1}, A) \subset C^{s,an}(w\Nbar_I w^{-1}, A), \]
with orthonormalizable basis $f^{\sigma}_{g, \underline{a}^\sigma_E, \underline{a}^\sigma_K}$ with respect to $w\Nbar_I w^{-1} \simeq \cO_E^{d_e} \times \cO_K^{d_k}$ and indexed by $g \in w\Nbar_I w^{-1} / w\Nbar_I^s w^{-1} \simeq \Nbar_I / \Nbar_I^s$, embeddings $\sigma$ and degrees $\underline{a}^\sigma_E, \underline{a}^\sigma_K$.

 We use $\underline{\sigma x}^w(t)$ to denote $$\prod\limits_{1 \leq i \leq r}  \sigma_i(\delta^{E_i+}_1)^{w_i}(t_i)^{a_1^{\sigma_i}} \cdots \sigma_i(\delta^{E_i+}_{d_{e,i}})^{w_i}(t_i)^{a_d^{\sigma_i}} \cdot \prod\limits_{1 \leq i \leq r}  \sigma_i(\delta^{K_i+}_1)^{w_i}(t_i)^{a_1^{\sigma_i}} \cdots \sigma_i(\delta^{K_i+}_{d_{k,i}})^{w_i}(t_i)^{a_d^{\sigma_i}}.$$ 
 
Then similar calculations show
  $$t\cdot f^{\sigma}_{g, \underline{a}^\sigma_E, \underline{a}^\sigma_K}=\chi^w(t)^{-1} \cdot\underline{\sigma x}^w(t) \cdot f^{\sigma}_{g, \underline{a}^\sigma_E, \underline{a}^\sigma_K}$$ for any $x \in g\cdot w\overline{N_I}^{s}w^{-1}$ and any $t \in T^s$.

\begin{prop}
\label{CS1}
Let $V$ be a finite free eigen orthonormalizable Banach $A[B_I]$-module. 
The $A[T^s]$-module $C^{s,1}(\Nbar_I, V)$ is convergent eigen orthonormalizable, i.e. $C^{s,1}(\Nbar_I, V)$ is isomorphic to a $(\displaystyle\prod_{s \in S} V_s)^c$ in Def \ref{cons} as Banach $A[T^s]$ module with endowed norm. 
There is a natural continuous embedding of $$C^{s,1}(\Nbar_I, V) \hookrightarrow C^{s,an}(\Nbar_I, V)$$  of Banach $A$-modules with dense image.
\end{prop}
We defer the proof to section \S \ref{analysis}.

 \begin{rem}
We dualize $C^{s,1}(\Nbar_I,V)$ to get $\DD^{s,1}_I(V):=\cL_A(C^{s,1}(\Nbar_I,V), A)$. 
As the image of $C^{s,1}(\Nbar_I, V)$ in $C^{s,an}(\Nbar_I, V)$ is dense by Prop \ref{CS1}, we have a continuous embedding of $\DD^s_I(V) \hookrightarrow \DD^{s,1}_I(V)$. 
In particular, $$\DD_{w, \chi}^s \hookrightarrow \DD_{w, \chi}^{s,1}:=\cL_A(C^{s,1}_{w,\chi}, A).$$
The motivation for introducing $\DD_{w, \chi}^{s,1}$ is that $\DD_{w, \chi}^s$ itself is not bounded eigen orthonormalizable, and $\DD_{w, \chi}^{s,1}$ can be regarded as a completion of $\DD_{w, \chi}^s$ with respect to $T^0$.
\end{rem}
\begin{prop}\label{D1 beo}
Let $A=\sO(\Omega)$ in the Def \ref{cons}. The $A[T]$-module $\DD_{w, \chi}^{s,1} \simeq (\displaystyle\prod_{s \in S} V_s)^b$ for some $S \subset \mathbb{Z}^m_{\leq 0}$ is bounded eigen orthonormalizable as $T$-unitary Banach representation.  \end{prop}
\begin{proof}
Note that $\DD_{w, \chi}^{s,1}  =  \mathcal{L}_A(\mathrm{Ind}_{w, \chi}^{s,1}, A)$ and $\mathrm{Ind}_{w, \chi}^{s,1}  \simeq  C^{s,1}_{w,\chi} \subset C^{s,an}(w\Nbar_Iw^{-1}, A)$, we see from proposition \ref{CS1} $\mathrm{Ind}^{s,1}_{w, \chi}$ is of the form $(\displaystyle\prod_{s \in S} V_s)^c$. Now the conclusion follows from the Lem \ref{dual}.
\end{proof}
 We give a basis of $\DD_{w, \chi}^{s,1}$ by 
$$f^{\vee,\sigma}_{g, \underline{a}^\sigma_E, \underline{a}^\sigma_K}(f^{\sigma'}_{g', \underline{a'}^{\sigma'}_E, \underline{a'}^{\sigma'}_K}) := \begin{cases}
1 & g^{-1}g' \in w\Nbar_I^{s}w^{-1}, \sigma=\sigma' ~ \mathrm{and} ~ \underline{a}^\sigma_E=\underline{a'}^{\sigma'}_E, \underline{a}^\sigma_K=\underline{a'}^{\sigma'}_K\\
0 & \mathrm{otherwise}.
\end{cases}$$ 
And by the Lem \ref{dual}, $\DD_{w, \chi}^{s,1}$ is bounded eigen orthonormalizable, i.e., an element of $\DD_{w, \chi}^{s,1}$ can be expressed as a sum of $f^{\vee,\sigma}_{g, \underline{a}^\sigma_E, \underline{a}^\sigma_K}$ with bounded coefficients. And
\begin{eqnarray*} t \cdot f^{\vee,\sigma}_{g, \underline{a}^\sigma_E, \underline{a}^\sigma_K}=\chi^w(t) \cdot \underline{\sigma x}^w(t)^{-1} \cdot f^{\vee,\sigma}_{g, \underline{a}^\sigma_E, \underline{a}^\sigma_K} \end{eqnarray*} for $t \in T^s$.

\begin{prop}
\label{weights in Ds1}
The weights appearing in $\mathbb{D}_{w,\chi}^{s,1}$ 
are exactly those of the form $\chi^w \cdot \mu^-$, where $\mu^-$ is an algebraic nonpositive weight with respect to $\prod\limits_{1 \leq i \leq l} w_i\cN_{B_i}w_i^{-1}$, i.e. $\mu^-$ can be written as a nonnegative integral linear combination of negative roots in $\sqcup_{1 \leq i \leq l} w_i\Delta^-_{H_i,p}$ with respect of the Borel pairs $(w_i\cB_i w_i^{-1}, w_i\cT_i w_i^{-1})$ and 
$w_i\Delta_{H_i,p} = w_i\Delta^+_{H_i,p}  \sqcup w_i\Delta^-_{H_i,p}$ for $1 \leq i \leq l$. 
Moreover, the multiplicity of $\chi^w$ ($\mu^- $ trival) in $\mathbb{D}_{w,\chi}^{s,1}$ equals to the index of $w\Nbar_I^sw^{-1} \trianglelefteq w\Nbar_Iw^{-1}= \prod\limits_{1 \leq i \leq l} |\overline{\mathscr{N}}_{B_{i}}(\cO_{E_{i}}/\varpi_{i}^{s-1} \cO_{E_i})|$. 
Especially the highest weights correspond to $\prod\limits_{1 \leq i \leq r} \overline{\mathscr{N}}_{B_{v_i}}(\cO_{F^+_{v_i}}/\varpi_{v_i}^{s-1} \cO_{F^+_{v_i}})$.
\end{prop}
\begin{proof}
The formula before Prop \ref{weights in Ds1} combined with Lem \ref{eigenweight} gives the proof. The highest weights correspond to $\underline{a}=\vec{0}$ and arbitrary $g \in w\Nbar_Iw^{-1}/w\Nbar_I^{s}w^{-1} \simeq \Nbar_I/\Nbar_I^{s} \simeq \prod\limits_{1 \leq i \leq l} |\overline{\mathscr{N}}_{B_{i}}(\cO_{E_{i}}/\varpi_{i}^{s-1} \cO_{E_i})|$.
\end{proof}

\begin{prop}
\label{aux mod}
$C^{s,1}(\Nbar_I, V)$ is $I^s_0$-stable, and it is a locally analytic $I^s_0$ representation. 
$\DD^{s,1}_I(V)$ is a continuous $D(I^s_0, \QQ_p)$ module, and $\DD_I^s(V) \hookrightarrow \DD_I^{s,1}(V)$ is $I^s_0$ equivariant. 
Moreover, any group element of $I^s_0$ induces a nice automorphism of $\DD_I^{s,1}(V)$ in the sense of Def \ref{nice}. 
In particular, these properties hold for $C^{s,1}_{w,\chi}$ as a dense $wI^{s,\circ}w^{-1}$-stable submodule of $C^{s,an}_{w,\chi}$ and $\DD_{w, \chi}^s \hookrightarrow \DD_{w, \chi}^{s,1}$. 
\end{prop}
We leave the proof to section \S \ref{analysis}, which reduces to the $\GL_n$ case.

\section{Some $p$-adic Banach analysis statements}\label{analysis}
In this section we include some statements about $p$-adic Banach analysis, mostly in terms of analysing coefficients and coordinates. 

We start from proving the assumptions used in the beginning of \S \ref{la rep} for $H_{/ E}$. We call this choice of coordinates \emph{standard}. In all the cases, our chosen Borel subgroup and Iwahori subgroup are implicit.

When $H=\GL_n$: coordinates of $\Nbar_I^s \times T^0$ correspond to entries of lower triangular matrices of radius $|\varpi|^s$ and diagonal entries of standard $n$-tuples $\cO_E^\times$.

When $H=\mathrm{Sp}_{2n}$: we use \begin{eqnarray*} \psi_T : (\cO_E^\times)^n \simeq T^0 & \hookrightarrow & I \hookrightarrow \GL_{2n}(\cO_E) \\
(x_1^\ast,\cdots,x_n^\ast) & \mapsto & \mathrm{diag}(x_1^\ast,\cdots,x_n^\ast,(x_1^\ast)^{-1},\cdots,(x_n^\ast)^{-1}) \end{eqnarray*} to denote a set of standard coordinates on $T^0$. We embed $\mathscr{H}(\cO_E) \hookrightarrow \GL_{2n}(\cO_E)$ such that  
\begin{eqnarray*}
 \psi^s_I \times \psi_T : \cO_E^d \times (\cO_E^\times)^n \xrightarrow{\sim} \Nbar_I^s \times T^0 & \hookrightarrow & \GL_{2n}(\cO_E) \\
 (x_1,\cdots,x_{n^2}) \times (x_1^\ast,\cdots,x_n^\ast) & \mapsto & \begin{pmatrix} 
 D & 0 \\
 \varpi^s X & D^{-1}
 \end{pmatrix}, \\
 \mathrm{where} ~~ X & = & \begin{pmatrix}
 x_1 & \cdots & x_n \\
 \vdots & \ddots & \vdots \\
 x_{n^2-n+1} & \cdots & x_{n^2}
 \end{pmatrix} \\ D & = & \mathrm{diag}(x_1^\ast,\cdots,x_n^\ast).
 \end{eqnarray*} It is clear that there is a projection $p_{\GL}^s$ realizing the embedding $i_{\GL}^s: \Nbar_I^s \hookrightarrow \Nbar_{I_{\GL_{2n}}}^s$ as a section of it with respect to the standard coordinates of $\Nbar_{I_{\GL_{2n}}}^s$ by simply forgetting coordinates of lower triangular parts of both top left $n \times n$ block and bottom right $n \times n$ block.
 
When $H$ is a unitary group defined by an unramified quadratic extension $K/E$ (Gal($K/E$)$\simeq\{1,c\}$) and Hermitian form 
\[
J_n = \begin{pmatrix}
0 & \Psi_n \\
-\Psi_n & 0
\end{pmatrix}, 
\] where $\Psi_n$ is the matrix with $1$’s on the anti-diagonal and $0$’s elsewhere: we use \begin{eqnarray*} \psi_T : (\cO_K^\times)^n \simeq T^0 & \hookrightarrow & I \hookrightarrow \GL_{2n}(\cO_K) \\
(x_1^\ast,\cdots,x_n^\ast) & \mapsto & \mathrm{diag}(x_1^\ast,\cdots,x_n^\ast,c(x_n^\ast)^{-1},\cdots,c(x_1^\ast)^{-1}) \end{eqnarray*} to denote a set of standard coordinates on $T^0$. We embed $\mathscr{H}(\cO_E) \hookrightarrow \GL_{2n}(\cO_K)$ such that 
\begin{eqnarray*}
 \psi^s_I \times \psi_T : \cO_E^d \times (\cO_K^\times)^n & \xrightarrow{\sim} & \Nbar_I^s \times T^0 \hookrightarrow \GL_{2n}(\cO_K) \\
 \underline{x} \times (x_1^\ast,\cdots,x_n^\ast) & \mapsto & \begin{pmatrix} 
 X & 0 \\
 X' \Psi_n X & \Psi_n ({}^{t}X^c)^{-1} \Psi_n
\end{pmatrix}, \\
 \mathrm{where} ~~ \underline{x} & = & \left((x_E^1,\cdots,x_E^n),(x_K^1,\cdots,x_K^{n(n-1)})\right) \\
  X & = & D + \varpi^s\cdot\begin{pmatrix}
 0 & 0 & \cdots & 0 \\
 x_K^{\frac{n(n-1)}{2}+1} & 0 & \cdots & 0 \\
 \vdots & \vdots & \ddots & \vdots \\
 x_K^{n^2-2n+1} & x_K^{n^2-2n+2} & \cdots & 0
 \end{pmatrix},  \\
 D & = & \mathrm{diag}(x_1^\ast,\cdots,x_n^\ast) \\
 X' & = & \varpi^s(X'_d +X'_{ut} + {}^tX'^c_{ut}), \\
 \mathrm{where} ~~ X'_d & = & \mathrm{diag}(x_E^1,\cdots,x_E^n), \\
 X'_{ut} & = &\begin{pmatrix}
 0 & x_K^1 & \cdots & x_K^{n-2} & x_K^{n-1} \\
 0& 0 & \cdots & x_K^{2n-4} & x_K^{2n-3} \\
 \vdots & \vdots & \ddots & \ddots & \vdots \\
 0 & 0 & \cdots & 0 & x_K^{\frac{n(n-1)}{2}} \\
0 & 0 & \cdots & 0 & 0
 \end{pmatrix}.
 \end{eqnarray*}
Choose any splitting $p_{K/E}: \cO_K \to \cO_E$, we use $\bar{t}$ to denote a map from $n \times n$ matrices over $\cO_K$ to $n \times n$ Hermitian matrices over $\cO_E$  \begin{eqnarray*}
\bar{t}: M_{n \times n}(\cO_K) & \twoheadrightarrow & H_{n \times n}(\cO_K/\cO_E):=\{M \in M_{n \times n}(\cO_K) | M={}^tM^c\} \\
\begin{pmatrix} x_{11} & \cdots & x_{1n} \\ \ddots & \vdots & \ddots \\ x_{n1} & \cdots & x_{nn} \\ \end{pmatrix} & \mapsto & \begin{pmatrix} p_{K/E}(x_{11}) & \cdots & x_{1n} \\ \ddots & \vdots & \ddots \\ c(x_{1n}) & \cdots & p_{K/E}(x_{nn}) \\ \end{pmatrix},
\end{eqnarray*} i.e., $M \mapsto \bar{M}$ such that $\bar{M}_{ij}=\begin{cases} M_{ij} & i<j \\ p_{K/E}(M_{ij}) & i=j \\ c(M_{ji}) & i>j \end{cases}$.
We construct a projection $p_{\GL}^s$ realizing the embedding $i_{\GL}^s: \Nbar_I^s \hookrightarrow \Nbar_{I_{\GL_{2n}}}^s$ as a section with respect to the standard coordinates of $\Nbar_{I_{\GL_{2n}}}^s$ as follows:
\begin{eqnarray*}
p_{\GL}^s: \cO_E^{d'} \xrightarrow{\sim} \Nbar_{I_{\GL_{2n}}}^s & \twoheadrightarrow & \Nbar_I^s \xrightarrow{(\psi^s_I)^{-1}} \cO_E^d \\
I_{2n} + \begin{pmatrix} A & 0 \\ B & C \end{pmatrix} & \mapsto & \begin{pmatrix} A & 0 \\ \bar{t}(BA^{-1}\Psi_n) & \Psi_n({}^tA^c)^{-1}\Psi_n \end{pmatrix}.
\end{eqnarray*} We see coordinate functions of both $i_{\GL}^s$ and $p_{\GL}^s$ are polynomials of \hfill\break $\{x_E^1,\cdots,x_E^n, x_K^1,\cdots,x_K^{n(n-1)}, c(x_K^1),\cdots, c(x_K^{n(n-1)})\}$ with integral coefficients, hence polynomials of variables for $\cO_E^d$.
 
We now give an explicit basis of $C^{s,an}(\Nbar_I,A)$ defined in \ref{la rep}. 
Suppose $\imath_E$ gives rise to $\cO_E = \bigoplus\limits_{1 \leq i \leq [E:\QQ_p]} \ZZ_p u_i$, where $u_1,\cdots,u_{[E:\QQ_p]} \in \cO_E$. 
For each $g \in \Nbar_I/\Nbar_I^s$, we choose a lift of it $\tilde{g} \in \Nbar_I$, $$(\tilde{g}_1,\cdots,\tilde{g}_d):=(\psi^1_{I})^{-1}(\tilde{g}).$$ 
We use $\cG \subset \Nbar_I$ to denote this set of representatives of $\Nbar_I/\Nbar_I^s$ with $\tilde{\id}=\id$.
Moreover, $\psi^1_I: (\tilde{g}_1+\varpi^{s-1}\cO_E,\cdots,\tilde{g}_d+\varpi^{s-1}\cO_E) \xrightarrow{\sim} g\cdot \Nbar_I^s$ through calculations for all cases. For $x \in g \cdot \Nbar_I^s$, 
we change variables and let $$(\psi^1_I)^{-1}(x)=(\tilde{g}_1+\varpi^{s-1}x_1,\cdots,\tilde{g}_d+\varpi^{s-1}x_d),$$ $$x_i=u_1\cdot z_{(i-1)[E:\QQ_p]+1}+\cdots+u_{[E:\QQ_p]}\cdot z_{i[E:\QQ_p]} ~\mathrm{for} ~ 1 \leq i \leq d.$$ 
\begin{rem}\label{ano basis}
We claim that a basis of $C^{s,an}(\Nbar_I, A)$ is given by
$$f^{\circ}_{g,\underline{a}}(x):= \begin{cases} 
      z_1^{a_1}\cdots z_{d[E:\QQ_p]}^{a_{d[E:\QQ_p]}} & x \in g\cdot \Nbar_I^s \\
      0 & x \not\in g\cdot \Nbar_I^s,
   \end{cases}$$ which is indexed by $g \in \Nbar_I/ \Nbar_I^s$ and $\underline{a} \in \ZZ_{\geq 0}^{d[E:\QQ_p]}$.
\end{rem}
   
\begin{proof}   
The rigid analytic functions $C^{an}(\Nbar_I^s, A)$ on $\Nbar_I$ extend to $s$-analytic functions on $\Nbar_I$, $C^{an}(\Nbar_I^s, A) \hookrightarrow C^{s,an}(\Nbar_I, A)$ by defining their values to be $0$ on any other points. 
Then the following decomposition follows from definition: $$C^{s,an}(\Nbar_I, A)=\bigoplus_{g \in \Nbar_I/\Nbar_I^s} g\cdot C^{an}(\Nbar_I^s, A).$$ 
It suffices to argue that for each $g \in \Nbar_I/\Nbar_I^s$, $\{f^\circ_{g,\underline{a}}, \underline{a} \in \ZZ_{\geq 0}^{d[E:\QQ_p]}\}$ is also an orthonormal basis of $g \cdot C^{an}(\Nbar_I^s, A)$. 
The $\cO_E$ coordinates $$(x_1',\cdots,x_d')=(\psi^s_I)^{-1}\left(g^{-1}\psi^1_I(\tilde{g}_1+\varpi^{s-1}x_1,\cdots,\tilde{g}_d+\varpi^{s-1}x_d)\right)$$ of $g\cdot \Nbar_I^s$ correspond to translation of coordinates on $\Nbar_I^s$ by $g$ in terms of $(x_1,\cdots,x_d)$. 
$$x'_i=u_1\cdot z'_{(i-1)[E:\QQ_p]+1}+\cdots+u_{[E:\QQ_p]}\cdot z'_{i[E:\QQ_p]} ~\mathrm{for} ~ 1 \leq i \leq d,$$ a basis for $g\cdot C^{an}(\Nbar_I^s, A)$ is given by monomials of $z'_1,\cdots,z'_{d[E:\QQ_p]}$. 
Translation function of $(\psi^1_I)^{-1} \circ \psi^s_I$ is multiplication by $\varpi^{s-1}$. 
Direct calculations from explicit coordinates for all cases express $(x'_1,\cdots,x'_d)$ (resp. $(x_1,\cdots,x_d)$) in terms of polynomials of $(x_1,\cdots,x_d)$ (resp. $(x'_1,\cdots,x'_d)$) with integral coefficients without constant terms. 
The same statement for transition functions between monomials of $z'_1,\cdots,z'_{d[E:\QQ_p]}$ and $z_1,\cdots,z_{d[E:\QQ_p]}$ holds as well. 
\end{proof}

\begin{proof}[Proof of Proposition \ref{CS1}]
Let $e_1,\cdots,e_m$ be an eigenbasis of $V$ for $T^0$ as free $A$-module, $t \cdot e_i=\chi_i(t) e_i$ for any $t \in T^0$. For the first part, it suffices to show that any sum of $f^{\underline{\sigma}}_{g, \underline{a}^{\underline{\sigma}}} \otimes e_i$ with converging to zero coefficients is a $s$-analytic function on $\Nbar_I$ and for any such sum, $$\sum\limits_{g,\underline{a},\underline{\sigma}} \lambda_{g,\underline{a}}^{\underline{\sigma}} f^{\underline{\sigma}}_{g, \underline{a}^{\underline{\sigma}}}=0 \iff \mathrm{all} ~  \lambda_{g,\underline{a}}^{\underline{\sigma}}=0.$$ We choose $\tilde{g} \in \Nbar_I$ lifting each $g$, $$(\tilde{g}^1,\cdots,\tilde{g}^l):=(\psi^1_{I})^{-1}(\tilde{g}), ~~ \tilde{g}^i=(\tilde{g}^i_1,\cdots,\tilde{g}^i_{d_i}) \in \cO_{E_i}^{d_i}.$$ 
Suppose $\imath_{E_i}$ gives rise to $\cO_{E_i}=\bigoplus\limits_{1 \leq j \leq [E_i:\QQ_p]} \ZZ_p u^i_j$.
We choose coordinates $x^i=(x^i_1,\cdots,x^i_{d_i})$ on $\cO_{E_i}^{d_i}$, $x=(x^1,\cdots,x^l)$ for $\Nbar_I=\prod\limits_{1 \leq i \leq l} \Nbar_{I_i}$ through $\psi^1_I$ as in \S \ref{la rep} and coordinates $z^i_{j,k}$ for $1 \leq i \leq l, 1 \leq j \leq d_i, 1 \leq k \leq [E_i:\QQ_p]$ such that $$x^i=(x^i_1,\cdots,x^i_{d_i})=\big((\tilde{g}^i_1+\varpi_i^{s-1}(u^i_1\cdot z^i_{1,1}+\cdots+u^i_{[E_i:\QQ_p]} \cdot z^i_{1,[E_i:\QQ_p]})),$$ $$\cdots,(\tilde{g}^i_{d_i}+\varpi_i^{s-1}(u^i_1\cdot z^i_{d_i,1}+\cdots+u^i_{[E_i:\QQ_p]}\cdot z^i_{d_i,[E_i:\QQ_p]}))\big).$$
Apply Rem \ref{ano basis} to a product of Iwahori groups so an orthonormalizable basis for $C^{s,an}(N_I, A)$ is given by $$f^{\circ}_{g,\underline{a}}(x):= \begin{cases} 
     \prod\limits_{1 \leq i \leq l} \prod\limits_{1 \leq j \leq d_i} \prod\limits_{1 \leq k \leq [E_i:\QQ_p]} (z^i_{j,k})^{a^i_{j,k}} & x \in g\cdot \Nbar_I^s \\
      0 & x \not\in g\cdot \Nbar_I^s,
   \end{cases}$$ for $g$ running through $\Nbar_I/\Nbar_I^s$, $\underline{a}$ running through $\ZZ_{\geq 0}^{\sum_{1 \leq i \leq l} d_i[E_i:\QQ_p]}$. For each, $1 \leq i \leq l$, we index the set of embedding of $E_i$ to $\overline{\QQ}_p$ as $\{\sigma^i_1,\cdots,\sigma^i_{[E_i:\QQ_p]}\}$. For each $1 \leq i \leq l, 1 \leq j \leq d_i, 1 \leq  k \leq [E_i:\QQ_p]$, there is a $[E_i:\QQ_p] \times [E_i:\QQ_p]$ transition matrix $$M^i_{j,k}:=\begin{pmatrix}
\sigma^i_1 u^i_1 & \cdots & \sigma^i_1 u^i_{[E_i:\QQ_p]} \\
\vdots & \cdots & \vdots \\
\sigma^i_{[E:\QQ_p]} u^i_1 & \cdots & \sigma^i_{[E:\QQ_p]} u^i_{[E_i:\QQ_p]} \\
\end{pmatrix} $$ such that 
$$\begin{pmatrix}
\sigma^i_1 x^i_j \\ \vdots \\ \sigma^i_{[E_i:\QQ_p]} x^i_j \end{pmatrix} = \varpi_i^{s-1} \begin{pmatrix}
\sigma^i_1 u^i_1 & \cdots & \sigma^i_1 u^i_{[E_i:\QQ_p]} \\
\vdots & \cdots & \vdots \\
\sigma^i_{[E:\QQ_p]} u^i_1 & \cdots & \sigma^i_{[E:\QQ_p]} u^i_{[E_i:\QQ_p]} \\
\end{pmatrix}\cdot \begin{pmatrix}
z^i_{j,1} \\ \vdots \\ z^i_{j,[E_i:\QQ_p]}
\end{pmatrix} + \begin{pmatrix} \sigma^i_1 \tilde{g}^i_j \\ \vdots \\\sigma^i_{[E_i:\QQ_p]} \tilde{g}^i_j
\end{pmatrix}.$$
Then we have $$f^{\underline{\sigma}}_{g, \underline{a}^{\underline{\sigma}}}=\sum_{\underline{a'} \leq \underline{a}} t_{g,\underline{a'}}\cdot f^{\circ}_{g,\underline{a'}},$$ 
 for integral transition coefficients $t_{g,\underline{a'}}$. 
This shows that any convergent sum of $f^{\underline{\sigma}}_{g, \underline{a}^{\underline{\sigma}}}$ belongs to $C^{s,an}(\Nbar_I, A)$. 
$$C^{s,an}(\Nbar_I, V) \simeq \bigoplus_{i=1}^m C^{s,an}(\Nbar_I, A) \otimes e_i,$$ any convergent sum of $f^{\underline{\sigma}}_{g, \underline{a}^{\underline{\sigma}}} \otimes e_i$ belongs to $C^{s,an}(\Nbar_I, V)$. 
The claim for density follows from the fact that $\mathrm{det}(M^i_{j,k}) \neq 0$, hence invertible over $A$ for all $1 \leq i \leq l, 1 \leq j \leq d_i, 1 \leq  k \leq [E_i:\QQ_p]$.
Now we form $(\displaystyle\prod_{s \in S} V_s)^c$ for the basis $f^{\sigma}_{g, \underline{a}^\sigma_E, \underline{a}^\sigma_K} \otimes e_i$ such that $(\displaystyle\prod_{s \in S} V_s)^c \to C^{s,an}(\Nbar_I, V)$ is $T^s$ equivariant. We have the commutative diagram \[
\xymatrix{
& C^{s,1}(\Nbar_I, V) \ar@{^{(}->}[rd] & \\
(\displaystyle\prod_{s \in S} V_s)^c \ar[rr] \ar@{->>}[ru] & & C^{s,an}(\Nbar_I, V).\\
} \] The bottom arrow is $T^s$-equivariant and continuous since all $f^{\underline{\sigma}}_{g, \underline{a}^{\underline{\sigma}}}$ are in the unit ball of $C^{s,an}(\Nbar_I, A)$ by the discussion above.  

We claim the arrow $(\displaystyle\prod_{s \in S} V_s)^c \twoheadrightarrow C^{s,1}(\Nbar_I, V)$ is a bijection. Otherwise the kernel of the bottom arrow is nonempty, by the Lem \ref{rep of t}, there exists a weight mapping to zero, i.e., a finite linear combination of $f^{\sigma}_{g, \underline{a}^\sigma_E, \underline{a}^\sigma_K} \otimes e_i$ equals to zero, which implies a finite sum of $f^{\sigma}_{g, \underline{a}^\sigma_E, \underline{a}^\sigma_K}$ equals to zero. The coefficients must be zero.  
We endow the Banach $A$-module structure to $C^{s,1}_{w,\chi}$ from $(\displaystyle\prod_{s \in S} V_s)^c$.
\end{proof}

Let $K/E$ be a finite extension of $p$-adic local fields with the ring of integers $\cO_K / \cO_E$ and uniformizers $\varpi_K, \varpi_E$. 
We call an analytic function $f \in K\langle z_1,\cdots,z_d \rangle$ on $(\cO_K)^d$ \emph{overconvergent} if for $f=\sum\limits_{\underline{a} \in \ZZ^d_{\geq 0}} c_{\underline{a}} z^{\underline{a}}$, there exists $\varepsilon>1$ such that $$\lim_{|a| \to \infty} |c_{\underline{a}}|\cdot\varepsilon^{|a|} < \infty,$$
in which case $f$ is said to be of rate $\varepsilon$.
This is equivalent to that $f$ converges in a larger open ball than the unit ball in $d$ dimensional space over an extension of $K$, or $f$ is represented by a function in a certain formal power series ring over $\cO_K$ with $p$ inverted. 

\begin{df}
An analytic map $g: (\cO_K)^{d_1} \to (\cO_K)^{d_2}$ is said to be \emph{overconvergent} (resp. \emph{overconvergent with integral coefficients}) if all the $d_2$ coordinate functions are so. 
If $d_1=d_2$, $g$ is bijective and $g^{-1}$ is also overconvergent (resp. overconvergent with integral coefficients), $g$ is called an \emph{overconvergent} (resp. \emph{overconvergent with integral coefficients}) isomorphism. 
Moreover, we say $g$ is of rate $\varepsilon$ if the coordinate functions are so.
\end{df}

\begin{rem}
Our notion of overconvergent isomorphism is stronger than the notion ``strict isomorphism" in \cite{AS}, where the authors only require the coordinate functions to be in the Tate algebra of radius $1$.
\end{rem}

\begin{lem}\label{comp adm}
For two analytic maps $g: (\cO_K)^{d_1} \to (\cO_K)^{d_2}$ and $g': (\cO_K)^{d_2} \to (\cO_K)^{d_3}$, if $g$ is overconvergent with integral coefficients of rate $\varepsilon$, then $g' \circ g$ is overconvergent of rate $\varepsilon$.
overconvergent with integral coefficients map $\cO_K^{d_1} \to \cO_K^{d_2}$ pulls back any rigid function (analytic power series with coefficients converging to $0$) on $\cO_K^{d_2}$ to a rigid function on $\cO_K^{d_1}$.
\end{lem}

\begin{proof}
These are derived by direct calculations. 
\end{proof}

We choose an isomorphism $\imath_{K/E}: \cO_K \simeq \cO_E^{[K:E]}$. For an analytic map $g: (\cO_K)^{d_1} \to (\cO_K)^{d_2}$, $\imath_{K/E}$ induces 
$$\xymatrix{
g: \cO_K^{d_1} \ar[r] \ar@{}[d]|*[@]{\cong}_{(\imath_{K/E})^{d_1}} & \cO_K^{d_2} \ar@{}[d]|*[@]{\cong}^{(\imath_{K/E})^{d_2}} \\
\tilde{g}: \cO_E^{[K:E]d_1} \ar[r] & \cO_E^{[K:E]d_2}.\\
}$$ 

\begin{prop}\label{bc oc}
If $g$ is overconvergent (resp. overconvergent with integral coefficients) of rate $\varepsilon$, then so is $\tilde{g}$ (resp. overconvergent with integral coefficients).
\end{prop}

\begin{proof}
We use $p_j : \cO_K \to \cO_E, 1 \leq j \leq [K:E]$ to denote the $j$-th projection to the $j$-th $\cO_E$ of the map $\imath_{K/E}$. For any $x \in \cO_E$, $x=\varpi_E^n \cdot y$ for some $n \geq 0$, and $y \in \cO_E$ such that $|y|>|\varpi_E|$. 
This map is $\cO_E$-linear and thus has the property that $|p_j(x)| \leq \frac{|x|}{|\varpi_E|}$ for any $x \in \cO_E$. $\imath_{K/E}$ gives a basis $\{ e_1, \cdots, e_{[K:E]} \in \cO_K \}$ for $\cO_K$ such that $z_i=z_i^1 e_1 +\cdots+z_i^{[K:E]} e_{[K:E]}$ for coordinate functions $z_i$ of $\cO_K$ and $\{ z_i^1, \cdots, z_i^{[K:E]}\}$ of $\cO_E^{[K:E]} (1 \leq i \leq d_1)$. 
Consider any coordinate function $f: \cO_K^{d_1} \to \cO_K$ of $g$, $f=\sum\limits_{\underline{a} \in \ZZ^{d_1}_{\geq 0}} c_{\underline{a}} z^{\underline{a}}$, $$p_j \circ f \circ (\imath_{K/E}^{-1})^{[K:E]d_1}=\sum\limits_{\underline{a} \in \ZZ^d_{\geq 0}} p_j(c_{\underline{a}}) \prod_{1 \leq i \leq d_1} \sum_{\underline{j} \in \ZZ_{\geq 0}^{[K:E]}, |j|=a_i} e_i^{\underline{j}} z_i^{\underline{j}}$$ for some $e_i^{\underline{j}} \in \cO_E$. 
The coefficient of each monomial is bounded by $p_j(c_{\underline{a}}) \leq \frac{|c_{\underline{a}}|}{|\varpi_E|}$. 
The same rate $\varepsilon$ of staying bounded for $f$ then works for the coordinate function $p_j \circ f \circ (\imath_{K/E}^{-1})^{[K:E]d_1}$ in terms of new coordinates $\cO_E^{[K:E]d_1}$ as well.
\end{proof}

\begin{rem}
The overconvergent property of $\tilde{g}$ does not depend on choices of $\imath_{K/E}$ by the lemma above since linear isomorphisms are overconvergent.
\end{rem}

For $\mathscr{H}_{/\cO_E}=\GL_n$, $\mathrm{Sp}_{2n}$ or a unitary group associated to $J_n$ and an unramified extension $K/E$. and each $s \geq 1$, we use $i_T^s$ to denote 
\begin{eqnarray*}
i_T^s: \cO_E^n ~ (\mathrm{resp}. ~ \cO_K^n) & \simeq & T^s \\
(y_1,\cdots,y_n) & \mapsto & \psi_T(1+\varpi^s y_1,\cdots,1+\varpi^s y_n).
\end{eqnarray*}

\begin{lem}\label{adm int}
Let $\mathscr{H}_{/\cO_E}=\GL_n$, $\mathrm{Sp}_{2n}$ or a unitary group associated to $J_n$ and an unramified extension $K/E$. We equip $\Nbar_I \subset I$ the coordinates described in the beginning of this section. For any $g \in I$, $g$ induces an automorphism \begin{eqnarray*}
\hat{g}_n: \cO_E^d \xrightarrow{\psi^1_I} \Nbar_I & \hookrightarrow & I = \Nbar_I \times T^0 \times N_I \to \Nbar_I \xrightarrow{(\psi^1_I)^{-1}} \cO_E^d \\
x & \mapsto & g\cdot x \to \mathrm{proj}_1(g \cdot x), \end{eqnarray*} where the second map is projection to the $\Nbar_I$ factor via the Iwahori decomposition. 
Let $I^s_0$ be the subgroup of $I$ defined in \S \ref{Notation}. Let $\{t_1,\cdots,t_l\}$ be a set of representatives of $T^0/T^s$. For $g=\overline{n}_g \cdot t_g \cdot n_g \in I^s_0$, for $\overline{n}_g \in \Nbar_I$, $t_g \in T^0$ (with representative $t_{g_0}$), $n_g \in N_I^s$, $g$ induces a map
 \begin{eqnarray*}
\hat{g}_t: \cO_E^d \xrightarrow{\psi^1_I} \Nbar_I & \hookrightarrow & I_0^s \to T^0 \xrightarrow{\times t_{g_0}^{-1}} T^s \xrightarrow{(i_T^s)^{-1}} \cO_E^n ~ (\mathrm{resp}. ~ \cO_K^n \simeq \cO_E^{2n}) \\
x & \mapsto & g\cdot x \to \mathrm{proj}_2(g \cdot x) \to t_{g_0}^{-1}\mathrm{proj}_2(g \cdot x), \end{eqnarray*} where $\mathrm{proj}_2$ is projection to the $T^0$ factor and the image lands in $t_{g_0}\cdot T^s$.

Coordinate functions of $\hat{g}_n$ and $\hat{g}_t$ are represented by functions $\frac{u}{\varpi}$, with $u \in \cO_E[[(\varpi x_1),\cdots,(\varpi x_m)]], ~ \varpi | u(0)$.
In particular, $\hat{g}_n$ is an overconvergent with integral coefficients isomorphism of rate $|\varpi|^{-1}$ and coordinate functions of $\hat{g}_t$ are overconvergent with integral coefficients of rate $|\varpi|^{-1}$. 
\end{lem}

\begin{proof}
We prove both claims for a fixed $g \in I^s_0$ at the same time, note that for the first claim for $\hat{g}_n$, we just apply the case $I^0_0=I$. We first prove the case $H=\GL_n$. If $g \in \Nbar_I \times T^0 = T^0 \times \Nbar_I$, coordinate functions are linear polynomials with integral coefficients. By the Iwahori decomposition and Lem \ref{comp adm}, it reduces to only consider $$g=\begin{pmatrix} 
    1        &  g_1 & \cdots & \cdots & \vdots      \\
      & 1      & \cdots  & \cdots & \vdots    \\
      &          & \ddots   &       & \vdots     \\
          & \text{\Huge0} &   & 1     & g_m      \\
      &              &   &   & 1 \end{pmatrix} \in N_I^s, ~ m=\sum\limits_{1 \leq i \leq n-1} i=\frac{n(n-1)}{2}.$$ 
      $$\mathrm{Let} ~ x=\begin{pmatrix} 
    1                                    \\
     \varpi x_1 & 1             &   & \text{\Huge0}\\
    \vdots  &    \vdots      & \ddots                \\
      \vdots    & \vdots &   & 1            \\
   \cdots  &       \cdots        &  \cdots &  \varpi x_m & 1 \end{pmatrix} \in \Nbar_I$$ 
   and $x_{1,g},\cdots,x_{m,g}$ be translated coordinate functions $\hat{g}$ of $\Nbar_I$ with respect to left multiplication by $g$.  
   $$x_g=\begin{pmatrix} 
    1                                    \\
     x_{1,g} & 1             &   & \text{\Huge0}\\
    \vdots  &    \vdots      & \ddots                \\
      \vdots    & \vdots &   & 1            \\
    \cdots  &       \cdots        &  \cdots &  x_{m,g} & 1 \end{pmatrix} \in \Nbar_I$$ satisfying $g\cdot x=x_g \cdot t \cdot n_I$, $t \in T^s, n_I \in N_I$. $\Rightarrow$ 
 $$   \begin{pmatrix} 
    1                                    \\
     x_{1,g} & 1             &   & \text{\Huge0}\\
    \vdots  &    \vdots      & \ddots                \\
      \vdots    & \vdots &   & 1            \\
    \cdots  &       \cdots        &  \cdots &  x_{m,g} & 1 \end{pmatrix} \cdot t \cdot n_I= \begin{pmatrix} 
    1        &  g_1 & \cdots & \cdots & \vdots                       \\
      & 1      & \cdots  & \cdots & \vdots \\
      &          & \ddots   &       & \vdots     \\
          & \text{\Huge0} &   & 1     & g_m      \\
      &              &   &   & 1 \end{pmatrix} \cdot \begin{pmatrix} 
    1                                    \\
   \varpi  x_1 & 1             &   & \text{\Huge0}\\
    \vdots  &    \vdots      & \ddots                \\
      \vdots    & \vdots &   & 1            \\
   \cdots  &       \cdots        &  \cdots & \varpi x_m & 1 \end{pmatrix},$$ 
$g_i \in \cO_E$ for $1 \leq i \leq m$, we implement Gram–Schmidt process to the result of right hand side to produce our desired $x_{1,g}, \cdots, x_{m,g}$. 
Now we view all the coefficients as inside the formal power series ring of Tate algebra $\cO_E[[(\varpi x_1),\cdots,(\varpi x_m)]] \subset E \left\langle x_1, \cdots, x_m \right\rangle$. Let $$\frak{I}=(\varpi^s), ~ \frak{m}=(\varpi x_1,\cdots,\varpi x_m) \subset \cO_E[[(\varpi x_1),\cdots,(\varpi x_m)]].$$ 
The matrix $g\cdot x$ has the property that lower triangular entries are in $\frak{m}$, upper triangular entries are in $\frak{I}$ and diagonal entries are in $1+\frak{m}\frak{I}$. 
The Gram–Schmidt process produces $x_g, t, n_I$ such that entries of $x_g$ are in $\frak{m}$, entries of $t$ are in $1+\frak{m}\frak{I}$ and entries of $n_I$ are in $\frak{I}$.
This is true since the right hand side is an upper unipotent matrix mod $\frak{m}$ and every time we want to invert an element, that element must be a diagonal element and always be in $1+\frak{m}\frak{I}$. 
So we obtain all $x_{1,g}, \cdots, x_{m,g} \in \frak{m} \subset \cO_E[[(\varpi x_1),\cdots,(\varpi x_m)]] \subset E \left\langle x_1, \cdots, x_m \right\rangle$ and $\frac{1+\frak{m}\frak{I} - 1}{\varpi^s}=\frak{m}$, which translates to the statement that both $\hat{g}_n$ and $\hat{g}_t$ are overconvergent with integral coefficients. 
\footnote{The Gram–Schmidt process used here can be regarded as an analogue of Iwahori decomposition for the complete local ring $\cO_E[[(\varpi x_1),\cdots,(\varpi x_m)]]$.}
   
For other cases, we have constructed embeddings $i_{\GL}^1: \Nbar_I \hookrightarrow \Nbar_{I_{\GL_{2n}}}$ and as sections of projections $p_{\GL}^1: \Nbar_{I_{\GL_{2n}}} \twoheadrightarrow \Nbar_I$ at the beginning of this section. 
We have seen that coordinate functions of both $i_{\GL}^s$ and $p_{\GL}^s$ ($s \geq 1$) are polynomials of integral coefficients. 
Since $i_{\GL}: I \to I_{\GL_{2n}}$ is compatible with respect to the Iwahori decomposition, we have the following commutative diagram:
$$\xymatrix{
\widehat{i_{\GL}(g)}: \Nbar_{I_{\GL_{2n}}} \ar[r] & I_{\GL_{2n}} \ar[r] & \Nbar_{I_{\GL_{2n}}} \ar@/_/ [d]_{p_{\GL}^1}\\
\hat{g}: \Nbar_I \ar[r] \ar[u]^{i_{\GL}^1} & I \ar[r] \ar[u] & \Nbar_I. \ar@/_/ [u]_{i_{\GL}^1}\\
}$$ 
By Lem \ref{comp adm}, we conclude that $\hat{g}_n=p_{\GL}^1 \circ \widehat{i_{\GL}(g)} \circ i_{\GL}^1$ is overconvergent with integral coefficients. 
Replace $g$ by $g^{-1}$, $\widehat{g^{-1}}_n=\hat{g}_n^{-1}$ is overconvergent with integral coefficients as well. The argument works similarly for $\hat{g}_t$ as well.
\end{proof} 

\begin{proof}[Proof of Proposition \ref{aux mod}]
By Lem \ref{dual nice}, if a prescribed group element induces a nice endomorphism of $\DD^{s,1}_I(V)$ if it stabilizes $C^{s,1}(\Nbar_I, V)$. 
It remains to prove that $C^{s,1}(\Nbar_I, A)$ is $I^s_0$-stable for $I$ being a single factor of Iwahori ($l = 1$ case), as any element of the eigenbasis we introduced is a product of eigenfunctions on $\Nbar_{I_i}$ factors. 
Let $\{t_1,\cdots,t_l\}$ be in Lem \ref{adm int}, $\chi_{\Omega}(t_{g_0}) \in A^\times$. 
For any $f \in C^{s,1}(\Nbar_I, A)$ and $g \in I^s_0$, $$(g^{-1}\cdot f)(x)=f(g\cdot x)=f(\hat{g}_n(x))\chi_{\Omega}(t_{g_0}\hat{g}_t(x)).$$ Since $\hat{g}_n$ is an overconvergent with integral coefficients automorphism by Lem \ref{adm int}, $f(\hat{g}_n \cdot -) \in C^{s,1}(\Nbar_I, A)$ by Lem \ref{comp adm}. 
Recall the assumption on $s$ in \S \ref{la rep} ($s \geq s[\Omega]$) and results in \S \ref{Weight spaces}, $\chi_{\Omega}|_{T^s}$ is uniquely decomposed as $(\sigma,E)$-analytic characters of $i_T^s: \cO_E^n ~ (\mathrm{resp}. ~ \cO_E^{2n}) ~ \simeq T^s$, $\chi_{\Omega}(\hat{g}_t \cdot -) \in C^{s,1}(\Nbar_I, A)$ as well. 

The group multiplication $I \times I \to I$ is analytic, and so is its compositions with $\mathrm{proj}_1,~ \mathrm{proj}_2$: \begin{eqnarray}\label{group mult} I \times \Nbar_I \to I \to \Nbar_I, ~ \hat{m}_T: I^s_0 \times \Nbar_I \to I \to T^0. \end{eqnarray}
The first morphism (\ref{group mult}) induces the pullback $C^{an}(\Nbar_I, A) \to C^{an}(I, \QQ_p) \hat{\otimes} C^{an}(\Nbar_I, A)$, therefore inducing $C^{s,1}(\Nbar_I, A) \to C^{an}(I, \QQ_p) \hat{\otimes} C^{s,1}(\Nbar_I, A)$ since the space of characteristic functions on $\Nbar_I^s$ cosets are preserved by left $I$ translations.
And by results in \S \ref{Weight spaces}, $\chi_{\Omega}$ is $\QQ_p$-analytic on $T^s$, the pullback of this character via $\hat{m}_T$ is represented by a function in $C^{an}(I^s_0, \QQ_p) \hat{\otimes} C^{s,1}(\Nbar_I, A)$ by the same reason.
By definition, the $I^s_0$ action on $C^{s,1}(\Nbar_I, A)$ is analytic. 
By \cite[Cor 5.1.9]{Eme17}, $\DD^{s,1}_I(V)$ is a continuous $D(I^s_0, \QQ_p)$ module.
\end{proof}

\section{Koszul complex and a comparison of $N$ cohomology}\label{Kos comp}
Let $\Gamma \simeq \ZZ^n$ be a finitely generated torsion free abelian group. Let $N \simeq \ZZ_p^n$ be a compact $p$-adic analytic abelian group containing $\Gamma$ as a dense subgroup with the $\ZZ_p$ action via module structure of $\ZZ_p^n$. Let $\frak{n}$ be its Lie algebra over $\QQ_p$. 
We prove $H^\ast(\Gamma, \mathbb{D}_{w,\Omega}^{s}) \simeq H^\ast_{an}(N, \mathbb{D}_{w,\Omega}^{s}) \simeq H^\ast(\mathfrak{n}, \mathbb{D}_{w,\Omega}^{s})^N$ in the sense of Kohlhaase, and the second isomorphism can be deduced by applying a general result in \cite{Koh11}.

The orbits of this action are $1$-parameter subgroups of $N$. Choose a basis of generators $e_1,\cdots,e_n$ of $\Gamma$, $\Gamma \subset \ZZ_p\cdot e_1+\cdots+\ZZ_p\cdot e_n \subset N$ for a dense $\Gamma$, hence $N=\ZZ_p\cdot e_1 \oplus \cdots \oplus \ZZ_p e_n$. 
Therefore we assume $\Gamma < N$ is induced by the standard inclusion $\ZZ \hookrightarrow \ZZ_p$ with (topological) generators $e_1,\cdots,e_n$.

Let $K$ be a finite extension of $\QQ_p$. Let $\cM_N$ denote the category of complete Hausdorff locally convex $K$-vector spaces with the structure of a separately continuous $D(N, K)$-module,
taking as morphisms all continuous $D(N, K)$-linear maps as in \cite{Koh11}.

Let $\mathscr{W}_N \simeq \mathscr{W}_{\ZZ_p}^n$ be the (continuous) weight space of $N$ over $K$. For any $z=(z_1,\cdots,z_n) \in \mathscr{W}_N(K)$, there is a locally $\QQ_p$-analytic $K$-valued character $\kappa_z$ such that $$\kappa_z(a)=\prod_{i=1}^n (1+z_i)^{a_i}, ~~ a=\sum_{i=1}^n a_i e_i \in N.$$
So we have the embedding \begin{eqnarray*}
\mathscr{W}_N(K) & \hookrightarrow & C^{an}(N, K) \\
z & \mapsto & \kappa_z.
\end{eqnarray*} 
Hence any linear form $\lambda \in D(N, K)$ gives rise to the function $$F_\lambda(z):=\lambda(\kappa_z)$$ on $\mathscr{W}_N(K)$ which is called the \emph{Fourier transform} of $\lambda$. 
Moreover, $F_\lambda$ is a rigid function on $\mathscr{W}_N$ with coefficients in $K$. 
Let $\sO(\mathscr{W}_N)$ denote the ring of all $K$-rigid functions on $\mathscr{W}_N$, 
$$\sO(\mathscr{W}_N)=\{ F(T_1,\cdots,T_n)=\sum_{\underline{m} \in \ZZ_{\geq 0}^n} a_{\underline{m}} \underline{T}^{\underline{m}}, ~ a_{\underline{m}} \in K, ~ \mathrm{which ~ converge ~ on} ~ \mathscr{W}_{N_{/\CC_p}}(\CC_p).\}$$

We have the following multivariable Amice's theorem stated as Thm 2.2. in \cite{ST01}.
\begin{thm}[Amice]
The Fourier transform \begin{eqnarray*}
D(N, K) & \xrightarrow{\simeq} & \sO(\mathscr{W}_N) \\
\lambda & \mapsto & F_\lambda
\end{eqnarray*} is an isomorphism of $K$-Fr\'echet algebra.
\end{thm}

One can embed the usual group algebra $K[\Gamma] \subset K[N]$ as dense subalgebras of $D(N, K)$ viewing a group element as a Dirac distribution as in \cite{ST02}. 
Under the Fourier transform,
$$e_i \mapsto 1+T_i, ~~ 1 \leq i \leq n,$$ $$K[\Gamma] \simeq K[T_1,\cdots,T_n,(1+T_1)^{-1},\cdots,(1+T_n)^{-1}] \subset D(N, K).$$

\begin{lem}\label{reg seq D}
$(T_1,\cdots,T_n) \subset K[\Gamma] \subset \sO(\mathscr{W}_{\ZZ_p^n})$ form a regular sequence for both algebras.
\end{lem}

\begin{proof}
For $K[\Gamma]$, the result follows from $$K[T_1,\cdots,T_n,(1+T_1)^{-1},\cdots,(1+T_n)^{-1}]/(T_1) \simeq K[T_2,\cdots,T_n,(1+T_2)^{-1},\cdots,(1+T_n)^{-1}].$$
For any $\underline{m}=(m_2,\cdots,m_n) \in \ZZ_{\geq 0}^{n-1}$, we use $\underline{\hat{T^1}}^{\underline{m}}$ to denote $T_2^{m_2}\cdots T_n^{a_n}$. We construct an algebra homomorphism \begin{eqnarray*}
\sO(\mathscr{W}_{\ZZ_p^n}) & \twoheadrightarrow & \sO(\mathscr{W}_{\ZZ_p^{n-1}}) \\
\sum_{i=0}^{+ \infty} \sum_{\underline{m} \in \ZZ_{\geq 0}^{n-1}} a_{i,\underline{m}} T_1^i\underline{T}^{\underline{m}} & \mapsto & \sum_{\underline{m} \in \ZZ_{\geq 0}^{n-1}} a_{0,\underline{m}} \underline{T}^{\underline{m}}.
\end{eqnarray*} If we show that the kernel of this map is the principal ideal generated by $T_1$, then we prove the statement inductively since $\sO(\mathscr{W}_{\ZZ_p^n})$ is an integral domain. 
Suppose that $f \in \sO(\mathscr{W}_{\ZZ_p^n})$ maps to $0$, then $\frac{f}{T_1}$ is a power series with coefficients in $K$, it certainly converges on $\mathscr{W}_{\ZZ_p^n}(\CC_p)$ as $f$ converges on it.
\end{proof}

Let $R$ be a commutative ring and a sequence $x_1,\cdots,x_l$ of elements of $R$. There is a so called \emph{Koszul complex} $K(x_1,\cdots,x_l)$ associated to such a sequence.
$$0 \rightarrow \bigwedge^{l} R^l \xrightarrow{d_l} \bigwedge^{l-1} R^l \rightarrow \cdots \rightarrow \bigwedge^{1} R^l \xrightarrow{d_1} R \rightarrow R/(x_1,\cdots,x_l) \rightarrow 0,$$
where $R^l=\bigoplus\limits_{i=1}^{l}Rs_i$ and the differential $d_k$ is given by: for any $1 \leq i_1 < \cdots < i_k \leq l$,
$$d_{k}(s_{i_1} \wedge \cdots \wedge s_{i_k})=\sum_{j=1}^{k}(-1)^{i+1} x_{i_j} s_{i_1} \wedge \cdots \wedge \widehat{s_{i_j}} \wedge \cdots \wedge s_{i_k}.$$
A fundamental theorem for the Koszul complex is:
\begin{thm}[{\cite[Thm 16.1]{Mat89}}]\label{Koszul}
If $(x_1,\cdots,x_l)$ is a regular sequence of elements in $R$, then the Koszul complex $K(x_1,\cdots,x_l)$ is a free resolution of the quotient ring $R/(x_1,\cdots,x_l)$.
\end{thm}

Now we are ready to state and prove the main theorem of this section.

\begin{thm}\label{group lie comparison}
There are natural $K$-linear isomorphisms
$$H^q(\Gamma, V) \simeq H^q(\frak{n}, V)^N$$ for all $q \geq 0$ and any object $V$ of $\cM_N$.
\end{thm}

\begin{proof}
$\ZZ[\Gamma] \hookrightarrow K[\Gamma]$ is flat since flatness is stable under base change $\ZZ \hookrightarrow K$, therefore $$H^q(\Gamma, V)=\mathrm{Ext}^q_{\ZZ[\Gamma]}(\ZZ, V) \simeq \mathrm{Ext}^q_{K[\Gamma]}(K, V).$$ Lem \ref{reg seq D} and \ref{Koszul} enables us to use the Koszul complex $K(T_1,\cdots,T_n)_{/ K[\Gamma]}$ to represent $K$. 
One checks that $$K(T_1,\cdots,T_n)_{/ K[\Gamma]} \otimes_{K[\Gamma]} D(N, K) \simeq K(T_1,\cdots,T_n)_{/ D(N, K)},$$ which again represents $K$ as a $D(N,K)$-module by Lem \ref{reg seq D} and Lem \ref{Koszul}. 
Then $$\mathrm{Ext}^q_{K[\Gamma]}(K, V) \simeq \mathrm{Ext}^q_{D(N, K)}(K, V) \simeq H^q_{an}(N, V) \simeq H^q(\frak{n}, V)^N$$ by \cite[Thm 4.8, Thm 4.10]{Koh11}.
\end{proof}

\begin{rem}
If moreover $V$ has a $A=\sO(\Omega)$-module structure which commutes with the group action, then this isomorphism is as an $A$-linear isomorphism since $A$ action commutes with the $D(N, K)$ action.
\end{rem}

\section{$N$ cohomology of $\mathbb{D}_{w,\Omega}^{s}$}\label{N coho}
We keep the same notation as in the previous sections. 
Let $G_i$ be either $\GL_{2n}$ or $\Sp_{2n}$ over $E_i$ or the unitary group associated to the Hermitian form $J_n$ and an unramified extension $K_i / E_i$ with various subgroups and their derived compact $p$-adic groups as in \ref{Notation}. 
For distinguishing these $p$-adic groups associated with the different algebraic groups, we use subscripts to specify.
For example, the Iwahori subgroup of $G_i(E_i)$ is denoted by $I_{G_i}$.

For the case of $\GL_{2n}$ and $\Sp_{2n}$, we specify each Borel $B_i$ as the upper triangular Borel subgroup, and each $P_i$ to be the standard Siegel parabolic whose lower left $n \times n$ quadrant is zero:
\[
\left[ 
\begin{array}{c|c} 
  \ast & \ast \\ 
  \hline 
  0 & \ast 
\end{array} 
\right].
\]
For the unitary group case, $B_i$ and $P_i$ are taken to be the pullbacks of the upper triangular Borel and standard Siegel parabolic via the standard embedding $G_i \hookrightarrow \Res_{K_i / E_i} \GL_{2n}$.

We make several abbreviations for this section:
$$W_G := \prod\limits_{i=1}^l W(R_{E_i/\QQ_p} G_i, R_{E_i/\QQ_p} T_i)(\overline{\QQ}_p),$$
$$W_L := \prod\limits_{i=1}^l W(R_{E_i/\QQ_p} L_i, R_{E_i/\QQ_p} T_i)(\overline{\QQ}_p),$$  $$W^P:=\prod\limits_{i=1}^l W^{P_i}(G_i, S_i), ~~ w=\prod\limits_{i=1}^l (w_i) \in W^P,$$
$$N:=\prod\limits_{i=1}^l (N_i(E_i) \cap w_i I_{G_i} w_i^{-1}),$$  $$\frak{n}_i:=\mathrm{Lie}(N_i(E_i))_{/ E_i}, ~~ \frak{n}:=\mathrm{Lie}(N)_{/ \QQ_p}=\bigoplus_{i=1}^l \frak{n}_i.$$
As $T^s_{G_i} = T^s_{L_i}$, we use $T^s_i$ to denote both groups. 
$$T^0 := \prod\limits_{i=1}^l T^0_i, ~~ I_G:=\prod\limits_{i=1}^l I_{G_i}, ~~ I^s_G:=\prod\limits_{i=1}^l I^s_{G_i},$$   
$$I^s_{0,G}:=\prod\limits_{i=1}^l I^s_{0,G_i}, ~~ N_G:=\prod\limits_{i=1}^l N_{G_i}, ~~ \Nbar_G:=\prod\limits_{i=1}^l \Nbar_{G_i},$$
$$ \frak{g}_i := \Lie(I_{G_i})_{/\QQ_p}, ~~ \frakn_{b,i} := \Lie(N_{G_i})_{/\QQ_p}, ~~ \bar{\frakn}_{b,i} := \Lie(\Nbar_{G_i})_{/\QQ_p},$$
$$\frak{g}:= \prod\limits_{i=1}^l \frak{g}_i, ~~ \frakn_{b} := \prod\limits_{i=1}^l  \frakn_{b,i}, ~~ \bar{\frakn}_{b} := \prod\limits_{i=1}^l  \bar{\frakn}_{b,i}.$$
For the Levi $L$, we define 
$$I_L:=\prod\limits_{i=1}^l I_{L_i}, ~~ B_L:=T^0\cdot N_L, ~~ I^s_{0,L}:=\prod\limits_{i=1}^l I^s_{0, L_i}, ~~ \frakl:=\Lie(I_L)_{/\QQ_p},$$ $$N_L=\prod\limits_{i=1}^l N_{L_i}, ~~ \overline{N}_L^s=\prod\limits_{i=1}^l \Nbar_{L_i}^s, ~~ \overline{N}_L:=\overline{N}_L^1,$$ 
$$\overline{N}^s_+:=\prod\limits_{i=1}^l w_i \Nbar_{G_i}^s w_i^{-1} \cap N_i(E_i), ~~ \overline{N}^s_-:=\prod\limits_{i=1}^l w_i \Nbar_{G_i}^s w_i^{-1} \cap \overline{N}_i(E_i).$$
There are Iwahori decompositions 
$$I_G = \Nbar_G \times T^0 \times N_G, ~~ I_L = \Nbar_L \times T^0 \times N_L.$$

Let $\Omega$ be an irreducible Zariski closed subspace of an affinoid subdomain of the weight space of $T^0$. 
The integral domain $A=\sO(\Omega)$ contains a big enough $p$-adic field $K$ such that the weight decomposition for Lie algebra exists over $K$. 
The Lie algebra $\frak{g} \otimes_{\QQ_p} K$ decomposes as a direct sum of eigenspaces for $T^0$ action with (positive/negative) roots $\Delta=\Delta^+ \sqcup \Delta^-$. 
Positive (resp. negative) roots correspond to eigenspaces of $\frak{n}_b \otimes_{\QQ_p} K$ (resp. $\overline{\frak{n}}_b \otimes_{\QQ_p} K$). 
Let $\delta$ be the half sum of positive roots and $w\delta-\delta$ is a $\QQ_p$-algebraic character of $T^0$.

There are $A$-modules $\DD^s_{w,\chi} \subset \DD^{s,1}_{w,\chi}$ (in particular, $\DD^s_{w,\Omega} \subset \DD^{s,1}_{w,\Omega}$) for a weight character $\chi: T^0 \to A^\times$ (resp. $\chi_{\Omega}$ for an irreducible affinoid subdomain $\Omega$ of $\mathscr{W}_{T^0}$). 
We remind the readers that we only know $\DD^s_{w,\chi} \subset \DD^{s,1}_{w,\chi}$ as $wI^s_{0,G}w^{-1}$ representations by Prop \ref{aux mod}. 
We still use $\Ind^s_{w,\chi}$ ($\Ind^s_{w,\Omega}$) as in \S \ref{la rep} to denote the same representations for $wI_Gw^{-1}$. 
We use $\Ind^s_{L,w\cdot\Omega}$ for $\Ind^s_{I_L,w\cdot \chi_\Omega}$ and $\DD^s_{L,w\cdot\Omega}$ its $A$-linear Banach dual. 
We use $\Ind^{s,1}_{w,\chi}$, $\Ind^{s,1}_{w,\Omega}$, $\Ind^{s,1}_{L,w\cdot \Omega}$ to denote the auxiliary subrepresentations $C^{s,1}(\Nbar_I, A)$ constructed in \S \ref{auxiliary modules} for various choices of Iwahori and characters (or finite rank $A$ modules of $B_I$).

Let $U(\frak{l})$ be the universal enveloping algebra with $A$-coefficients for $\frak{l}$ with center $Z(\frak{l})$.
For any $Z(\frak{l})$ module $M$ and infinitesimal character $\theta: Z(\frak{l}) \to A$, we set $$M[\theta]:=\{m \in M ~ | ~ (x-\theta(x))^n m=0, ~ \for ~ \mathrm{some} ~ n, ~ \forall x \in Z(\frak{l}) \}.$$

We want to understand the structure of $H^\ast(\frak{n}, \mathbb{D}_{w,\Omega}^{s})^N$ as an $I_L$ representation. 
And we establish the structure of the $N$ cohomology of $\mathbb{D}_{w,\Omega}^{s}$ via $N$ cohomology of the auxiliary module $\mathbb{D}^{s,1}_{w,\chi}$.
We have the following main theorem as below.
\begin{thm}
\label{pppp}
For each $w \in W^P$, let $l(w):=\sum\limits_{i=1}^l [E_i:\QQ_p] l(w_i)$, where the later $l$ refers to the length function on absolute Weyl group $W(G_i, T_i)$. 
The $N$ cohomology $H^{l(w)}_{an}(N, \DD_{w,\Omega}^s) \simeq H^{l(w)}(\mathfrak{n}, \mathbb{D}_{w,\Omega}^{s})^N$ in degree $l(w)$ admits a $I_L$-equivariant direct summand
$$i : \mathbb{D}_{L, w \cdot \Omega}^{s} \hookrightarrow H^{l(w)}(\mathfrak{n}, \mathbb{D}_{w,\Omega}^{s})^N. $$ 
\end{thm}

By Harish-Chandra homomorphism (for example \cite[Lem 8.17]{Kna86}), there is an infinitesimal character $\theta^w_\Omega$ of $Z(\frakl)$ on $\DD^s_{L,w\cdot\Omega}$.

We endow $\overline{N}_L^s$ the analytic structure coming from $\GL_n$ case in \S \ref{analysis}. $\overline{N}^s_{\pm}$ is isomorphic to copies of $\ZZ_p$ as $p$-adic Lie groups, we use the canonical analytic structure on it.

\begin{lem}\label{decomp wN}
Fix a $w \in W^P$, $w\Nbar_G^s w^{-1} \simeq \overline{N}^s_+ \times \overline{N}_L^s \times \overline{N}^s_-$ as strict $p$-adic manifolds (\cite[\S 3.3]{AS}) for any $s \in \ZZ_{\geq 1}$.
\end{lem}

\begin{proof}
It suffices to prove the statement for one group of interests $G$ over $E$. We apply \cite[Prop 2.1.8 (3), Prop 2.1.12]{CGP15} to the solvable group $w \overline{\cN}_{B} w^{-1}$ and the Siegel parabolic group $\cP$ as $P(\lambda)$. 
Together with the fact $w \Nbar_{G}^s w^{-1} \cap L(E)=\Nbar_{L}^s$ for $w \in W^{P}(G, S)$, this gives us the desired bijection.
Note that the (strict) analytic structures on both sides are compatible with the correlating schematic structures and the isomorphism in \cite[Prop 2.1.12]{CGP15} is schematic, the transition isomorphism is actually given by polynomials which is in the Tate algebra, hence strict in the sense of \S 3.3, \cite{AS}.
\end{proof}

We shall use Lem \ref{decomp wN} to give a projection $p_L$ from $w \Nbar_G w^{-1}$ to $\Nbar_L$ \begin{eqnarray*} 
\overline{N}_L & \xrightarrow{i_L} w\Nbar_G w^{-1} \xrightarrow{p_L} & \overline{N}_L \\
x & \mapsto  (\mathrm{id}, x, \mathrm{id})  \mapsto & x
\end{eqnarray*} 
such that $p_L \circ i_L = \id$.

We now construct $i$ and $p$ as claimed in Thm \ref{pppp}.
First, we construct $\tilde{i}$ and $\tilde{p}$ as morphisms of $\sO(\Omega)$-modules for $l(w)$-th term of the chain complex.  
There are natural $I_L$ (resp. $I^s_{0,L}$)-equivariant inclusions \[ \xymatrix{
\tilde{i} : \mathbb{D}_{L, w \cdot \Omega}^{s} \ar@{^{(}->}[r] \ar@{^{(}->}[d] & \bigwedge^{l(w)} \frakn^\ast \otimes \mathbb{D}_{w,\Omega}^{s} \ar@{^{(}->}[d]\\
\tilde{i} : \mathbb{D}_{L, w \cdot \Omega}^{s,1} \ar@{^{(}->}[r] & \bigwedge^{l(w)} \frakn^\ast \otimes \mathbb{D}_{w,\Omega}^{s,1},\\
} \]
which come from the natural $I_L$ (resp. $I^s_{0,L}$)-equivariant surjections by the Banach duality relation 
\[ \xymatrix{
\tilde{i}^\ast : \bigwedge^{l(w)} \frakn \otimes \mathrm{Ind}^s_{w, \Omega} \ar@{->>}[r] & \mathrm{Ind}^s_{L, w \cdot \Omega}\\
\tilde{i}^\ast : \bigwedge^{l(w)} \frakn \otimes \mathrm{Ind}^{s,1}_{w, \Omega} \ar@{->>}[r] \ar@{^{(}->}[u] & \mathrm{Ind}^{s,1}_{L, w \cdot \Omega}. \ar@{^{(}->}[u]\\
} \]
We need two steps to construct the horizontal arrows above. Firstly, note that there is a natural $I_L$-equivariant surjection 
\begin{eqnarray*}
\mathrm{Ind}^s_{w, \Omega} & \twoheadrightarrow & \mathrm{Ind}^s_{L, \chi_{\Omega}^w} \\
f & \mapsto & f|_{I_L}
\end{eqnarray*} 
where $s$-analytic functions restrict to $s$-analytic functions due to Lem \ref{decomp wN} and \cite[Prop 3.3.6]{AS}.

Secondly, we need to use Mackey's tensor product theorem: if $V$ is a finite dimensional representation of $I_L$,
\[ \phi : V \otimes \mathrm{Ind}^s_{L, \chi_{\Omega}^w} \to \mathrm{Ind}^{s,an}_L (V \otimes \chi^w_{\Omega}) \]
\[ [\phi(v \otimes f)](g) = g^{-1} v \otimes f(g) \] 
is an $I_L$-equivariant linear isomorphism.
Therefore there is a $I_L$-filtration on $\bigwedge^{l(w)} \frakn \otimes \mathrm{Ind}^s_{L, \chi_{\Omega}^w}$, which comes from a filtration of $\bigwedge^{l(w)} \frakn \otimes_{\QQ_p} K$ as $B_L$ representations. 

\begin{lem}\label{Kostant tri}
If $w \in W^P$,
$\bigwedge^{l(w)} \frakn^\ast_{/ \QQ_p}$ has a highest weight vector as a $\QQ_p$-algebraic $I_L$-representation, in particular, \[\bigwedge^{l(w)} (\frakn \cap w \overline{\frak{n}}_b w^{-1})^\ast \subset (\bigwedge^{l(w)} \frakn^\ast)^{N_L} \neq 0.\] 
$B_L$ action on such a vector factors through an ($\QQ_p$)-algebraic character of $T^0 \simeq B_L/N_L$ of weight $w\delta-\delta$.
And \[ \bigwedge^{l(w)} \frakn \cap w \overline{\frak{n}}_b w^{-1} \subset (\bigwedge^{l(w)} \frakn)^{\overline{N}_L} \neq 0. \]
\end{lem}
\begin{proof}
The adjoint representation on the Lie algebra is algebraic, so do its wedge products, restrictions to subgroups and subquotients. 
Our Weyl group element $w$ acts on $\frak{g}$ by conjugation and $$\dim_{\QQ_p} \overline{\frak{n}}_b \cap w^{-1}\frak{n}w=\sum_{i=1}^l \dim_{\QQ_p} \overline{\frak{n}}_{b,i} \cap w_i^{-1}\frak{n}_i w_i=\sum_{i=1}^l [E_i:\QQ_p] l(w_i)=l(w).$$ 
The second equality is due to $\dim_{E_i} \overline{\frak{n}}_{b,i} \cap w_i^{-1}\frak{n}_i w_i=l(w_i)$ since $w_i \in W^{P_i}(G_i, S_i)$ by \cite[chapter VI, \S 1, Cor 2. of Prop. 17]{Bou02}. 

The rest follows from the Kostant's theorem on Lie algebra cohomology of $\frakn$ for trivial representation of $\mathrm{Lie}(I_L)$ when the base field is allowed to be non-algebraically closed (Kostant's theorem applies since the involved representations of $T^0$ or Cartan subalgebra splits as sum of characters over $K$). See \cite{Kos61}, \cite{Car61}, 4.2-4.4 of \cite{CO}.

More specifically, The line $\bigwedge^{l(w)} (\frakn \cap w \overline{\frak{n}}_b w^{-1})^\ast$ gives a highest weight of $H^{l(w)}(\frakn, \mathds{1})^{N_L}$ by \cite[Thm 5.14]{Kos61}.
It should lift to a highest weight of $\bigwedge^{l(w)} \frakn^\ast$.
Cartier further proves that $\bigwedge^{l(w)} (\frakn_b \cap w \overline{\frak{n}}_b w^{-1})^\ast$ occurs in $\bigwedge^{l(w)} \frakn_b^\ast$ with multiplicity one.
As $w \in W^P$, \[ \frakn_b \cap w \overline{\frak{n}}_b w^{-1} = \frakn \cap w \overline{\frak{n}}_b w^{-1}, ~ \bigwedge^{l(w)} \frakn_b^\ast \twoheadrightarrow \bigwedge^{l(w)} \frakn^\ast, \]
we see \[\bigwedge^{l(w)} (\frakn \cap w \overline{\frak{n}}_b w^{-1})^\ast \subset (\bigwedge^{l(w)} \frakn^\ast)^{N_L} \] gives rise to a highest weight.

For the last claim, note that $\overline{\frakn} \simeq (\frakn)^\ast$ as $I_L$-representations. 
Then the same argument applying to $\bigwedge^{l(w)} \overline{\frakn}^\ast$ yields the claim.
\end{proof}

By Lem \ref{Kostant tri}, the algebraic representation $\bigwedge^{l(w)}\frakn \otimes_{\QQ_p} K$ admits a filtration $0 = F_0 \subsetneq N_1 \subsetneq \cdots F_m = \bigwedge^{l(w)}\frakn \otimes_{\QQ_p} K$ as $B_L$ representations such that dim $F_i=i$ and $m=\mathrm{dim}_{\QQ_p} (\bigwedge^{l(w)}\frakn)$.
The $1$-dimensional quotient of the filtration $F_m/F_{m-1}$ in the $N_L$-coinvariants is of algebraic weight $-(w\delta-\delta)$. If we have a short exact sequence of finite dimensional $B_L$ representations over $K$
$$0 \rightarrow U \rightarrow V \rightarrow W \rightarrow 0 ,$$ tensoring it with the rank $1$ $\sO(\Omega)$-module $\chi^w_{\Omega}$, where $B_L$ acts on it through $T^0$ via the character $\chi^w_{\Omega}$, we arrive at $$0 \rightarrow U \otimes_K \chi^w_{\Omega} \rightarrow V \otimes_K \chi^w_{\Omega} \rightarrow W \otimes_K \chi^w_{\Omega} \rightarrow 0.$$ 
By the left exactness of the induction we defined,  
\[
\begin{tikzcd}[matrix scale=0.4, row sep=large, ar symbol/.style = {draw=none,"\textstyle#1" description,sloped},
  isomorphic/.style = {ar symbol={\simeq}}, transform shape, nodes={scale=0.6}]
0 \arrow[r] & \mathrm{Ind}^{s,an}_L (U \otimes \chi^w_{\Omega}) \arrow[rr] \arrow[dr,isomorphic]  &&
\mathrm{Ind}^{s,an}_L (V \otimes \chi^w_{\Omega}) \arrow[dr,isomorphic] \arrow[rr] && \arrow[dr,isomorphic] \mathrm{Ind}^{s,an}_L (W \otimes \chi^w_{\Omega}) & \\
& & C^{s,an}(\overline{N}_L, U \otimes \chi^w_{\Omega}) \arrow[rr] &&
C^{s,an}(\overline{N}_L, V \otimes \chi^w_{\Omega})  \arrow[rr] &&  C^{s,an}(\overline{N}_L, W \otimes \chi^w_{\Omega}) \\
0 \arrow[r] & C^{s,1}(\overline{N}_L, U \otimes \chi^w_{\Omega}) \arrow[rr] \arrow[uu]  && C^{s,1}(\overline{N}_L, V \otimes \chi^w_{\Omega}) \arrow[rr] \arrow[uu]  && C^{s,1}(\overline{N}_L, W \otimes \chi^w_{\Omega})  \arrow[uu] & 
\end{tikzcd}
\]
and the commutativity of the whole diagram follows from the definition of $C^{s,1}$ and Prop \ref{aux mod}.
 
Note that any $K$-linear section $W \rightarrow V$ gives rise to $\sO[{\Omega}]$-linear section $W \otimes \chi^w_{\Omega} \rightarrow V \otimes \chi^w_{\Omega}$, which tells us the surjectivity of $C^{s,an}(\overline{N}_L, V \otimes \chi^w_{\Omega}) \twoheadrightarrow C^{s,an}(\overline{N}_L, W \otimes \chi^w_{\Omega})$. 
This shows the sequence is right exact as well. 
\[ \begin{tikzcd}
0 \rightarrow \mathrm{Ind}^{s,an}_L (U \otimes \chi^w_{\Omega}) \ar[r]                 & \mathrm{Ind}^{s,an}_L (V \otimes \chi^w_{\Omega}) \ar[r]                   & \mathrm{Ind}^{s,an}_L (W \otimes \chi^w_{\Omega})  \rightarrow 0           \\
  0 \rightarrow C^{s,1}(\overline{N}_L, U \otimes \chi^w_{\Omega}) \ar[u,hook] \ar[r] & C^{s,1}(\overline{N}_L, V \otimes \chi^w_{\Omega}) \ar[u,hook] \ar[r] & C^{s,1}(\overline{N}_L, W \otimes \chi^w_{\Omega}) \ar[u,hook]  \rightarrow 0   \\
\end{tikzcd} \]
Let $U$ be $F_{m-1}$, let $V$ be $F_m$, and let $W$ be $F_m/F_{m-1}$, we have 
\[ \xymatrix{
\tilde{i}^\ast : \bigwedge^{l(w)} \frakn \otimes \mathrm{Ind}^s_{w, \Omega} \ar@{->>}[r] & \mathrm{Ind}^{s,an}_L ( (w\delta-\delta) \otimes \chi^w_{\Omega}) = \mathrm{Ind}^s_{L, w \cdot \Omega}\\
\tilde{i}^\ast : \bigwedge^{l(w)} \frakn \otimes \mathrm{Ind}^{s,1}_{w, \Omega} \ar@{->>}[r] \ar@{^{(}->}[u] & \mathrm{Ind}^{s,1}_L ( (w\delta-\delta) \otimes \chi^w_{\Omega}) = \mathrm{Ind}^{s,1}_{L, w \cdot \Omega}. \ar@{^{(}->}[u].\\
} \]
By composing these two $I_L$ (resp. $I^s_{0,L}$)-equivariant natural surjections, we get our desired surjection $\tilde{i}^\ast$, hence $\tilde{i}$. 
From the arguments above one can see that this arrow carries $\bigwedge^{l(w)} \frakn \otimes \mathrm{Ind}^{s,1}_{L, \chi_{\Omega}^w}$ compatibly to $\mathrm{Ind}^{s,1}_{L, w \cdot \Omega}$.

Next, we construct $\tilde{p}$. For weights in $\frakn$, we pick $e_{\alpha}$ for an eigenvector corresponding to $\alpha \in \Delta^+$, $e:=\wedge_{\alpha \in \Delta^{+,w}} ~e_{\alpha} \subset \bigwedge^{l(w)} \frakn$. 
By Lem \ref{Kostant tri}, $e$ is fixed by $\Nbar_L$.

As $w \in W^P$, recall we identify $$\mathrm{Ind}^s_{L, w \cdot \Omega}  \simeq  C^{s,an}(\overline{N}_L, \sO({\Omega})), ~ \mathrm{Ind}^s_{w, \Omega}  \simeq  C^{s,an}(w\Nbar_Gw^{-1}, \sO({\Omega}))$$ as $\sO({\Omega})$ modules. 
$T^0$ acts on $wI_Gw^{-1}$. 
For a function $f$ on $\Nbar_L$, we define $\tilde{f}$ to be pullback of $f$ via the projection $p_L: w\Nbar_Gw^{-1} \twoheadrightarrow \overline{N}_L$ defined after Lem \ref{decomp wN} between strict $p$-adic manifolds (Lem \ref{decomp wN} together with \cite[Prop 3.3.6]{AS} ensure that $\Ind^s_{L,w\Omega} \to \Ind^s_{w,\Omega}$). 

By properly scaling $e$, 
\[ \xymatrix{
\tilde{p}^\ast :  \mathrm{Ind}^s_{L, w \cdot \Omega} \ar@{^{(}->}[r] & \bigwedge^{l(w)} \frakn \otimes \mathrm{Ind}^s_{w, \Omega}\\
\tilde{p}^\ast :  \mathrm{Ind}^{s,1}_{L, w \cdot \Omega} \ar@{^{(}->}[r] \ar@{^{(}->}[u] & \bigwedge^{l(w)} \frakn \otimes \mathrm{Ind}^{s,1}_{w, \Omega} \ar@{^{(}->}[u]\\
} \]

$$f  \mapsto  e \otimes  \tilde{f}$$
 which gives sections of $\tilde{i}^\ast$ as a morphism of $\sO(\Omega)$ Banach modules:
 for $\forall \bar{n} \in \Nbar_L$, the projection of $\bar{n}^{-1}e \otimes \tilde{f}(\bar{n})$ equals to $f(\bar{n})$.

By dualizing $\tilde{p}^\ast$ in the sense of Banach spaces, we obtain the sections $\tilde{p}$ of $\tilde{i}$, 
\[ \xymatrix{
\tilde{p}: \bigwedge^{l(w)} \frakn^\ast \otimes \mathbb{D}_{w,\Omega}^{s} \ar@{->>}[r] \ar@{^{(}->}[d] & \mathbb{D}_{L, w \cdot \Omega}^{s} \ar@{^{(}->}[d]\\
\tilde{p}: \bigwedge^{l(w)} \frakn^\ast \otimes \mathbb{D}_{w,\Omega}^{s,1} \ar@{->>}[r] & \mathbb{D}_{L, w \cdot \Omega}^{s,1}.\\
} \]

\begin{rem}
$\tilde{p}$ is not $I_L$ (resp. $I^s_{0,L}$)-equivariant, although $\tilde{i}$ is. The choice of the projection $p_L$ is not really important for us. 
\end{rem}

We claim that $\tilde{i}$ gives an $I_L$ (resp. $I^s_{0,L}$)-equivariant inclusion from $\mathbb{D}_{L, w \cdot \Omega}^{s}$ to the cohomology (Prop \ref{inj to coh})
\[ \xymatrix{
i: \mathbb{D}_{L, w \cdot \Omega}^{s} \ar@{^{(}->}[r] \ar@{^{(}->}[d] & H^{l(w)}(\mathfrak{n}, \mathbb{D}_{w,\Omega}^{s})^N \ar[d]\\
i: \mathbb{D}_{L, w \cdot \Omega}^{s,1} \ar@{^{(}->}[r] & H^{l(w)}(\mathfrak{n}, \mathbb{D}_{w,\Omega}^{s,1})^N.\\
} \]
Similarly, $p$ is defined on cohomology and induced by $\tilde{p}$. Currently we only know that horizontal arrows
 \[ \xymatrix{
p: H^{l(w)}(\mathfrak{n}, \mathbb{D}_{w,\Omega}^{s}) \ar@{->>}[r] \ar[d] & \mathbb{D}_{L, w \cdot \Omega}^{s} \ar@{^{(}->}[d]\\
p: H^{l(w)}(\mathfrak{n}, \mathbb{D}_{w,\Omega}^{s,1}) \ar@{->>}[r] & \mathbb{D}_{L, w \cdot \Omega}^{s,1}.\\
} \]
 are sections as $\sO(\Omega)$-modules. 
Let $d_k$ be the chain maps of the Chevalley–Eilenberg complexes
\[ \xymatrix{
\bigwedge^k \frakn^\ast \otimes \mathbb{D}_{w,\Omega}^{s} \ar[r]^{d_k \quad} \ar@{^{(}->}[d] & \bigwedge^{(k+1)} \frakn^\ast \otimes \mathbb{D}_{w,\Omega}^{s} \ar@{^{(}->}[d]\\
\bigwedge^k \frakn^\ast \otimes \mathbb{D}_{w,\Omega}^{s,1} \ar[r]^{d_k^1 \quad} & \bigwedge^{(k+1)} \frakn^\ast \otimes \mathbb{D}_{w,\Omega}^{s,1}.\\
} \]
We will study the weights appearing in the bottom complex.

\begin{lem}
\label{nice2}
The chain maps $d_k^1$ of the Chevalley–Eilenberg complex $$\bigwedge^k \frakn^\ast \otimes \mathbb{D}_{w,\Omega}^{s,1} \xrightarrow{d_k^1} \bigwedge^{(k+1)} \frakn^\ast \otimes \mathbb{D}_{w,\Omega}^{s,1}$$ are $I^s_{0,L}$-equivariant and nice.
For any $\delta \in D(I^s_{0,L}, K)$, the $\delta$ action on $\bigwedge^k \frakn^\ast \otimes \mathbb{D}_{w,\Omega}^{s,1}$ is nice.
\end{lem}
\begin{proof}
The differentials are $I^s_{0,L}$-equivariant by the same argument before \cite[Lem 2.1]{CO}.
Pick the weight vectors basis $e_1,\cdots,e_l \in \frak{n}$ and $e_1^\ast,\cdots,e_l^\ast \in \frak{n}^\ast$ such that $e_i^\ast(e_j)=\delta_{ij}$. $d^1_k=\sum_{1 \leq \alpha \leq l} (-1)^{\alpha} \cdot d_{k,\alpha}^1,$ where \begin{eqnarray*}
d_{k,\alpha}^1: \bigwedge^k \frakn^\ast \otimes \mathbb{D}_{w,\Omega}^{s,1} & \to & \bigwedge^{k+1} \frakn^\ast \otimes \mathbb{D}_{w,\Omega}^{s,1} \\ e^\ast \otimes v & \mapsto & e_{\alpha}^\ast\wedge e^\ast\otimes e_{\alpha}\cdot v.
\end{eqnarray*} 
We see $d_{k,\alpha}^1$ is $T^0$-equivariant, hence $d_k^1$. 
Note that $e_{\alpha}$ only shifts weights in $\mathbb{D}_{w,\Omega}^{s,1}$ by $e_{\alpha}$ with a less or equal to $1$ multiplier. 
More explicitly, if $v \in \mathbb{D}_{w,\Omega}^{s,1}$ is a $T^s$-eigenvector with weight $\chi_v$ and $\cO(\Omega) \cdot v$ generates a direct summand of $\chi_v$ part of $\mathbb{D}_{w,\Omega}^{s,1}$, $e_\alpha$ corresponds to weight $\alpha$, then $e_\alpha v=c(\alpha,v') v'$ for some $|c(\alpha,v')| \leq 1$ and a $T^s$-eigenvector $v'$, where $v'$ has weight $\chi_v \cdot \alpha$ and $\cO(\Omega) \cdot v'$ generates a direct summand of $\chi_v \cdot \alpha$ part of $\mathbb{D}_{w,\Omega}^{s,1}$. 
By Lem \ref{tensor nice}, $d_{k,\alpha}^1$ are nice. 
And a sum of nice maps is nice, so is $d^1_k$. 

For the second part, we know $D(wI^s_{0,G}w^{-1}, K)$ acts continuously on $\mathbb{D}_{w,\Omega}^{s,1}$ by Prop \ref{aux mod}. 
And a dense subalgebra $K [wI^s_{0,G}w^{-1}] \subset D(wI^s_{0,G}w^{-1}, K)$ (\cite[Lem 3.1]{ST02}) acts via nice endomorphisms on $\mathbb{D}_{w,\Omega}^{s,1}$. So $D(wI^s_{0,G}w^{-1}, K)$ acts via nice endomorphisms on $\mathbb{D}_{w,\Omega}^{s,1}$ by Lem \ref{limit nice}.
$e_{\alpha} \in D(wI^s_{0,G}w^{-1}, K)$ induces a nice endomorphism on $\mathbb{D}_{w,\Omega}^{s,1}$. 
\end{proof}

We define \begin{eqnarray}\label{add w} \Delta^{+,w}=\{\alpha \in \Delta^+| w^{-1} \alpha \in \Delta^-\}, \text{and} ~ e^\ast := \wedge_{\alpha \in \Delta^{+,w}} ~e^\ast_{\alpha} \end{eqnarray}
as a weight vector of $\wedge^{l(w)} \frak{n}^\ast$.

\begin{prop}
\label{inj to coh}
$\tilde{i}$ induces \[ \xymatrix{
i: \mathbb{D}_{L, w \cdot \Omega}^{s} \ar@{^{(}->}[r] \ar@{^{(}->}[d] & H^{l(w)}(\mathfrak{n}, \mathbb{D}_{w,\Omega}^{s})^N \ar[d]\\
i: \mathbb{D}_{L, w \cdot \Omega}^{s,1} \ar@{^{(}->}[r] & H^{l(w)}(\mathfrak{n}, \mathbb{D}_{w,\Omega}^{s,1})^N,\\
} \] $\tilde{p}$ induces  \[ \xymatrix{
p: H^{l(w)}(\mathfrak{n}, \mathbb{D}_{w,\Omega}^{s}) \ar@{->>}[r] \ar[d] & \mathbb{D}_{L, w \cdot \Omega}^{s} \ar@{^{(}->}[d]\\
p: H^{l(w)}(\mathfrak{n}, \mathbb{D}_{w,\Omega}^{s,1}) \ar@{->>}[r] & \mathbb{D}_{L, w \cdot \Omega}^{s,1}.\\
} \] 
\end{prop}
\begin{proof}
For the first part, we need to show \begin{enumerate}
  \item \label{1} $d_{l(w)} \circ \tilde{i}=0$.
  \item \label{2} Im($d_{l(w)-1}$) $\cap$ Im($\tilde{i}) = 0$.
  \item \label{3} Image of $i$ is invariant under $N$ action on the cohomology.
\end{enumerate}
\ref{1}: 
We use $\partial_{l(w)}$ to denote the corresponding $l(w)$-th differential map of the Chevalley–Eilenberg homology complex \begin{eqnarray*} \partial_{l(w)}: \bigwedge^{l(w)+1} \frakn \otimes \mathrm{Ind}^{s(,1)}_{w, \Omega} & \to & \bigwedge^{l(w)} \frakn \otimes \mathrm{Ind}^{s(,1)}_{w, \Omega} \\ e_{\alpha_1} \wedge \cdots \wedge e_{\alpha_{l(w)+1}} \otimes f & \mapsto & \sum\limits_{i=1}^{l(w)+1} e_{\alpha_1} \wedge \cdots \hat{e}_{\alpha_i} \cdots \wedge e_{\alpha_{l(w)+1}} \otimes e_{\alpha_i}\cdot f.
\end{eqnarray*} 
The image of $\tilde{i}^\ast \circ \partial_{l(w)}$ is a $I_L$-subrepresentation of the space of $\sO(\Omega)$-valued functions.
It suffices to show that $\tilde{i}^\ast \circ \partial_{l(w)}(e_{\alpha_1} \wedge \cdots \wedge e_{\alpha_{l(w)+1}} \otimes f)$ vanish at identity of $I_L$ for all roots $e_{\alpha_1},\cdots,e_{\alpha_{l(w)+1}}$ and $f \in \mathrm{Ind}^{s(,1)}_{w, \Omega}$. 
By construction of $\tilde{i}^\ast$, this expression vanishes if $\{e_{\alpha_1},\cdots,e_{\alpha_{l(w)+1}}\}$ (up to scalar multiples) does not contain $\Delta^{+,w}$. 
Otherwise let $e_\alpha$ be $\{e_{\alpha_1},\cdots,e_{\alpha_{l(w)+1}}\} - \Delta^{+,w}$.  
The value of $f|_{I_L}$ at identity is $N \cap wN_b w^{-1}$-invariant since $I_L$ stabilizes $N$ and $n\cdot f(1)=f(n^{-1})=f(1)$ for any $n \in N \cap wN_b w^{-1}$. As $e_\alpha \in \mathrm{Lie}(N \cap wN_b w^{-1})$, the value of $f |_{I_L}$ at identity vanishes.

There is another interesting viewpoint to see the highest weight vectors map to zero. 
Under the embedding $\tilde{i}$, the highest weight vectors \{$f^\vee_{h,\vec{0}}, h \in \overline{N}_L/\overline{N}_L^s$\} in $\mathbb{D}_{L, w \cdot \Omega}^{s,1}$ are identified with the combination of weight vectors $e^\ast \otimes f^\vee_{g, \vec{0}}$, where $e^\ast := \wedge_{\alpha \in \Delta^{+,w}} ~e^\ast_{\alpha}$ (\ref{add w}) and for any $g \in w\Nbar_Gw^{-1}/w\Nbar_G^sw^{-1}$. 
These vectors correspond to the weight $w\cdot \chi_{\Omega}$, for which we write as $-\sum\limits_{\alpha \in \Delta^+, w\cdot \alpha \in \Delta^-} \alpha + w\Omega$ where $w\Omega$ is the additive form of the multiplicative weight $\chi_{\Omega}^w$. 
It suffices to show $w\cdot \chi_{\Omega}$ is not a weight in $\bigwedge^{l(w)+1} \frakn^\ast \otimes \mathbb{D}_{w,\Omega}^{s,1}$. 
Weights appearing in $\bigwedge^i \frakn^\ast$ are opposite of sums of $i$ distinct roots in $\frakn$. 
By proposition Prop \ref{weights in Ds1} and Lem \ref{eigenweight}, weights in $\bigwedge^{i} \frakn^\ast \otimes \mathbb{D}_{w,\Omega}^{s,1}$ are of the form $$-\sum_{\substack{\alpha \in \Delta_i, |\Delta_i|=i \\ \Delta_i \subset \Delta_N \subset \Delta^+}} \alpha +\sum_{\alpha \in w\Delta^-} n_{\alpha} \alpha+w\Omega,$$ where $\Delta_N \subset \Delta^+$ is the set of roots in $\frakn$, $\Delta_i$ is a subset of $\Delta_N$ with $i$ elements, $n_{\alpha}$ are natural numbers. 
It remains to show: \begin{eqnarray*}
& & -\sum_{\alpha \in \Delta_i} \alpha +\sum_{\alpha \in w\Delta^-} n_{\alpha} \alpha+w\Omega= -\sum_{\alpha \in \Delta^{+,w}} \alpha + w\Omega \\ & & \Longrightarrow \forall n_\alpha=0, i=l(w), \Delta_i=\Delta^{+,w}.
\end{eqnarray*} The equation is equivalent to $$\sum_{\alpha \in \Delta^{+, w} - \Delta_i} \alpha + \sum_{\alpha \in w\Delta^-} n_{\alpha} \alpha=\sum_{\alpha \in \Delta_i-\Delta^{+,w}} \alpha.$$ 
Notice that $\Delta^{+,w}=\Delta^+ \cap w\Delta^- \subset w\Delta^-$, thus $\Delta_i-\Delta^{+,w} \subset \Delta^+-w\Delta^- \subset w\Delta^+$. 
The left hand side is a nonnegative sum of negative roots in $w\Delta^-$ and at the same time the right hand side is a nonnegative sum of positive roots in $w\Delta^+$. All coefficients in the above equation vanish. $\forall n_\alpha=0, i=l(w), \Delta_i=\Delta^{+,w}$. 

\ref{2}: We have already seen that $\tilde{p}$ gives a section of $\tilde{i}$. To show this part, it suffices to show $\tilde{p} \circ d_{l(w)-1}=0$. It suffices to prove that for any $\delta \in \mathbb{D}^s_{w,\Omega}$ (resp. $\mathbb{D}^{s,1}_{w,\Omega}$) and any $ e^\ast_{\alpha_1} \wedge \cdots \wedge e^\ast_{\alpha_{l(w)-1}}$, wedge of $l(w)-1$ roots in $\frakn^\ast$, $\tilde{p} \circ d_{l(w)-1}(e^\ast_{\alpha_1} \wedge \cdots \wedge e^\ast_{\alpha_{l(w)-1}} \otimes \delta)=0$. 
Pick the weight vectors basis $e_1,\cdots,e_l \in \frak{n}$ and $e_1^\ast,\cdots,e_l^\ast \in \frak{n}^\ast$ such that $e_i^\ast(e_j)=\delta_{ij}$. For any $f \in \mathrm{Ind}^s_{L,w\cdot\Omega}$, 
\begin{eqnarray*}
\tilde{p} \circ d_{l(w)-1}(e^\ast_{\alpha_1} \wedge \cdots \wedge e^\ast_{\alpha_{l(w)-1}} \otimes \delta)(f) &=& d_{l(w)-1}(e^\ast_{\alpha_1} \wedge \cdots \wedge e^\ast_{\alpha_{l(w)-1}} \otimes \delta)(e \otimes \tilde{f}) \\ &=& \sum_{1 \leq \alpha \leq l} (-1)^\alpha (e^\ast_{\alpha} \wedge e^\ast_{\alpha_1} \wedge \cdots \wedge e^\ast_{\alpha_{l(w)-1}})(e) \otimes (e_\alpha\cdot\delta) \tilde{f} \\
 &=& \sum_{1 \leq \alpha \leq l} c(\alpha) \cdot \delta(e_\alpha \cdot \tilde{f}),
\end{eqnarray*} where $c(\alpha)$ is a scalar valued in \{$0,\pm1$\} depending only on $\alpha$. $c(\alpha)$ is non-zero if and only if $e_\alpha$ and $e_{\alpha_1}, \cdots, e_{\alpha_{l(w)-1}}$ correspond exactly to all elements in $\Delta^{+,w}$ defined in \ref{1}. When this is the case, $e_\alpha \in \frakn \cap w\overline{\frakn}w^{-1} \subset \mathrm{Lie}(\overline{N}_+^s)$. $e_\alpha \cdot \tilde{f}=0$ since  $\tilde{f}$ is constant on $\overline{N}_+^s$ fibres.

\ref{3}: It suffices to prove for any $f \in \mathrm{Ind}^s_{L, w \cdot \Omega}$ and $n \in N$, $\tilde{i}( n \cdot \tilde{p}(f) - \tilde{p}(f) ) = 0$.
For any $\bar{n}_L \in N_L$, $\bar{n}_L^{-1} n \bar{n}_L \in N$, we write $\bar{n}_L^{-1} n \bar{n}_L = n_+ \cdot n_b$ such that $n_+ \in \Nbar_+^1$, $n_b \in N \cap wN_G w^{-1}$.
Note that
\[ e \otimes n\cdot \tilde{f}(\bar{n}_L) = e \otimes \tilde{f}(n^{-1} \bar{n}_L) = e \otimes \tilde{f}( (\bar{n}_L n_+ \bar{n}_L^{-1}) \bar{n}_L n_b ) = e \otimes f(\bar{n}_L), \] by the construction of $\tilde{f}$.

\end{proof}

\begin{rem}
\label{19x}
$p$ is defined on $H^{l(w)}(\mathfrak{n}, \mathbb{D}_{w,\Omega}^{s}) \supset H^{l(w)}(\mathfrak{n}, \mathbb{D}_{w,\Omega}^{s})^N$.
\end{rem}

If $\Omega'$ is closed in $\Omega$, there is a specialization map
$\DD^s_{w,\Omega} \twoheadrightarrow \DD^s_{w,\Omega'}$ with respect to $\sO(\Omega) \twoheadrightarrow \sO(\Omega')$.
The analytic induction is compatible with base change $\Ind^s_{w,\Omega} \otimes_{\sO(\Omega)} \sO(\Omega') \simeq \Ind^s_{w,\Omega'}$, yielding
\[ \mathcal{L}_{\mathscr{O}(\Omega)}(\mathrm{Ind}_{w, \Omega}^{s},\mathscr{O}(\Omega)) \twoheadrightarrow \mathcal{L}_{\mathscr{O}(\Omega')}(\mathrm{Ind}_{w, \Omega'}^{s},\mathscr{O}(\Omega')). \]

It remains to prove $p$ is $I^s_{0,L}$-equivariant for completing the proof of Thm \ref{pppp}.
We define $V^\bot, V^{\bot,1} := \ker p$ correspondingly for $H^{l(w)}(\mathfrak{n}$, $\mathbb{D}_{w,\Omega}^{s}), H^{l(w)}(\mathfrak{n}, \mathbb{D}_{w,\Omega}^{s,1})$. For $\forall a \in V^\bot, l \in I^s_{0,L}$, it suffices to show $$x:= p(l\cdot a)=0.$$ If $x \neq 0$, there exists a $f_0 \in \mathrm{Ind}^s_{w, \Omega}$ such that $x(f_0) \neq 0$. Then both $\supp(f_0)$ and $\supp(x(f_0))$ are Zariski open dense in $\Omega$. 
By viewing $f_0 \in C^{s,an}(\overline{N}_L, \sO(\Omega))$, vanishing of $f_0$ means vanishing of all coefficients of $f_0$, we have
$$\supp(f_0) \cap \supp(x(f_0)) \subset \supp(x),$$ 
where $\supp(x)$ is a Zariski open dense subset of $\Omega$. 
To prove $x=0$, it suffices to find a dense subset of $\Omega$ on which $x$ vanishes. 
\bigskip

We say a character $\chi$ of $T^0$ is \emph{generic} if $\chi (w\cdot \chi)^{-1}$  are not locally ($\QQ_p$-)algebraic for all $w \in W_G$.

\begin{lem}\label{wts in CE}
If $\mathfrak{s}$ is a weight in $\bigwedge^\bullet \frakn^\ast \otimes \mathbb{D}_{w,\lambda}^{s,1}$, set $\imath(s):= d (\mathfrak{s}) -d (w\cdot\lambda)$, where $d$ passes a character to Lie algebra. Then $\imath(s)$ is nonpositive with respect to $w\Delta^+$ (Prop \ref{weights in Ds1}). 
\end{lem}
We endow the partial ordering on weights of $\bigwedge^\bullet \frakn^\ast \otimes \mathbb{D}_{w,\lambda}^{s,1}$ under the map $\imath$ with respect to $w\Delta^+$ as well.
\begin{proof}
This is essentially in (1) of proof of proposition \ref{inj to coh}, which is a combination of Prop \ref{weights in Ds1} and the fact that the difference of any weight in $\bigwedge^\bullet \frakn^\ast$ and $e^\ast=\wedge_{\alpha \in \Delta^{+,w}} ~e^\ast_{\alpha}$ is nonpositive with respect to $w\Delta^+$.
\end{proof}
Viewing $w\Delta^+$ as positive roots, $\imath$ maps any such $\mathfrak{s}$ to non-positive span of them. 
Moreover, $\vec{0}$ must correspond to weights of the form $e^\ast \otimes c$, where $e^\ast = \wedge_{\alpha \in \Delta^{+,w}} ~e^\ast_{\alpha} \subset \bigwedge^{l(w)} \frakn^\ast$ (\ref{add w}) for $e^\ast_{\alpha}$ being an eigenvector corresponding to $\alpha \in \Delta^+$ and $c$ is a distribution vanishing on all non locally constant locally monomial functions for $w\Nbar_Gw^{-1}$. 

\begin{thm}
\label{vv}
If $\lambda=\Omega$ is a generic weight, $$V^{\bot,1}|_{\lambda} \simeq \widehat{\overline{\mathrm{Im} (d_{l(w)-1}^1)}}/\mathrm{Im} (d_{l(w)-1}^1).$$ 
Here $V^{\bot,1}|_{\lambda}$(resp. $V^\bot|_{\lambda}$) is defined to be $\ker(p|_{\lambda} : H^{l(w)}(\mathfrak{n}, \mathbb{D}_{w,\lambda}^{s,1}) \twoheadrightarrow \mathbb{D}_{L, w \cdot \lambda}^{s,1})$(resp. $\ker(p|_{\lambda} : H^{l(w)}(\mathfrak{n}, \mathbb{D}_{w,\lambda}^{s}) \twoheadrightarrow \mathbb{D}_{L, w \cdot \lambda}^{s})$).
\end{thm}
We use Lem \ref{rep of t} and the infinitesimal character argument of \cite{CO} to prove Thm \ref{vv}.

\begin{prop}\label{ker dlw}
For a generic weight $\lambda$, $\ker (d_{l(w)}^1) = \widehat{\overline{\mathbb{D}_{L, w \cdot \lambda}^{s,1} \oplus \mathrm{Im} (d_{l(w)-1}^1)}}$.
\end{prop}
\begin{proof}
We know that $d_{l(w)}^1$ is nice from Lem \ref{nice2}. Applying Lem \ref{nice1} to $\ker (d_{l(w)}^1) \hookrightarrow \bigwedge^{l(w)} \frakn^\ast \otimes \mathbb{D}_{w,\lambda}^{s,1}$, 
there exists $\imath: K \subset $ non-positive combinations of $w\Delta^+$ in the weight lattice and $V_k$ for each $k \in K$ corresponding to $(\displaystyle\prod_{k \in K} V_k)^b:=\ker (d_{l(w)}^1)$. 
Let $\tilde{V}:=\overline{\mathbb{D}_{L, w \cdot \lambda}^{s,1} \oplus \mathrm{Im} (d_{l(w)-1}^1)} \cap (\displaystyle\prod_{k \in K} V_k)^c$, apply Lem \ref{rep of t} again to $\tilde{V} \subset (\displaystyle\prod_{k \in K} V_k)^c$ to get $\overline{V}=(\displaystyle\prod_{k \in K} V_k)^c/\tilde{V}$. 
The Lie algebra action on $(\displaystyle\prod_{k \in K} V_k)^b$ is locally finite, therefore preserving $(\displaystyle\prod_{k \in K} V_k)^c$.

Let $\imath$ be as in Lem \ref{wts in CE}.
We want to show $\overline{V}=0$ by contradiction. If $\overline{V} \neq 0$, choose $s \in K$ such that $\imath(s)$ is maximal in the partial order set \{$\imath(s'), s'\in K \big|\overline{V}_{s'} \neq 0$\}. 
Note that $\overline{V}$ is a natural subquotient of $H^{l(w)}(\mathfrak{n}, \mathbb{D}_{w,\lambda}^{s,1})$ as an $I^s_{0,L}$ representation, by \cite[Cor 2.7]{CO}, $\overline{V}$ has the infinitesimal character of $Z(U(\mathrm{Lie}(I_L)))$ extending the infinitesimal character of $Z(U(\frak{g}))$. 
Pick a non-zero element $v$ in $\overline{V}_s \hookrightarrow \overline{V}$. 
Since $\imath(s)$ is maximal in \{$\imath(s'), s'\in K |\overline{V}_{s'} \neq 0$\}, Lie($N_L$) kills it. 
By the infinitesimal character restriction, $$\chi_s=w'\cdot(w\cdot\lambda)$$ for some $w' \in W_L$ considered as characters of Lie algebra of $T^0$. 
On the other hand, $\overline{V}_{w\cdot\lambda}$ must be zero since all the weights which are of the form $e^\ast \otimes c$ all live in $(\displaystyle\prod_{k \in K} V_k)^c$ part of $\mathbb{D}_{L, w \cdot \lambda}^{s,1}$. 
$(w\cdot\lambda) \times \chi_s^{-1}$ is a strictly dominant algebraic weight. 
$\lambda$ can not be generic. 
Therefore we have $\overline{V}=0, \tilde{V}=(\displaystyle\prod_{k \in K} V_k)^c$, and \begin{eqnarray}\label{Vc inclusion sum} (\displaystyle\prod_{k \in K} V_k)^c \subset \overline{\mathbb{D}_{L, w \cdot \lambda}^{s,1} \oplus \mathrm{Im} (d_{l(w)-1}^1)}, \end{eqnarray} which concludes $$\ker (d_{l(w)}^1) = (\displaystyle\prod_{k \in K} V_k)^b= \widehat{\overline{\mathbb{D}_{L, w \cdot \lambda}^{s,1} \oplus \mathrm{Im} (d_{l(w)-1}^1)}}.$$ 
\end{proof}
\begin{proof}[Proof of Theorem \ref{vv}]
Let $(\displaystyle\prod_{k \in K} V_k)^b:=\ker (d_{l(w)}^1)$ as in Prop \ref{ker dlw}. 
For $\forall ~ x \in (\displaystyle\prod_{k \in K} V_k)^c$ such that $\overline{x} \in V^{\bot,1}|_{\lambda}$, one can find a sequence $(a_n,b_n) \in \mathbb{D}_{L, w \cdot \lambda}^{s,1} \oplus \mathrm{Im} (d_{l(w)-1}^1)$, converging to $x$ by (\ref{Vc inclusion sum}).
We have the following commutative diagram with continuous arrows
\[ \begin{tikzcd}[row sep=tiny]
                                    & \ker(d_{l(w)}^1) \arrow[dd] \arrow[dr, "\tilde{p}"] & \\
\mathbb{D}_{L, w \cdot \lambda}^{s,1} \arrow[ur, "\tilde{i}"] \arrow[dr, "i"] &           &  \mathbb{D}_{L, w \cdot \lambda}^{s,1}  \\
                                    & H^{l(w)}(\mathfrak{n}, \mathbb{D}_{w,\Omega}^{s,1}) \arrow[ur, "p"] &
\end{tikzcd} \]
$$\tilde{p}(x)=0 \Longrightarrow \lim_{n \to \infty} a_n=\tilde{p}((a_n,b_n)) \to 0.$$ 
$\tilde{p}$ is the dual map of $\tilde{p}^\ast$, hence nice by Lem \ref{dual nice}.
It sends convergent series to convergent series.
Since $\tilde{p} \circ \tilde{i}=\mathrm{id}$ and both $\tilde{p}, \tilde{i}$ are continuous, $\tilde{i}$ induces equivalent norms on $\mathbb{D}_{L, w \cdot \lambda}^{s,1}$, $$\Longrightarrow (a_n,0) \to 0 \Longrightarrow (0,b_n) \to x \Longrightarrow x \in \overline{\mathrm{Im} (d_{l(w)-1}^1)}.$$ 
This proves that $$(\displaystyle\prod_{k \in K} V_k)^c=(\tilde{i}(\mathbb{D}_{L, w \cdot \lambda}^{s,1}) \cap (\displaystyle\prod_{k \in K} V_k)^c) \oplus (\overline{\mathrm{Im} (d_{l(w)-1}^1)} \cap (\displaystyle\prod_{k \in K} V_k)^c).$$
Now by applying Lem \ref{direct sum} for $\mathbb{D}_{L, w \cdot \lambda}^{s,1}$ as $V_1$, $\overline{\mathrm{Im} (d_{l(w)-1}^1)}$ as $V_2$ and $\ker (d_{l(w)}^1)$ as $(\displaystyle\prod_{s \in S} V_s)^b$ in the lemma, we have $$\ker (d_{l(w)}^1) = \mathbb{D}_{L, w \cdot \lambda}^{s,1} \oplus \widehat{\overline{\mathrm{Im} (d_{l(w)-1}^1)}}, ~ V^{\bot,1}|_{\lambda}=\widehat{\overline{\mathrm{Im} (d_{l(w)-1}^1)}}/\mathrm{Im} (d_{l(w)-1}^1).$$
\end{proof}

\begin{lem}
\label{ccs}
$\widehat{\overline{\mathrm{Im} (d_{l(w)-1}^1)}}$ is $I^s_{0,L}$ stable.
\end{lem}
\begin{proof}
The Chevalley–Eilenberg complex for $\mathbb{D}_{w,\lambda}^{s,1}$ is $I^s_{0,L}$-equivariant as a complex of continuous $K[I^s_{0,L}] \subset D(I^s_{0,L}, K)$ modules in the sense of Schneider-Teitelbaum by Prop \ref{aux mod}. $\overline{\mathrm{Im} (d_{l(w)-1}^1)}$ is $I^s_{0,L}$ stable. 
By Prop \ref{aux mod} and Lem \ref{tensor nice}, any group element action of $I^s_{0,L}$ is nice for $\bigwedge^{l(w)}\frak{n}^\ast \otimes \mathbb{D}_{w,\lambda}^{s,1}$. 
Then apply last part of Lem \ref{nice1} to $d_{l(w)-1}^1$ and Lem \ref{extend nice} to $\overline{\mathrm{Im} (d_{l(w)-1}^1)} \hookrightarrow \bigwedge^{l(w)}\frak{n}^\ast \otimes \mathbb{D}_{w,\lambda}^{s,1}$. 
\end{proof}

\begin{rem}
It should be much easier to see $\widehat{\overline{\mathrm{Im} (d_{l(w)-1}^1)}}$ is stable under Lie($I^s_{0,L}$) action. 
But by \cite[Lem 1.2.5, Prop 1.2.8]{Koh07}, the universal enveloping algebra of Lie($I^s_{0,L}$) is only dense in a closed subalgebra of $D(I^s_{0,L}, K)$ while $K [I^s_{0,L}]$ is dense in $D(I^s_{0,L}, K)$ by \cite[Lem 3.1]{ST02}.
\end{rem}

\begin{proof}[Proof of Theorem \ref{pppp}]
By Lem \ref{ccs}, $\widehat{\overline{\mathrm{Im} (d_{l(w)-1}^1)}}$ is $I^s_{0,L}$ stable. So $V^{\bot,1}$ is $I^s_{0,L}$ invariant. We have the following morphism of short exact sequences
\[ \xymatrix{
0 \ar[r] & V^\bot \ar[r] \ar[d] & H^{l(w)}(\mathfrak{n}, \mathbb{D}_{w,\lambda}^{s}) \ar[r] \ar[d] & \mathbb{D}_{L, w \cdot \lambda}^{s} \ar[r] \ar@{^{(}->}[d] & 0\\
0 \ar[r] & V^{\bot,1} \ar[r] & H^{l(w)}(\mathfrak{n}, \mathbb{D}_{w,\lambda}^{s,1}) \ar[r] & \mathbb{D}_{L, w \cdot \lambda}^{s,1} \ar[r] & 0.\\
} \] Since the middle vertical arrow is $I^s_{0,L}$ equivariant and the rightmost arrow is injective, $V^\bot$ is also $I^s_{0,L}$ stable. Let $x$ be the element defined below Rem \ref{19x}.
 $x|_{\lambda}=0$ for all generic weights $\lambda \in \Omega$. Generic weights are dense. $x=0$, $p$ is $I^s_{0,L}$-equivariant.

Lastly, if $U \hookrightarrow V$ is an embedding of characteristic zero representations of a group which splits over a subgroup of finite index, then $U$ is a direct summand.
\end{proof}

\end{document}